\documentclass{amsart}

\usepackage{color}
\usepackage{amsthm,amssymb,verbatim}
\usepackage{graphicx}
\usepackage{enumerate}
\DeclareRobustCommand{\SkipTocEntry}[5]{}

\newcommand{\ONE}{I}
\newcommand{\TWO}{II}

\newcommand{\rR}{\mathbf R}

\newcommand{\nupt}{\nu_{\partial M}}

\newcommand{\Mm}{\mathcal{M}}

\newcommand{\length}{\operatorname{length}}

\newcommand{\flux}{\operatorname{Flux}}    

 \newcommand{\graph}{\operatorname{graph}}
 \newcommand{\Cc}{\mathcal C}

 \newcommand{\MM}{\mathcal{M}}
 
 \newcommand{\slope}{\operatorname{slope}}
 \newcommand{\FF}{\mathcal{F}}
 \newcommand{\sS}{\mathbf{S}}

 \newcommand{\RR}{\mathbf{R}}  
 \newcommand{\ZZ}{\mathbf{Z}}  
 \newcommand{\BB}{\mathbf{B}}  
 \newcommand{\CC}{\mathbf{C}}  
 \newcommand{\dist}{\operatorname{dist}}
 
 \newcommand{\area}{\operatorname{area}}
 \newcommand{\eps}{\epsilon}
 \newcommand{\Tan}{\operatorname{Tan}}

 \newcommand{\genus}{\operatorname{genus}}

\newcommand{\interior}{\operatorname{interior}}

\newcommand{\vv}{\mathbf v}

\newcommand{\ee}{\mathbf{e}}

\newcommand{\grad}{\nabla}

 \newcommand{\Div}{\operatorname{div}}

\newcommand{\pdf}[2]{\frac{\partial #1}{\partial #2}}

\newcommand{\talpha}{\tilde\alpha}

\def\QED{\hfill$\Box$}   

\input epsf
\def\begfig {
\begin{figure}[b]
\small }
\def\endfig {
\normalsize
\end{figure}
}

    \newtheorem{THEOREM} {Theorem}         

\swapnumbers

    \newtheorem{theorem}    {Theorem}       [section]
    \newtheorem{lemma}      [theorem]       {Lemma}
    \newtheorem{corollary}  [theorem]     {Corollary}
    \newtheorem{proposition}       [theorem]       {Proposition}
    
    \newtheorem{alt-claim}    [theorem]       {Claim}                     
    \newtheorem*{claim}{Claim}
    \newtheorem*{theorem*}{Theorem}
    \newtheorem*{corollary*}{Corollary}

    \theoremstyle{definition}
    \newtheorem{definition}  [theorem] {Definition}

    \theoremstyle{definition}
    \newtheorem{remark}   [theorem]       {Remark}
    \newtheorem*{remark*}{Remark}

\usepackage{hyperref} 
\usepackage[alphabetic, msc-links, backrefs]{amsrefs}

\begin{document}

\renewcommand{\thesubsection}{\thetheorem}

\title{Helicoidal minimal surfaces of prescribed genus}
\subjclass[2010]{Primary: 53A10; Secondary: 49Q05, 53C42}
\author{David Hoffman}
\address{Department of Mathematics\\ Stanford University\\ Stanford, CA 94305}
\email{dhoffman@stanford.edu}
\author{Martin Traizet}
\address{Laboratoire de Math\'{e}matiques et Physique Th\'{e}orique,Universit\'{e} Fran\c cois Rabelais, 37200 Tours, France}
\email{Martin.Traizet@lmpt.univ-tours.fr}
\author{Brian White}
\address{Department of Mathematics\\ Stanford University\\ Stanford, CA 94305}
\thanks{The research of the second author was partially supported by ANR-11-ISO1-0002.
The research of the third author was supported by NSF grants~DMS--1105330
and DMS~1404282.}
\email{bcwhite@stanford.edu}
\date{April 6, 2013. Revised February, 2016.}
\begin{abstract}
For every genus $g$, we prove that $\sS^2\times\RR$
contains complete, properly embedded, genus-$g$ minimal
surfaces whose two ends are asymptotic to helicoids of any prescribed pitch.   
We also show that as the radius
of the $\sS^2$ tends to infinity, these examples converge smoothly to complete, properly embedded minimal surfaces in
$\RR^3$ that are helicoidal at infinity.   
We prove that helicoidal surfaces in $\RR^3$ of every prescribed genus occur as such limits of examples 
in $\sS^2\times\RR$.
\end{abstract}                

\maketitle


In 2005, Meeks and Rosenberg proved that the only complete, non flat, properly embedded minimal
surface of genus zero and one end is the standard helicoid~\cite{MeeksRosenbergUniqueness}.
Subsequently, Bernstein-Breiner~\cite{bernstein-breiner-conformal} 
     and Meeks-Perez~\cite{meeks-perez-end} proved that 
any complete, non flat, properly embedded minimal surface in $\rR^3$
of finite genus $g$ with one end must be asymptotic to a helicoid at infinity.
We call such a surface a genus-$g$ helicoid.
Until 1993, the only known example was the helicoid itself.
In that year, 
 Hoffman, Karcher and Wei \cite{hoffman-karcher-wei} discovered a genus-one minimal surface
 asymptotic to a helicoid at infinity (see Figure \ref{figure1}, left), and numerical 
 computations gave compelling evidence that it was embedded.
 Subsequently,
Weber, Hoffman and Wolf proved existence of an embedded example, i.e., of 
a genus-one helicoid\cite{weber-hoffman-wolf}.   
In~\cite{hoffman-white-genus-one}, Hoffman and White gave
a different proof for the existence of a genus-one helicoid.

A genus-$2$ helicoid was computed numerically by the second author in 1993
while he was a postdoc in Amherst (see Figure \ref{figure1}, right).
Helicoids of genus up to six have been computed by Schmies \cite{schmies} using the
theoretical techniques developed by Bobenko \cite{bobenko}.
These surfaces were computed using the Weierstrass Representation and
the period problem was solved numerically.
However, there was no proof that the period problem could be solved for genus $2$ or higher.

In this paper we prove:

\begin{THEOREM}\label{MAIN}
For every $g$, there exist genus-$g$ helicoids in $\rR^3$.
\end{THEOREM}

To construct higher genus helicoids in $\rR^3$, we first construct helicoid-like minimal
surfaces of prescribed genus in the Riemannian 3-manifold $\sS^2\times\rR$,
where $\sS^2$ stands for a round sphere. This is achieved by a degree argument.
Then we let the radius of the sphere $\sS^2$ go to infinity and we prove that in the
limit we get helicoids of prescribed genus in $\rR^3$. The delicate part in this
limiting process is to ensure that the limit has the desired topology, in other words
that the handles do not all drift away.

The paper is divided into two parts. In \hyperref[part1]{part~\ONE}, we construct helicoidal minimal surfaces
in $\sS^2\times\rR$, and we prove that they converge to helicoidal minimal surfaces in
$\rR^3$ as the radius goes to infinity.
In \hyperref[part2]{part~\TWO}, we prove that the limit has the desired topology by proving that, if we work
with suitable helicoids of an even genus in $\sS^2\times\rR$ and let the radius go to infinity, then
exactly half of the handles drift away. 

Parts~\hyperref[part1]{\ONE} and~\hyperref[part2]{\TWO}
 are in some ways independent of each other,
and the methods used are very different.
Of course, \hyperref[part2]{part~\TWO} uses some properties of the $\sS^2\times\RR$ surfaces obtained in \hyperref[part1]{part~\ONE}, but
otherwise it does not depend on the way in which those surfaces were obtained.
We have stated those properties as they are needed in \hyperref[part2]{part~\TWO}, so that \hyperref[part2]{part~\TWO} can
be read independently of \hyperref[part1]{part~\ONE}.

\begin{figure}
\begin{center}
\includegraphics[height=60mm]{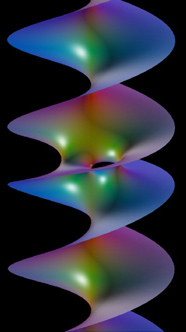}
\hspace{2cm}
\includegraphics[height=60mm]{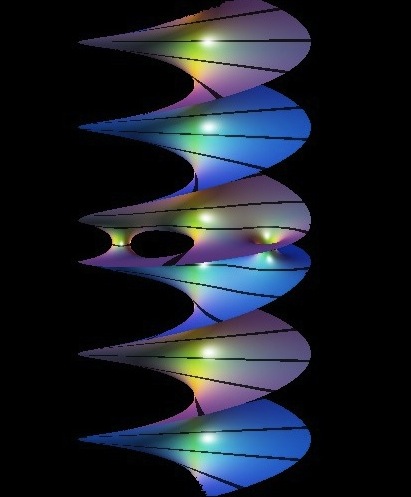}
\end{center}
\caption{
{\bf Left:} A genus-one helicoid, computed by David Hoffman, Hermann Karcher and Fusheng Wei.
{\bf Right:}  A  genus-two helicoid, computed by  Martin Traizet.
Both surfaces were computed numerically using the Weierstrass representation,
and the images were made with Jim Hoffman using
visualization software he helped  to develop.}
\label{figure1}
\end{figure}

\tableofcontents


\bigskip
\section*{\bf Part~\ONE: Genus-\texorpdfstring{$g$}{Lg} Helicoids in \texorpdfstring{$\sS^2\times\RR$}{Lg}}
\label{part1}

\bigskip

The study of complete, properly embedded minimal surfaces in 
  $\Sigma\times \RR$, where $\Sigma$ is a complete Riemannian $2$-manifold, was
initiated by Harold Rosenberg in \cite{rosenberg2002}.
The general theory of such surfaces was further developed by Meeks and Rosenberg
 in~\cite{MeeksRosenbergTheory}.
In the case of $\sS^2\times\RR$, 
 if such a surface has finite topology, then
 either it is a union of horizontal spheres $\sS^2\times\{t\}$, 
or else it is conformally a connected, twice-punctured, 
compact Riemann surface,
with one end going up and the other end going down~\cite{rosenberg2002}*{Theorems~3.3,~4.2,~5.2}.
In that same paper, Rosenberg described a class of such
surfaces in $\sS^2\times \RR$ 
that are very similar to helicoids in $\RR^3$, and hence
are also called helicoids.   
They may be characterized as the complete, non-flat minimal
surfaces in $\sS^2\times\RR$ whose 
horizontal slices are all great circles. 
(See section~\ref{part-1-preliminaries} for a more explicit
description of helicoids in $\sS^2\times\RR$ and for a 
discussion of their basic properties.)

In \hyperref[part1]{part~\ONE} of this paper, we prove the existence of properly embedded minimal
surfaces in $\sS ^2\times \RR$  of  prescribed finite genus, with top and
bottom ends asymptotic to  an end of a helicoid of any
prescribed pitch. (The pitch of a helicoid in $\sS ^2\times \RR$   
is defined in section~\ref{part-1-preliminaries}. The absolute value of the pitch
is twice the vertical distance between successive sheets of the helicoid. 
The sign of the pitch depends on the sense in which the helicoid winds about its axes.) 
Although  the pitch of the helicoid  to which the top end is asymptotic equals 
 pitch of the helicoid to which the bottom is asymptotic, we do not know if these two helicoids coincide;
one might conceivably be a vertical translate of the other.  Each of the surfaces
we produce contains a pair of antipodal vertical lines $Z$ and $Z^*$ (called axes of the surface)
and a horizontal great circle $X$ that intersects each of the axes.  Indeed, for each  of our surfaces, there is a helicoid whose intersection with the surface is precisely 
$X\cup Z \cup Z^*$.

 For every genus, our method produces two examples that are not
 congruent to each other by any orientation-preserving isometry 
 of $\sS^2\times \RR$. 
 The two examples are distinguished by their behavior at the origin $O$:
 one is ``positive" at $O$ and the other is ``negative" at $O$. (The
 positive/negative terminology is explained in section~\ref{sign-section}.) 
 If the genus is odd, the two examples
 are congruent to each other by reflection $\mu_E$ in the totally geodesic cylinder
 consisting of all points equidistant from the two axes.
 If the genus is even, the two examples are not congruent to each other
 by any isometry of $\sS^2\times \RR$, but each one is invariant under the
 reflection $\mu_E$.
 The examples of even genus $2g$ are also invariant under 
 $(p,z)\mapsto (\tilde p, z)$, where $p$ and $\tilde p$ are antipodal points 
  in $\sS^2$, so their quotients under this involution are
 genus-$g$ minimal surfaces in $\RR \mathbf{P}^2\times\RR$ with helicoidal ends. 

For each genus $g$ and for each helicoidal pitch, we prove that as the radius
of the~$\sS^2$ tends to infinity, our examples converge subsequentially
to complete, properly embedded minimal surfaces in $\RR^3$ that are asymptotic
to helicoids at infinity.  
The arguments in \hyperref[part2]{part~\TWO} required to control the genus of the limit (by
preventing too many handles from drifting away)
are rather delicate.  It is much easier to control whether the 
limiting surface has odd or even genus: a limit (as the radius of $\sS^2$
tends to infinity) of ``positive'' examples must have even genus and a limit
of ``negative'' examples must have odd genus.
Such parity control   is sufficient 
(without the delicate arguments of \hyperref[part2]{part~\TWO}) 
to give a new  proof of the existence of a
 genus-one helicoid in $\RR^3$. 
 See the corollary to theorem~\ref{DIVERGING-RADII-THEOREM} in section~\ref{section:main-theorems} for
 details.
 
Returning to our discussion of examples in $\sS^2\times \RR$, we also prove existence of what might be
called periodic genus-$g$ helicoids.  They are properly embedded minimal
surfaces  that are invariant under a screw motion of $\sS^2\times\RR$
and that have fundamental domains of genus $g$.  Indeed, our nonperiodic examples in  $\sS^2\times\RR$ are obtained as limits of the periodic examples as the period  tends to infinity.

As mentioned above, all of our examples contain two vertical axes $Z\cup Z^*$
and a horizontal great circle $X\subset \sS^2\times\{0\}$ at height $0$.   
Let $Y$ be the great circle at height $0$ such that $X$, $Y$, and $Z$ meet
orthogonally at a pair of points $O\in Z$ and $O^*\in Z^*$.
All of our examples are invariant
under $180^\circ$ rotation about $X$, $Y$, and $Z$
 (or, equivalently,  about $Z^*$: rotations about $Z$ are also rotations about $Z^*$).
In addition, the nonperiodic examples (and suitable fundamental domains
of the periodic examples) are what we call ``$Y$-surfaces".  Intuitively, this
means that they are $\rho_Y$-invariant (where $\rho_Y$ is $180^\circ$
rotation about $Y$) and that the handles (if there are any)  occur along $Y$. 
The precise definition is that $\rho_Y$ acts by multiplication by $-1$ on
the first homology group of the surface.
This property is very useful because it means that when we
let the period of the periodic examples tend to infinity, the handles cannot drift
away: they are trapped along $Y$, which is compact.  
In \hyperref[part2]{part~\TWO},    
when we need to control handles drifting off to infinity
as we let the radius of the $\sS^2$ tend to infinity, 
the $Y$-surface property means that 
the handles can only drift off in one direction (namely along $Y$).
 
\hyperref[part1]{Part~\ONE} is organized as follows.
In section~\ref{part-1-preliminaries}, we present the basic facts about helicoids
in $\sS^2\times\RR$.
In section~\ref{section:main-theorems}, we state
the main results.
In section~\ref{sign-section}, we describe what it means
for a surface to be positive or negative at $O$ with respect to $H$.
In section~\ref{section:$Y$-surfaces}, we describe the general properties
of $Y$-surfaces.
In sections~\ref{construction-outline-section}\,--\ref{smooth-count-section},
 we prove existence of periodic genus-$g$ helicoids in 
  $\sS^2\times \RR$.
In sections~\ref{general-results-section} and~\ref{area-bounds-section} 
we present general results we will use in order to establish the existence of limits. 
In section~\ref{nonperiodic-section}, 
we get nonperiodic genus-$g$ helicoids as limit of periodic
examples by letting the period tend to infinity.
In section~\ref{R3-section} and \ref{Proof_of_DIVERGING-RADII-THEOREM}, 
we prove that as the radius of the $\sS^2$ tends to infinity,
our nonperiodic genus-$g$ helicoids in $\sS^2\times\RR$ converge
to properly embedded, minimal surfaces in $\RR^3$ with helicoidal ends.

\stepcounter{theorem}
\section{Preliminaries}\label{part-1-preliminaries}

\addtocontents{toc}{\SkipTocEntry}
\subsection*{Symmetries of $\sS^2\times\RR$}
Let $R>0$ and $\sS^2=\sS^2(R)$ be the sphere of radius $R$.  
Let $C$ be a horizontal great circle in $\sS^2\times \RR$ at height $a$, i.e., a great circle
in the sphere $\sS^2\times \{a\}$ for some $a$.
The union of all vertical lines through points in $C$ is a totally geodesic cylinder.
We let $\mu_C$ denote reflection in that cylinder: it is the 
orientation-reversing isometry of $\sS^2\times\RR$
that leaves points in the cylinder fixed and that interchanges the two components of
the complement. 
If we compose $\mu_C$ with reflection in the sphere $\sS^2\times\{a\}$, i.e., with the isometry 
\[
  (p,z)\in \sS^2\times \RR \mapsto (p, 2a-z),
\]
we get an orientation-preserving involution $\rho_C$ of $\sS^2\times \RR$ whose fixed-point set is precisely $C$.
Intuitively, $\rho_C$ is $180^\circ$ rotation about $C$.

Let $L$ be a vertical line $\{p\}\times\RR$ in $\sS^2\times \RR$
and let $L^*$ be the antipodal line, i.e., the line $\{p^*\}\times\RR$ where $p$ and $p^*$
are antipodal points in $\sS^2$.
Rotation through any angle $\theta$
about $L$ is a well-defined isometry of $\sS^2\times \RR$.  If $\theta$ is not a multiple of $2\pi$,
then the fixed point set of the rotation is $L\cup L^*$. 
Thus any rotation about $L$ is also a rotation about $L^*$.
We let $\rho_L (= \rho_{L^*})$ denote the $180^\circ$ rotation about $L$.

\addtocontents{toc}{\SkipTocEntry}
\subsection*{Helicoids in $\sS^2\times \RR$}
Let $O$ and $O^*$ be a pair of antipodal points in $\sS^2\times\{0\}$,
and let $Z$ and $Z^*$ be the vertical lines in $\sS^2\times\RR$ through those points.
Let $X$ and $Y$ be a pair of great circles in $\sS^2\times\{0\}$ that intersect
orthogonally at $O$ and $O^*$.  Let $E$ be the equator in $\sS^2\times \{0\}$ 
with poles $O$ and $O^*$,
i.e., the set of points in $\sS^2\times\{0\}$ equidistant from $O$ and $O^*$.

Fix a nonzero number $\kappa$ and consider the surface
\[
  H = H_\kappa = \bigcup_{t\in\RR}  \sigma_{2\pi t, \kappa t}X,
\]
where $\sigma_{\theta,v}:\sS^2\times \RR \to \sS^2\times\RR$ is the screw motion
given by rotation by $\theta$ about $Z$ (or, equivalently, about $Z^*$) together with
vertical translation by $v$.   We say that $H$ is the {\em helicoid of pitch $\kappa$
and axes $Z\cup Z^*$ that contains $X$}.

To see that $H$ is a minimal surface, note that it is fibered by horizontal great circles.
Let $p$ be a point in $H$ and let $C$ be the horizontal great circle in $H$ containing $p$.
 One easily checks that the involution $\rho_C$ ($180^\circ$ rotation about $C$) 
 maps $H$ to $H$, reversing its orientation.
It follows immediately that the mean curvature of $H$ at $p$ is $0$.  For if it were nonzero,
it would point into one of the two components of $(\sS^2\times\RR)\setminus H$.
But then by the symmetry $\rho_C$ (which interchanges the two components), it would also point
into the other component, a contradiction.

Unlike helicoids in $\RR^3$, the helicoid $H$ has two axes, $Z$ and $Z^*$.
Indeed, the reflection $\mu_E$ restricts to an orientation-reversing isometry of $H$
that interchanges $Z$ and $Z^*$.

The number $\kappa$ is called the {\em pitch} of the helicoid: its absolute value
is twice the vertical distance between successive sheets of $H$.  Without loss of generality
we will always assume that $\kappa>0$.  As $\kappa$ tends to $\infty$, the helicoid $H_\kappa$ converges smoothly
to the cylinder $X\times\RR$, which thus could be regarded as a helicoid of infinite pitch.

\section{The Main Theorems}\label{section:main-theorems}

 We now state our first main result in a form that
includes the periodic case ($h<\infty$) and the nonperiodic case ($h=\infty$.)
The reader may wish initially to ignore the periodic case.
Here $X$ and $Y$ are horizontal great circles at height $z=0$ that intersect
each other orthogonally at points $O$ and $O^*$, 
 $E$ is the great circle of points at height $z=0$ equidistant from $O$ and $O^*$,
and $Z$ and $Z^*$ are the
vertical lines passing through $O$ and $O^*$.  

\begin{THEOREM}\label{main-S2xR-theorem}
Let $H$ be a helicoid in $\sS^2\times \RR$ that has vertical axes $Z\cup Z^*$
and that contains the horizontal great circle $X$.
For each genus $g\ge 1$ and each height $h\in (0,\infty]$, there
exists a pair $M_{+}$ and $M_{-}$ of embedded
minimal surfaces in $\sS^2\times \RR$ of genus $g$ with the following
properties (where $s\in \{+, -\}$):
\begin{enumerate}[\upshape (1)]
\item If $h=\infty$, then $M_s$ has no boundary, it is properly embedded in $\sS^2\times\RR$, 
  and each of its two ends
  is asymptotic to $H$ or to a vertical translate of $H$.
\item\label{1:boundary-circles} If $h<\infty$, then $M_s$ is a smooth, compact 
  surface-with-boundary in $\sS^2\times[-h,h]$.  Its boundary consists of the 
  two great circles at heights $h$  and $-h$ that intersect $H$ orthogonally at points in $Z$ and in $Z^*$.
\item If $h=\infty$, then
     \[  
      M_s\cap H = Z\cup Z^*\cup X.
     \]
    If $h<\infty$, then 
    \[
       {\rm interior}(M_s)\cap H = Z_h \cup Z_h^* \cup X,
    \]
    where $Z_h$, and $Z^*_h$
   are the portions of $Z$ and $Z^*$ with $|z|<h$.
\item $M_s$ is a $Y$-surface. 
\item $M_s \cap Y$ contains exactly $2g+2$ points.
\item\label{1:sign}
   $M_{+}$ and $M_{-}$ are positive and negative, 
   respectively, with respect to $H$ at $O$.
 \item
If $g$ is odd, then $M_{+}$ and $M_{-}$ 
are congruent to each other by reflection $\mu_E$ in the cylinder $E\times \RR$.
They are not congruent to each other by any orientation-preserving isometry of $\sS^2\times\RR$.
\item\label{even-genus-congruence-assertion}
 If $g$ is even, then
 $M_{+}$ and $M_{-}$ are each invariant under reflection $\mu_E$  in the cylinder 
   $E\times\RR$.  They are not congruent to each other by any isometry of 
   $\sS^2\times\RR$.
\end{enumerate}
\end{THEOREM}

The positive/negative terminology in assertion~\eqref{1:sign} is explained
in section~\ref{sign-section}, and $Y$-surfaces are defined and discussed
in section~\ref{section:$Y$-surfaces}.

Note that if $h<\infty$, we can extend $M_s$
 by repeated Schwarz reflections to
get a complete, properly embedded minimal surface $\widehat{M}_s$ that
is invariant under the screw motion $\sigma$ that takes $H$ to $H$ (preserving its orientation) and $\{z=0\}$
to $\{z=2h\}$.   The intersection $\widehat{M}_s\cap H$ consists of $Z$, $Z^*$,
and the horizontal circles $H\cap \{z=2nh\}$, $n\in \ZZ$. 
The surfaces $\widehat{M}_s$ are the periodic genus-$g$ helicoids mentioned
in the introduction.

\begin{remark}\label{other-circles-remark}
Assertion~\eqref{1:boundary-circles} states (for $h<\infty$) 
that the boundary $\partial M_s$ consists of two great circles that meet
$H$ orthogonally.   Actually, we could allow $\partial M_s$ to be any $\rho_Y$-invariant pair of great circles
at heights $h$ and $-h$ that intersect $Z$ and $Z^*$.
We have chosen to state theorem~\ref{main-S2xR-theorem} for circles that meet
the helicoid $H$ orthogonally because when we extend $M_s$ by
repeated Schwarz reflections to get a complete, properly embedded surface $\widehat M$, that choice
makes the intersection set $\widehat M\cap H$ particularly simple.
In section~\ref{adjusting-pitch-section}, we explain why the choice does not matter: 
if the theorem is true for one choice, it is also true for any other choice.  
Indeed, in {\em proving} the $h<\infty$ case of theorem~\ref{main-S2xR-theorem}, it will be more convenient
to let $\partial M_s$ be the horizontal great circles $H\cap \{z=\pm h\}$ that lie in $H$.
(Later, when we let $h\to \infty$ to get nonperiodic genus-$g$ helicoids in $\sS^2\times\RR$, the
choice of great circles $\partial M_s$ plays no role in the proofs.)  
\end{remark}

\begin{remark}\label{non-round-remark}
Theorem~\ref{main-S2xR-theorem} remains true if the round metric on $\sS^2$ is replaced
by any metric that has positive curvature, that is rotationally symmetric about the poles $O$ and $O^*$,
and that is symmetric with respect to reflection in the equator of points equidistant from $O$ and $O^*$.
(In fact the last symmetry is required only for the assertions about $\mu_E$ symmetry.)
No changes are required in any of the proofs. 
\end{remark}

In the nonperiodic case ($h=\infty$) of theorem~\ref{main-S2xR-theorem}, 
we do not know whether the two ends of $M_s$ are asymptotic to
opposite ends of the same helicoid. Indeed, it is possible that the top end
is asymptotic to a $H$ shifted vertically by some amount $v\ne 0$; the bottom end would then be asymptotic to $H$ shifted vertically by $-v$.
 Also, we do not know whether
 $M_{+}$ and $M_{-}$ must be asymptotic to each other, or
 to what extent the pair $\{M_{+}, M_{-}\}$
is unique.

Except for the noncongruence assertions, the proof of theorem~\ref{main-S2xR-theorem}
 holds  for all helicoids $H$ including
    $H=X\times\RR$, which may be regarded as a helicoid of infinite pitch.
(When $H=X\times \RR$ and $h=\infty$, theorem~\ref{main-S2xR-theorem} was  proved
 by Rosenberg in section~4 of~\cite{rosenberg2002}  by completely different methods.)  
 When $X=H\times\RR$, the noncongruence
 assertions break down: see section~\ref{noncongruence-appendix}. 
The periodic (i.e., $h<\infty$) case
of theorem~\ref{main-S2xR-theorem} is proved at the end of section~\ref{construction-outline-section},
assuming theorem~\ref{special-existence-theorem}, whose proof is a consequence of the
material in subsequent sections.  
The nonperiodic ($h=\infty$) case is proved in section~\ref{nonperiodic-section}.

Our second main result  lets us take limits as the radius of
the $\sS^2$ tends to infinity. For simplicity we only deal with the
nonperiodic case  ($h=\infty$) here.\footnote{An analogous theorem is
true for the periodic case ($h<\infty$).}

\begin{THEOREM}\label{DIVERGING-RADII-THEOREM}
Let $R_n$ be a sequence of radii tending to
infinity.   For each $n$, let $M_{+}(R_n)$ and $M_{-}(R_n)$ be 
genus-$g$ surfaces in $\sS^2(R_n)\times \RR$ satisfying the list
of properties in theorem~\ref{main-S2xR-theorem}, where $H$ is the helicoid of pitch $1$
and $h=\infty$.
 Then, after passing to
a subsequence, the  $M_+(R_n)$ and $M_{-}(R_n)$ converge smoothly on
compact sets to limits $M_+$ and $M_{-}$ with the
following properties:
\begin{enumerate}
\item\label{2:helicoidlike} $M_{+}$ and $M_{-}$ are complete, properly
embedded minimal surfaces in $\RR^3$ that are asymptotic to the
standard helicoid $H\subset\RR^3$. 
\item\label{2:intersection-property} If $M_s \ne H$, then $M_s\cap H=X\cup Z$ and
  $M_s$ has sign $s$ at $O$ with respect to $H$.
\item\label{2:Y-surface-property}  $M_s$ is a $Y$-surface.
\item\label{2:point-count} 
  $\|M_s\cap Y\| = 2\, \|M_s\cap Y^+\| +1 = 2\,\genus(M_s)+1$.
 \item\label{2:genus-bound} If $g$ is even, then
      $M_{+}$ and $M_{-}$ each have genus at most $g/2$.
      If $g$ is odd, then $\genus(M_{+})+\genus(M_{-})$
       is at most $g$.
 \item\label{2:genus-parity} The genus of $M_{+}$ is even. 
          The genus of $M_{-}$ is odd.
\end{enumerate}
\end{THEOREM}
Here if $A$ is a set, then $\|A\|$ denotes the number of elements of $A$.

Theorem~\ref{DIVERGING-RADII-THEOREM} is proved in section~\ref{Proof_of_DIVERGING-RADII-THEOREM}.

As mentioned earlier, theorem~\ref{DIVERGING-RADII-THEOREM} gives a new proof of the existence of genus-one helicoids
in $\RR^3$:

\begin{corollary*}
If $g=1$ or $2$, then $M_{+}$ has genus $0$ and $M_{-}$ has genus $1$.
\end{corollary*}
The corollary follows immediately from statements~\eqref{2:genus-bound}
and~\eqref{2:genus-parity} of theorem~\ref{DIVERGING-RADII-THEOREM}.

In \hyperref[part2]{part~\TWO},
 we prove existence of helicoidal surfaces
of arbitrary genus in $\RR^3$:

\begin{THEOREM}\label{PART1-EXACT-GENUS-THEOREM}
Let $M_{+}$ and $M_{-}$ be the limit minimal surfaces in $\RR^3$
described in theorem~\ref{DIVERGING-RADII-THEOREM}, and suppose that $g$ is even.
 If $g/2$ is even, then $M_{+}$ has genus $g/2$.
 If $g/2$ is odd, then $M_{-}$ has genus $g/2$.
\end{THEOREM}

The sign here is crucial: if $g/2$ is even, then $M_{-}$ has
genus strictly less than $g/2$, 
and if $g/2$ is odd, then $M_{+}$ has genus strictly less than $g/2$.
(These inequalities follow immediately from statements~\eqref{2:genus-bound} 
and~\eqref{2:genus-parity} of theorem~\ref{DIVERGING-RADII-THEOREM}.)

\section{Positivity/Negativity of Surfaces at \texorpdfstring{$O$}{Lg}}\label{sign-section}

In this section, we explain the positive/negative terminology
used in theorem~\ref{main-S2xR-theorem}. 
Let $H$ be a helicoid that has axes $Z\cup Z^*$ and that contains $X$.
 The set 
 \[
     H\setminus (X \cup Z\cup Z^*)
\]
consists of four
components that we will call quadrants. The axes $Z$ and $Z^*$ are
naturally oriented, and we choose an orientation of $X$ allowing
us to label the components of $X\setminus\{O,O^*\}$ as $X^+$ and
$X^-$.   
We will refer to the  quadrant bounded by $X^+$, $Z^+$ and $(Z^*)^+$ and the
quadrant bounded by $X^-$, $Z^-$, and $(Z^*)^-$ as
the {\em positive quadrants} of $H$.  
The other two quadrants are called the {\em negative quadrants}.
 We orient $Y$ so
that the triple $(X,Y,Z)$ is positively oriented at $O$, and let
 $H^+$ denote the the component of the complement of $H$ that contains
$Y^+$. 

Consider an embedded minimal surface $S$  in $\sS^2\times\RR$ such that
in some open set $U$ containing $O$,
\begin{equation}\label{S-sign-definition}
     (\partial S)\cap U = (X\cup Z) \cap U.
\end{equation}
If $S$ and the two positive quadrants of $H\setminus (X\cup Z)$ are tangent to each other at $O$,
 we say that $S$ is {\em positive} at $O$.
If $S$ and the two negative quadrants of $H\setminus(X\cup Z)$ are tangent 
to each other at $O$, we say that $S$ is {\em negative} at $O$.
(Otherwise the sign of $S$ at $O$ with respect to $H$ is not defined.)

Now consider an embedded minimal surface $M$ in $\sS^2\times\RR$ such that 
\begin{equation}\label{M-sign-definition}
\begin{aligned}
   &\text{The origin $O$ is an interior point of $M$, and} \\
   &\text{$M\cap H$ coincides with $X\cup Z$ in some neighborhood of $O$.}
\end{aligned}
\end{equation}
We say that $M$ is positive or negative at $O$ with respect to $H$ according
to whether $M\cap H^+$ is positive or negative at $O$.

Positivity and negativity at $O^*$ is defined in exactly the same way.

\begin{remark}
A surface $S$ satisfying~\eqref{S-sign-definition} is positive (or negative) at $O$ 
if and only if $\mu_ES$ is positive (or negative)
at $O^*$, where $\mu_E$ denotes reflection in the totally geodesic cylinder
consisting of all points equidistant from $Z$ and $Z^*$.  
Similarly, a surface $M$ satisfying~\eqref{M-sign-definition} is positive (or negative) at $O$ with respect to $H$ if
and only $\mu_EM$ is positive (or negative) at $O^*$ with respect to $H$.
(If this is not clear, note that $\mu_E(H^+)=H^+$ and that $\mu_E(Q)=Q$ for each quadrant $Q$ of $H$.)
\end{remark}



\section{\texorpdfstring{$Y$}{Lg}-surfaces}\label{section:$Y$-surfaces}

As discussed in the introduction, the surfaces we construct
will be $Y$-surfaces.  In this section, we define ``$Y$-surface"
and prove basic properties of $Y$-surfaces.

\begin{definition}\label{Y-Surface}
Suppose $N$ is a Riemannian $3$-manifold that admits an order-two rotation
 $\rho_Y$
about a geodesic $Y$.  An orientable surface $S$ in $N$ 
is called a {\em $Y$-surface}
if $\rho_Y$ restricts to an orientation-preserving isometry of $S$
and if 
\begin{equation}
 \label{-1homology}
  \mbox{ $\rho _Y$ acts  on $H_1(S,\ZZ)$ by
multiplication by $-1$.}
\end{equation}
\end{definition}

 The following proposition shows that the definition of a $Y$-surface is equivalent to two other topological conditions.
 
\begfig
 \hspace{4.0in}
  \vspace{.2in}
  \centerline{
\qquad\qquad\qquad\quad \includegraphics[width=5.05in]{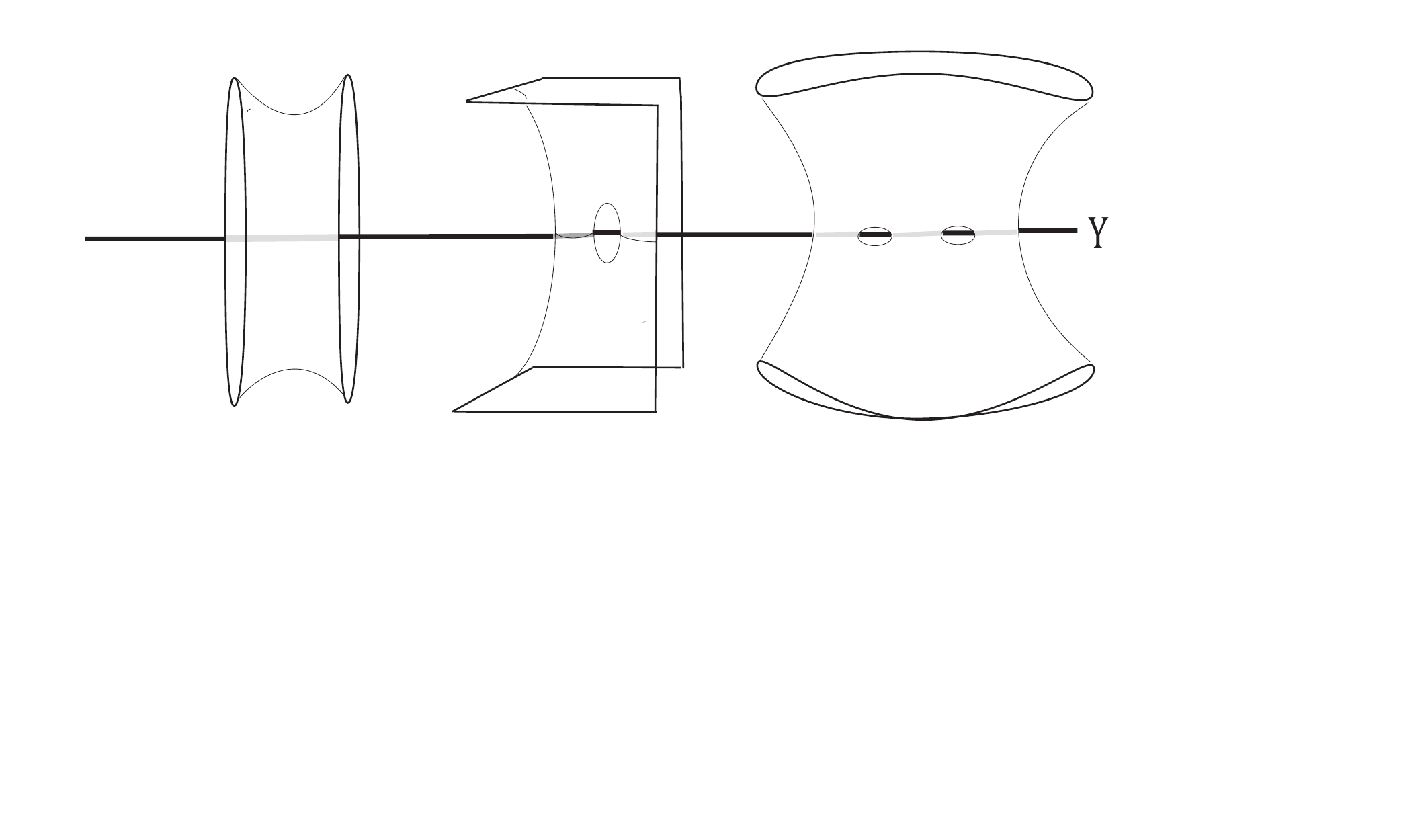}
   }
   \vspace{-2.0 in}
 \begin{center}
 \parbox{5.1in}{
\caption{
{\bf Right:} A $Y$-surface of genus two. 
The number of fixed points of $\rho _Y$ ($180^\circ$ rotation around $Y$) is even (equal to six)
and the number of  boundary components is two.
{\bf Center:} A $Y$-surface of genus one. The number of fixed points of $\rho _Y$ is odd (equal to three) and there is a single boundary component.
{\bf Left:} This annular surface $A$ is not a $Y$-surface. The rotation $\rho_Y$ acts  as the identity  on $H_1(A,\ZZ)$, not as multiplication by $-1$.
 }\label{Y-surface-figure} 
 }
 \end{center}
 \endfig

\newcommand{\Surf}{S}  

\begin{proposition}\label{Y-surface-topology-propostion}
Suppose that $\Surf$ is an open,  orientable Riemannian 
$2$-manifold of finite topology,
that $\rho: \Surf \to \Surf$ is an
 orientation-preserving
 isometry of order two, and that $\Surf/\rho$ is connected.
Then the following are equivalent:
\begin{enumerate}[\upshape(a)]
\item $\rho$ acts by multiplication by $-1$ on the first homology group 
    $H_1(\Surf,\ZZ)$.
\item\label{fixed-points-euler} the quotient $\Surf/\rho$ is topologically a disk. 
\item\label{exactly-fixed-points} $\Surf$ has exactly $2-\chi(\Surf)$ fixed points of $\rho$,
where $\chi(\Surf)$ is the Euler
  characteristic of $\Surf$.
\end{enumerate}
\end{proposition}

Note that proposition~\ref{Y-surface-topology-propostion} is intrinsic in nature.
It does not require that the orientation-preserving automorphism $\rho$ be a reflection in an ambient geodesic $Y$.   
Proposition~\ref{Y-surface-topology-propostion} is easily proved 
using a $\rho$-invariant triangulation of $S$ whose vertices include
the fixed points of $\rho$; details may be found in \cite{hoffman-white-number}.

\begin{corollary}\label{$Y$-corollary}
Let $S$ be an open, orientable $Y$-surface such that $S/\rho_Y$ is connected.
Let $k$ be the number of fixed points of $\rho_Y:S\to S$.
\begin{enumerate}[\upshape(i)]
\item\label{$Y$-corollary-ends}
           The surface $\Surf$ has either one or two ends, according to
          whether $k$ is odd or even.
\item\label{$Y$-corollary-disks}
       If $k=0$, then $S$ is the union of two disks.
\item\label{$Y$-corollary-genus}
       If $k>0$, then $S$ is connected, and the genus of  $S$ is 
      $(k-2)/2$ if $k$ is even and $(k-1)/2$ if $k$ is odd. 
\end{enumerate}
\end{corollary}  

In particular, $S$ is a single disk if and only if $k=1$.

\begin{proof}[Proof of Corollary~\ref{$Y$-corollary}]
Since $S/\rho_Y$ is a disk, it has one end, and thus $S$ has either one
or two ends.  The Euler characteristic of $S$ is $2c-2g-e$, where $c$
is the number of connected components, $g$ is the genus, and $e$
is the number of ends.  
Thus by proposition~\ref{Y-surface-topology-propostion}\eqref{fixed-points-euler},
\begin{equation}\label{genus-fixed-points}
   2 - k = 2c - 2g - e.
\end{equation}
Hence $k$ and $e$ are congruent are congruent modulo $2$.  
Assertion~\eqref{$Y$-corollary-ends} follows immediately.  
(Figure~\ref{Y-surface-figure} shows two examples of assertion~\eqref{$Y$-corollary-ends}.)

Note that if $S$ has more than one component, then since $S/\rho_Y$
is a disk, in fact $S$ must have exactly two components, each of which
must be a disk.  Furthermore, $\rho_Y$ interchanges the two disks,
so that $\rho_Y$ has no fixed points in $S$, i.e., $k=0$.

Conversely, suppose $k=0$.  Then $e=2$ by Assertion~\eqref{$Y$-corollary-ends},
so from~\eqref{genus-fixed-points} we see that
\[
   2c = 2g + 4.
\]
Hence $2c\ge 4$ and therefore $c\ge 2$, i.e., $S$ has two or more components.
But we have just shown that in that case $S$ has exactly two components,
each of which is a disk. 
This completes the proof of Assertion~\eqref{$Y$-corollary-disks}.

Now suppose that $k>0$.  Then as we have just shown, $S$ is connected,
so~\eqref{genus-fixed-points} becomes $k=2g+e$, or
\begin{equation}\label{genus-vs-fixed-points}
    g =  \frac{k-e}2
\end{equation}
This together with Assertion~\eqref{$Y$-corollary-ends} gives
Assertion~\eqref{$Y$-corollary-genus}.
\end{proof}

\begin{remark}
To apply proposition~\ref{Y-surface-topology-propostion} and corollary~\ref{$Y$-corollary}
 to a compact manifold $M$ with non-empty boundary, one
lets $\Surf=M\setminus \partial M$.  The number of ends of $\Surf$ is equal to the
number of boundary components of $M$.
\end{remark}

\begin{proposition}\label{genus-0-proposition}
 If $S$ is a $Y$-surface in $N$ and if $U$ is an open subset of $S$
such that $U$ and $\rho_YU$ are disjoint, then $U$ has genus $0$.
\end{proposition}

\begin{proof}
Note that we can identify $U$ with a subset of $S/\rho_Y$.  Since $S$ is a $Y$-surface,
$S/\rho_Y$ has genus $0$ (by proposition~\ref{Y-surface-topology-propostion})
and therefore $U$ has genus $0$.
\end{proof}

\section{Periodic genus-\texorpdfstring{$g$}{Lg} 
            helicoids in \texorpdfstring{$\sS^2\times\RR$}{Lg}:
            theorem~\ref{main-S2xR-theorem} for \texorpdfstring{$h<\infty$}{Lg}}
\label{construction-outline-section}  

Let $0<h<\infty$.
Recall that we are trying to construct a minimal surface $M$
in $\sS^2\times [-h,h]$ such that
\[
  {\rm interior}(M) \cap H = Z_h \cup Z_h^* \cup X
\]
(where $Z_h$ and $Z_h^*$ are the portions of  $Z$ and $Z^*$ 
where $|z|< h$)
and such that $\partial M$ is a certain pair of circles at heights $h$
and $-h$. 
Since such an $M$ contains $Z_h$, $Z_h^*$, and $X$, it must
(by the Schwarz reflection principle)
be invariant under $\rho_Z$ (which is the same as $\rho_{Z^*}$)
and under $\rho_X$, the $180^\circ$ rotations about $Z$ and about $X$.
It follows that $M$ is invariant under $\rho_Y$, the composition of $\rho_Z$
and $\rho_X$.
In particular, if we let $S=\text{interior}(M)\cap H^+$ be the portion of 
the interior\footnote{It will be convenient for us to have $S$ be an open
manifold, because although $S$ is a smooth surface, its closure has corners.}
 of $M$ in $H^+$,
then 
\[
    M = \overline{S \cup \rho_Z S} = \overline{S\cup \rho_X S}. 
\]
Thus to construct $M$, it suffices to construct $S$. 
 Note that the boundary
of $S$ is $Z_h \cup Z_h^* \cup X$ together with a great semicircle $C$ in $\overline{H^+}\cap\{z=h\}$
and its image $\rho_YC$ under $\rho_Y$.
Let us call that boundary $\Gamma_C$.  (See Figure~\ref{GammaFigure}.)
Thus we wish to construct embedded minimal surface $S$ in $H^+$ having
specified topology and having boundary $\partial S=\Gamma_C$.
Note we need $S$ to be $\rho_Y$-invariant; otherwise Schwarz reflection in $Z$
and Schwarz reflection in $X$ would not produce the same surface.

We will  prove existence by counting surfaces mod $2$. 
 Suppose for the moment
that the curve $\Gamma_C$ is nondegenerate in the following sense: if $S$ is a smooth
embedded, minimal, $Y$-surface in $H^+$ with boundary $\Gamma_C$, then $S$
has no nonzero $\rho_Y$-invariant jacobi fields that vanish on $\Gamma_C$.
For each $g\ge 0$, the number of such surfaces $S$ of genus $g$
turns out to be even.   Of course, for the purposes of proving existence, this fact is not
 useful, since $0$ is an even number.   However, if instead of considering
all $Y$-surfaces of genus $g$, we consider only those that are positive (or those that
are negative) at $O$, then the number of such surfaces turns out to be odd, and therefore
existence follows.

For the next few sections, we fix a helicoid $H$ and we fix an $h$ with $0<h<\infty$.
Our goal is to prove the following theorem:

\begfig
 \hspace{5.0in}
  \vspace{.2in}
  \centerline{
\includegraphics[width=4.05in]{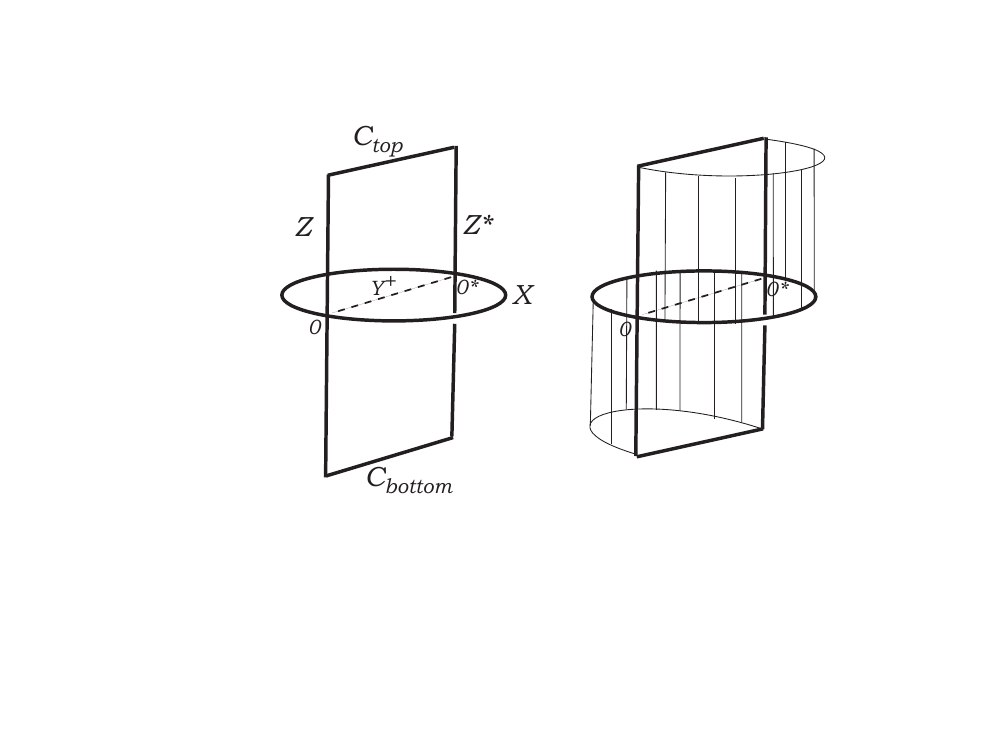}
   }
   \vspace{-1.3 in}
 \begin{center}
 \parbox{5.1in}{
\caption{\label{GammaFigure} {\bf The boundary curve $\Gamma_C$}.
We depict $\sS^2\times\RR$  in these illustrations as $\RR^3$ with  
each horizontal $\sS^2\times\{z\}$ represented as horizontal plane via stereographic projection, 
with one point of the sphere at infinity. Here, that point is the 
antipodal point of the midpoint of the semicircle $Y^+$.
{\bf Right:} For ease of illustration, we have 
chosen the reference helicoid $H$ to be  the vertical cylinder $X\times\RR$, and the semicircle 
$C=C_\text{top}$ to meet $H$ orthogonally.  The geodesics
$X$, $Z$ and $Z^*$ divide $H$ into four components, two of which are shaded. 
The helicoid $H$ divides $\sS^2\times\RR$ into two
components. The component $H^+$ is the interior of the solid cylinder bounded by $H$. 
{\bf Left:} The boundary curve $\Gamma=\Gamma _C$  consists of the great circle  $X$, 
two vertical line segments on the axes $Z\cup Z^*$ of height
$2h$ and two semicircles in $(\sS^2\times\{\pm h\})\cap H^+$. 
Note that $\Gamma$ has $\rho_Y$ symmetry. 
We seek a $\rho_Y$-invariant minimal surface in $H^+$  that has boundary $\Gamma_C$  
and has all of its topology
concentrated along $Y^+$. That is,   
we want a $Y$-surface as defined in section~\ref{section:$Y$-surfaces} with the properties    established in proposition~\ref{Y-surface-topology-propostion}.  According to  theorem~\ref{special-existence-theorem}, there are in fact two such surfaces for every positive genus.
}
 }
 \end{center}
 \endfig

\begin{theorem}\label{special-existence-theorem}
Let $0<h<\infty$, let $C$ be a great semicircle in $\overline{H^+}\cap\{z=h\}$
joining $Z$ to $Z^*$, and let $\Gamma_C$ be the curve given by
\[
   \Gamma_C = Z_h\cup Z_h^*\cup C\cup \rho_YC,
\]
where $Z_h=Z\cap \{|z|\le h\}$ and $Z^*=Z^*\cap\{|z|\le h\}$.  

For each sign $s\in\{+,-\}$ and for each $n\ge 1$, there exists
an open, embedded 
minimal  $Y$-surface $S=S_s$ in $H^+\cap \{|z|<h\}$ such that $\partial S=\Gamma_C$,
such that $Y^+\cap S$ contains exactly $n$ points,
and such that $S$ is 
positive\footnote{Positivity and negativity of $S$ at $O$ were defined in section~\ref{sign-section}.}
 at $O$ if $s=+$ and negative at $O$
if $s=-$.

If $n$ is even, there is such a surface that is invariant under
reflection $\mu_E$ in the totally geodesic cylinder $E\times\RR$.
\end{theorem}

Before proving theorem~\ref{special-existence-theorem}, let us show that it implies the periodic case
of theorem~\ref{main-S2xR-theorem} of section~\ref{section:main-theorems}:

\begin{proposition}\label{reduction-to-H^+-proposition}
Theorem~\ref{special-existence-theorem} implies theorem~\ref{main-S2xR-theorem} in the 
periodic case $h<\infty$.
\end{proposition}

\begin{proof}
Let $C$ be the great semicircle in $\overline{H^+}\cap \{z=h\}$ that has endpoints on
$Z\cup Z^*$ and that meets $H$ orthogonally at those endpoints.
First suppose $n$ is even, and let $S_s$ for $s\in\{+,-\}$ be the surfaces
given by theorem~\ref{special-existence-theorem}.
Let $M_s$ be the surface obtained by Schwarz
reflection from $S_s$:
\[
   M_s = \overline{S_s \cup \rho_Z S_s} = \overline{S_s\cup \rho_X S_s}, 
\]
(The second equality holds because $S_s$ is $\rho_Y$-invariant and $\rho _Z\circ \rho_Y =\rho_X$.)

By lemma~\ref{Schwarz-regularity-lemma} below, 
  $M_s$ is a smoothly embedded minimal surface.  
  Clearly it is $\rho_Y$-invariant, it lies
   in  $\sS^2\times [-h,h]$,
its interior has the desired intersection with $H$, it has the indicated sign at $O$,
it has $\mu_E$ symmetry, 
and its boundary is the desired pair of horizontal circles.
We claim that $M_s$ is a $Y$-surface.  To see this, note that 
 since $S_s$
is a $Y$-surface, the quotient $S_s/\rho_Y$ is 
topologically a disk by proposition~\ref{Y-surface-topology-propostion}.
The interior of $M_s/\rho_Y$ is two copies of $S_s/\rho_Y$ glued along a common
boundary segment.  Thus the interior of $M_s/\rho_Y$ is also topologically a disk, and therefore
$M_s$ is a $Y$-surface by proposition~\ref{Y-surface-topology-propostion}.

Note that $M_s\cap Y$ has $2n+2$ points: the $n$ points in $S_s\cap Y^+$,
an equal number of points in $\rho_ZS_s\cap Y^-$, and the two
points $O$ and $O^*$.  Thus by corollary~\ref{$Y$-corollary}, 
$M_s$ has genus $n$.  Since $n$
is an arbitrary even number, this completes the proof for even genus, except
for assertion~\eqref{even-genus-congruence-assertion}, 
the assertion that $M_{+}$ and $M_{-}$ are not congruent.

Now let $n$ be odd, and let $S_{+}$ be the surface given by 
theorem~\ref{special-existence-theorem}.
By lemma~\ref{parity-sign-lemma} below, $S_{+}$ is negative at $O^*$,
which  implies that $\mu_E(S_{+})$ is 
negative at $O$.  In this case, we choose our $S_{-}$ to be $\mu_E(S_{+})$.
Exactly as when $n$ is even, we extend $S_{\pm }$ by Schwarz reflection
to get $M_{\pm }$.   As before, the $M_{\pm }$ are $Y$-surfaces
of genus $n$.  The proof that they have the required properties is exactly
as in the case of even $n$, except for the statement that $M_{+}$ and $M_{-}$
are not congruent by any orientation-preserving isometry of $\sS^2\times \RR$.

It remains only to prove the statements about noncongruence of $M_{+}$ and 
 $M_{-}$.  Those statements (which we never use) 
 are proved in section~\ref{noncongruence-appendix}.
\end{proof}

The proof above used the following two lemmas:

\begin{lemma}\label{Schwarz-regularity-lemma}
 If $S$ is a $\rho_Y$-invariant embedded minimal surface in $H^+$  with boundary
$\Gamma$ and with $Y\cap S$ a finite set,  then the Schwarz-extended surface
\[
      M = \overline{S \cup \rho_Z S} = \overline{S \cup \rho_X S}. 
\]
is smoothly embedded everywhere.
\end{lemma}

\begin{proof}
One easily checks that if $q$ is a corner of $\Gamma$ other than $O$ or $O^*$, 
then the tangent cone to $S$ at $q$ is a multiplicity-one quarter plane.  
 Thus the tangent
cone to $M$ at $q$ is a multiplicity-one halfplane, which implies that $M$ is smooth at $q$
by Allard's boundary regularity theorem~\cite{allard-boundary}.
(In fact, the classical boundary regularity theory~\cite{hildebrandt} suffices here.)

Let $B$ be an open ball centered at $O$ small enough that $B$ contains no points of $Y\cap S$.
Now $S\cap B$ is a $Y$-surface, so by 
corollary~\ref{$Y$-corollary}\eqref{$Y$-corollary-disks}, it is topologically the union of two disks.  It follows
that $M\cap B$ is a disk, 
so $M$ is a branched minimal immersion at $O$ by~\cite{gulliver-removability}.
But since $M$ is embedded, in fact $M$ is unbranched.
\end{proof}

\begin{lemma}\label{parity-sign-lemma}
 Let $S\subset H^+$ be  a $Y$-surface with $\partial S=\Gamma$. 
 Then the signs of $S$ at $O$ and $O^*$ agree or disagree according to whether
 the number of points of $Y\cap S$ is even or odd.
\end{lemma}

\begin{proof}
Let $\widehat{S}$ be the geodesic completion of $S$.
We can identity $\widehat{S}$ with $\overline{S}=S\cup\partial S$, except
that $O\in \overline{S}$ corresponds to two points in $\widehat{S}$,
and similarly for $O^*$.
Note that the number of ends of $S$ is equal to the number of boundary components of 
 $\partial \widehat{S}$.

By symmetry, we may assume that the sign of $S$ at $O$ is $+$. 
Then at $O$, $Z^+$ is joined in $\partial \widehat{S}$ to $X^+$ and $Z^-$ is joined to $X^-$.
 If the sign of $S$ at $O^*$ is also $+$, then  the same pairing occurs at $O^*$, 
 from which it follows that 
 $\partial \widehat{S}$ has two components and therefore that $S$ has two ends. 
 If the sign of $S$ at $O^*$ is $-$, then  the pairings are crossed, so that
$\partial \widehat{S}$ has  only one component and therefore $S$ has only one end.
Thus $S$ has two ends or one end
according to whether the signs of $S$ at $O$ and $O^*$ are equal or not.
The lemma now follows from corollary~\ref{$Y$-corollary},
 according to which the number of ends of $S$
is two or one according to whether the number of points of $Y\cap S$ is even or odd.
\end{proof}






\section{Adjusting the pitch of the helicoid}\label{adjusting-pitch-section}

Theorem~\ref{special-existence-theorem}  of section~\ref{construction-outline-section} asserts that
the curve $\Gamma_C$ 
bounds various minimal surfaces in $H^+$.  In that theorem, $C=\Gamma\cap \{z=h\}$ is allowed
to be any semicircle in $\overline{H^+}\cap \{z=h\}$ with endpoints in $Z\cup Z^*$.  
In this section, we will show that in order to prove 
theorem~\ref{special-existence-theorem}, it is sufficient to prove it for the special case where
 $C$ is a semicircle in the helicoid $H$.

\begin{theorem}\label{special-existence-theorem-tilted}[Special case of 
theorem~\ref{special-existence-theorem}]
Let $0<h<\infty$ and let $C$ be
 the semicircle in $H\cap\{z=h\}$ joining $Z$ to $Z^*$ such that $C$ and $X^+$
lie in the same component of $H\setminus (Z\cup Z^*)$.
For each sign $s\in\{+,-\}$ and for each $n\ge 1$, there exists
an open embedded 
minimal  $Y$-surface $S=S_s$ in $H^+\cap \{|z|<h\}$ such that 
\[
   \partial S=\Gamma_C :=  Z_h \cup Z_h^* \cup C \cup \rho_YC,
\]
such that $Y^+\cap S$ contains exactly $n$ points,
and such that $S$ is positive at $O$ if $s=+$ and negative at $O$
if $s=-$.

If $n$ is even, there is such a surface that is invariant under
reflection $\mu_E$ in the totally geodesic cylinder $E\times\RR$.
\end{theorem}
 We will prove that theorem~\ref{special-existence-theorem-tilted}, a special case of theorem~\ref{special-existence-theorem}, is in fact equivalent to it:
 
 \begin{proposition} Theorem~\ref{special-existence-theorem-tilted} implies theorem~\ref{special-existence-theorem}. \label{tilted-suffices-proposition}
 \end{proposition}

 \begin{proof}[Proof of Proposition~\ref{tilted-suffices-proposition}]
Let $H$ be a helicoid and let $C$ be a great semicircle in
   $\overline{H^+}\cap\{z=h\}$.
We may assume that $C$ does not lie in $H$, as otherwise there is nothing to prove.
Therefore the interior of the semicircle $C$ lies in $H^+$.
Now increase (or decrease) the pitch of $H$  to get a one-parameter family of helicoids
$H(t)$ with $0\le t\le 1$ such that
\begin{enumerate}
\item $H(1)=H$,
\item $C \subset \overline{H(t)^+}$ for all $t\in [0,1]$,
\item $C \subset H(0)$.
\end{enumerate}

\begin{claim} Suppose $S$ is an open, $\rho_Y$-invariant, embedded minimal surface bounded by $\Gamma_C$
with $S\cap Y^+$ nonempty.
If $S$
 is contained in $H(0)^+$, then it is contained in $H(t)^+$ for all $t\in[0,1]$.
Furthermore, in that case
 the sign of $S$ at $O$ with respect to $H(t)$ does not depend on $t$.
\end{claim}

\begin{proof}[Proof of claim] 
Let $T$ be the set of $t\in [0,1]$ for which $S$ is contained in $\overline{H(t)^+}$. 
Clearly $T$ is a closed set.
We claim that $T$ is also open relative to $[0,1]$.
For suppose that $t\in T$, and thus that $S\subset \overline{H(t)^+}$.
Now $S$ is not contained in $H(t)$ since $S\cap Y^+$ is nonempty.
Thus by the strong maximum principle and the strong boundary maximum
principle, $S$ cannot touch $H(t)$, nor is $\overline{S}$ tangent
to $H(t)$ at any points of $\Gamma_C$ other than its corners.

At the corners $O$ and $O^*$, $S$ and $H(\tau)$ are tangent.  However, the curvatures of $H$ and
$
    M:=\overline{S\cup \rho_YS}
$
differ from each other\footnote{Recall that if two minimal surfaces in a $3$-manifold 
are tangent at a point, then the intersection set
near the point is like the zero set of a homogeneous harmonic polynomial.  In particular, it consists
of $(n+1)$ curves crossing through the point, where $n$ is the degree of contact of the two surfaces at the point.
 Near $O$, the intersection of $M$ and $H$ coincides with $X\cup Z$, so their order of contact  
at $O$ is exactly one.} 
 at $O$, and also at $O^*$. 
It follows readily that $t$ is in the interior of $T$ relative to $[0,1]$.  
 Since $T$ is open and closed in $[0,1]$ and is nonempty, $T=[0,1]$.
This proves the first assertion of the claim.  The second follows by continuity.
\end{proof}

By the claim, if theorem~\ref{special-existence-theorem-tilted} 
is true for $\Gamma_C$ and $H(0)$, then theorem~\ref{special-existence-theorem}
 is true for $H=H(1)$ and $\Gamma_C$. 
 This completes the proof of proposition~\ref{tilted-suffices-proposition}.
 \end{proof}


\section{Eliminating jacobi fields by perturbing the metric}\label{bumpy-section}

Our proof involves counting minimal surfaces mod $2$.  Minimal surfaces
with nontrivial jacobi fields tend to throw off such counts.  (A nontrivial jacobi field
is a nonzero normal jacobi field that vanishes on the boundary.)
Fortunately, if we fix a curve $\Gamma$ in a $3$-manifold, then 
a generic Riemannian metric on the $3$-manifold
 will be  ``bumpy" (with respect to $\Gamma$)
 in the following sense: $\Gamma$ will not bound any 
minimal surfaces with
nontrivial jacobi fields.  
 Thus instead of working with the 
standard product metric on
$\sS^2\times\RR$, we will use
 a slightly perturbed bumpy metric and prove
 theorem~\ref{special-existence-theorem-tilted}
for that perturbed metric.  By taking a limit of surfaces as the perturbation
goes to $0$, we get the surfaces whose existence is asserted in
 theorem~\ref{special-existence-theorem-tilted} for the standard metric.
In this section, we explain how to perturb the metric to make
it bumpy, and how to take the limit as the perturbation goes to $0$.

In what class of metrics should we make our perturbations?
The metrics should have $\rho_X$ and $\rho_Z$ symmetry so that we can do Schwarz reflection,
$\rho_Y$ symmetry so that the notion of $Y$-surface makes sense,
and $\mu_E$-symmetry so that the conclusion of theorem~\ref{special-existence-theorem-tilted} makes sense.
It is convenient to use metrics for which the helicoid $H$ and the spheres $\{z=\pm h\}$ are minimal, 
because we will need  
the region $N=\overline{H^+}\cap \{|z|\le h\}$ to be weakly mean-convex.
We will also need to have an isoperimetric inequality hold for minimal surfaces in $N$,
which is equivalent (see remark~\ref{isoperimetric-equivalence-remark}) 
to the nonexistence of any smooth, closed minimal surfaces
in $N$.  Finally, at one point (see the last sentence in section~\ref{smooth-count-section}) 
we will need the two bounded components of $H\setminus\Gamma$
to be strictly stable, so we restrict ourselves to metrics for which they are strictly stable.

The following theorem (together with its corollary) is theorem~\ref{special-existence-theorem-tilted}
with the standard metric on $\sS^2\times\RR$ replaced by 
a suitably bumpy metric in the class of metrics described above,
and with the conclusion strengthened to say that $\Gamma_C$
bounds an odd number of surfaces with the desired properties:

\begin{theorem}\label{main-bumpy-theorem}
Let $\Gamma=\Gamma_C$ be the curve in theorem~\ref{special-existence-theorem-tilted}:
\[
 \Gamma= Z_h \cup Z_h^* \cup C \cup \rho_YC,
\]
 where $C$ is the semicircle in $H\cap\{z=h\}$ joining $Z$ to $Z^*$ such that $C$ and $X^+$
lie in the same component of $H\setminus (Z\cup Z^*)$.
Let $G$ be the group of isometries of $\sS^2\times\RR$ generated
by $\rho_X$, $\rho_Y$, $\rho_Z=\rho_{Z^*}$, and $\mu_E$.
 Let $\gamma$ be a smooth, $G$-invariant Riemannian metric on $\sS^2\times\RR$
such that 
\begin{enumerate}
\item\label{minimal-anchors-hypothesis}
        the helicoid $H$ and the horizontal spheres $\{z=\pm h\}$ are $\gamma$-minimal surfaces.
\item\label{strictly-stable-hypothesis} the two bounded components of $H\setminus \Gamma$ are strictly stable
     (as $\gamma$-minimal surfaces).
\item\label{no-closed-minimal-surface-hypothesis} 
    the region $N:=\overline{H^+}\cap \{|z|\le h\}$ contains no smooth, closed, embedded $\gamma$-minimal surface.
\item\label{bumpy-hypothesis}
    the curve $\Gamma$ does not bound any embedded $\gamma$-minimal $Y$-surfaces
   in $\overline{H^+}\cap \{|z|\le h\}$ with nontrivial $\rho_Y$-invariant jacobi fields.
\end{enumerate}
For each nonnegative
 integer $n$ and each sign $s\in \{+,-\}$, let
\[
    \MM^s(\Gamma,n) = \MM^s_\gamma(\Gamma,n)
\]
denote the set of 
 embedded,  $\gamma$-minimal $Y$-surfaces $S$ in 
  $\overline{H^+}\cap\{|z|\le h\}$ bounded by $\Gamma$ such that $S\cap Y^+$ has exactly $n$
  points and such that $S$ has sign $s$ at $O$.
Then the number of surfaces in $\MM^s(\Gamma,n)$ is odd.
\end{theorem}

\begin{corollary}\label{extra-symmetry-corollary}
Under the hypotheses of the theorem, if $n$ is even, then the number of 
$\mu_E$-invariant surfaces in $\MM^s(\Gamma,n)$ is odd.
\end{corollary}

\begin{proof}[Proof of corollary]
Let $n$ be even.  By lemma~\ref{parity-sign-lemma},
 if $S\in \MM^s(\Gamma,n)$, then $S$ also has sign $s$ at $O^*$,
from which it follows that $\mu_E(S)\in \MM^s(\Gamma,n)$.
Thus the number of non-$\mu_E$-invariant surfaces in $\MM^s(\Gamma,n)$ is even because
such surfaces come in pairs ($S$ being paired with $\mu_ES$). 
By the theorem, the total number of surfaces in $\MM^s(\Gamma,n)$ is odd, so therefore
the number of $\mu_E$-invariant surfaces must also be odd.
\end{proof}

\begin{remark}\label{isoperimetric-equivalence-remark}
Hypothesis~\eqref{minimal-anchors-hypothesis} of theorem~\ref{main-bumpy-theorem}
 implies that the compact region $N:=\overline{H^+}\cap\{|z|\le h\}$ is $\gamma$-mean-convex.
It follows (see~\cite{white-isoperimetric}*{\S2.1 and \S5}) that 
condition~\eqref{no-closed-minimal-surface-hypothesis}
is equivalent to the following condition:
\begin{enumerate}
\item[(3$'$)] There is a finite constant $c$ such that $\area(\Sigma)\le c \length(\partial \Sigma)$
                for every $\gamma$-minimal surface $\Sigma$ in $N$.
\end{enumerate}
Furthermore, the proof of theorem~2.3 in~\cite{white-isoperimetric}
 shows that for any compact set $N$, the set of Riemannian metrics  
 satisfying~\thetag{3$'$} 
is open, with a constant $c=c_{\gamma}$ that depends 
upper-semicontinuously on the 
metric\footnote{As explained in \cite{white-isoperimetric}, for any metric $\gamma$,
we can let $c_\gamma$ be the supremum (possibly infinite)
of $|V|/|\delta V|$ among all $2$-dimensional varifolds $V$ in $N$ with $|\delta V|<\infty$,
 where $|V|$ is the mass of $V$ and $|\delta V|$
is its total first variation measure.
 The supremum is attained by a varifold $V_\gamma$
with mass $|V_\gamma|=1$.  Suppose $\gamma(i)\to\gamma$.  By passing to a subsequence, we may
assume that the $V_{\gamma(i)}$ converge weakly to a varifold $V$.
Under weak convergence, mass is continuous and total first variation measure is lower
semicontinuous.  Thus 
\[
   c_\gamma \ge \frac{ |V| }{|\delta V|} 
   \ge \limsup \frac{|V_{\gamma(i)}|}{|\delta V_{\gamma(i)}|}
   = \limsup c_{\gamma(i)}.
\]
This proves that the map $\gamma\mapsto c_\gamma\in (0,\infty]$ is upper semicontinuous,
and therefore also that the set of metrics $\gamma$ for which $c_\gamma<\infty$ is 
an open set.  (The compactness, continuity, and lower semicontinuity results used here
are easy and standard, and are explained in the appendix to~\cite{white-isoperimetric}.
See in particular~\cite{white-isoperimetric}*{\S7.5}.)
}.   
\end{remark}

\begin{proposition}\label{bumpy-suffices-proposition}
Suppose theorem~\ref{main-bumpy-theorem} is true.  
Then theorem~\ref{special-existence-theorem-tilted} is true.
\end{proposition}

\newcommand{\Gbig}{\mathcal{G}_1}
\newcommand{\Gmedium}{\widehat{\mathcal{G}}}
\newcommand{\Gsmall}{\mathcal{G}}

\begin{proof}
Let $\Gbig$ be the space of all smooth, $G$-invariant Riemannian
metrics $\gamma$ on $\sS^2\times\RR$
that satisfy hypothesis~\eqref{minimal-anchors-hypothesis} of 
the theorem.  
Let $\Gmedium$ be the subset consisting of those metrics $\gamma\in \Gbig$ such that
also satisfy hypotheses~\eqref{strictly-stable-hypothesis} 
and~\eqref{no-closed-minimal-surface-hypothesis} 
of  theorem~\ref{main-bumpy-theorem}, and let $\Gsmall$ be the set of metrics 
that satisfy all the hypotheses of the theorem.

We claim that the standard product metric $\gamma$ belongs to $\Gmedium$. 
Clearly it is $G$-invariant and satisfies hypothesis~\eqref{minimal-anchors-hypothesis}.
Note that each bounded component of $H\setminus \Gamma$ is strictly stable, because it is contained
in one of the half-helicoidal components of $H\setminus(Z\cup Z^*)$ and those half-helicoids
are stable (vertical translation induces  a positive jacobi field). 
Thus $\gamma$ satisfies the strict stability
hypothesis~\eqref{strictly-stable-hypothesis}.  
It also satisfies hypothesis~\eqref{no-closed-minimal-surface-hypothesis}
because if $\Sigma$ were a closed minimal surface in $N$, then the height function $z$ would attain
a maximum value, say $a$, on $\Sigma$, which implies by the strong maximum principle that
 the sphere $\{z=a\}$ would be contained in $\Sigma$, contradicting the fact that 
    $\Sigma\subset N\subset \overline{H^+}$.
 This completes the proof that the standard product metric $\gamma$ belongs to $\Gmedium$.
 
By lemma~\ref{bumpy-lemma} below, a generic metric in $\Gbig$ satisfies the bumpiness
hypothesis~\eqref{bumpy-hypothesis} of theorem~\ref{main-bumpy-theorem}.
Since $\Gmedium$ is an open subset of $\Gbig$ (see  remark~\ref{isoperimetric-equivalence-remark}),
it follows that a generic metric in $\Gmedium$ satisfies the bumpiness hypothesis.  
In particular, this means that $\Gsmall$ is a dense subset of $\Gmedium$.

Since the standard metric $\gamma$ is in $\Gmedium$, 
there is a sequence $\gamma_i$ of metrics in $\Gsmall$
that converge smoothly to $\gamma$. 
Fix a nonnegative integer $n$ and a sign $s$.
By theorem~\ref{main-bumpy-theorem},  $\MM^s_{\gamma_i}(\Gamma,n)$
contains at least one surface $S_i$.  If $n$ is even, we choose $S_i$ to
be $\mu_E$-invariant, which is possible by corollary~\ref{extra-symmetry-corollary}.

By remark~\ref{isoperimetric-equivalence-remark},
\begin{equation}\label{isoperimetric-ratio}
  \limsup_i \frac{\area_{\gamma_i}(S_i)}{\length_{\gamma_i}(\partial S_i)} \le c_\gamma
\end{equation}
where $c_\gamma$ is the constant 
 in remark~\ref{isoperimetric-equivalence-remark}
for the standard product metric $\gamma$.   Since
\[
   \length_{\gamma_i}(\partial S_i) = \length_{\gamma_i}(\Gamma)\to \length_\gamma(\Gamma)<\infty,
\]
we see from~\eqref{isoperimetric-ratio} that the areas of the $S_i$ are uniformly bounded.


Let
\[
  M_i = \overline{S_i\cup \rho_Z S_i}
\]
be obtained from $S_i$ by Schwarz reflection.
Of course the areas of the $M_i$ are also uniformly bounded.
Using  the Gauss-Bonnet theorem, the minimality of the $M_i$, and the fact that the sectional curvatures 
of $\sS^2\times\RR$ are bounded, it follows that 
\begin{equation}\label{total-curvature-bound}
  \sup_i \int_{M_i} \beta(M_i,\cdot)\,dA < \infty ,
\end{equation}
where $\beta(M_i,p)$ is the square of the norm of the 
second fundamental  form of $M_i$ at the point $p$.

The total curvature bound~\eqref{total-curvature-bound} implies (see~\cite{white-curvature-estimates}*{theorem~3})
that after passing to a further subsequence, the $M_i$ converge 
smoothly to an embedded minimal surface $M$, which implies that the $S_i$ converge uniformly smoothly to 
a surface $S$ in $N$ with $\partial S=\Gamma$ and with $M=\overline{S\cup\rho_Y S}$.
The smooth convergence $M_i\to M$ implies
that $S\in \MM^s_\gamma(\Gamma,n)$, where $\gamma$ is the standard product
metric.  Furthermore, if $n$ is even, then $S$ is $\mu_E$-invariant.
This completes the proof of theorem~\ref{special-existence-theorem-tilted} (assuming 
theorem~\ref{main-bumpy-theorem}).
\end{proof}

\begin{lemma}\label{bumpy-lemma}
Let $\Gbig$ be the set of smooth, $G$-invariant metrics $\gamma$ on $\sS^2\times \RR$ such
that the helicoid $H$ and the spheres $\{z=\pm h\}$ are $\gamma$-minimal.
For a generic metric $\gamma$ in $\Gbig$, the curve $\Gamma$ bounds no 
embedded, $\rho_Y$-invariant, $\gamma$-minimal surfaces with nontrivial $\rho_Y$-invariant
jacobi fields.
\end{lemma}

\begin{proof}
By the bumpy metrics theorem~\cite{white-bumpy}, a generic metric $\gamma$ in $\Gbig$ has the property
\begin{enumerate}
\item[(*)]\label{circles-bumpy-item}
The pair of circles $H\cap\{z=\pm h\}$ bounds no embedded $\gamma$-minimal
 surface in $H\cap\{|z|\le h\}$ with a nontrivial jacobi field.
\end{enumerate}
Thus it suffices to prove that if $\gamma$ has the property~\thetag{*},
and if $S\subset N$ is an embedded, $\rho_Y$-invariant, $\gamma$-minimal
surface with boundary $\Gamma$, then $S$ has no nontrivial $\rho_Y$-invariant jacobi field.

Suppose to the contrary that $S$ had such a nontrivial jacobi field $v$.  Then
$v$ would extend by Schwarz reflection to a nontrivial jacobi field on 
$
  M:= \overline{S\cup \rho_Y S},
$
contradicting~\thetag{*}.
\end{proof}


\section{Rounding the curve \texorpdfstring{$\Gamma$}{Lg} 
              and the family of surfaces \texorpdfstring{$t\mapsto S(t)$}{Lg}}
\label{rounding-section}

Our goal for the next few sections is to prove theorem~\ref{main-bumpy-theorem}. 
The proof is somewhat involved.  
It will be completed in section~\ref{smooth-count-section}.
From now until the end of section~\ref{smooth-count-section}, we fix a helicoid $H$ in $\sS^2\times\RR$
    and a height $h$ with $0<h<\infty$.
We let $\Gamma=\Gamma_C$ be the curve in theorem~\ref{main-bumpy-theorem}.
We also fix a  Riemannian metric on $\sS^2\times\RR$ that satisfies the
hypotheses of theorem~\ref{main-bumpy-theorem}.
In particular, in sections~\ref{rounding-section}\,--\,\ref{smooth-count-section}, 
 every result is with respect to that Riemannian metric.
In reading those sections, it may be helpful to imagine that the metric
is the standard product metric.  (In fact, for the purposes of proving theorem~\ref{main-S2xR-theorem},
the metric may as well be arbitrarily close to the standard product metric.)
Of course, in carrying out the proofs in sections~\ref{rounding-section}\,--\,\ref{smooth-count-section}, 
we must take care to use no property of the metric
other than those enumerated in theorem~\ref{main-bumpy-theorem}.

Note that theorem~\ref{main-bumpy-theorem} is about counting minimal surfaces mod $2$.
The  mod $2$ number of embedded minimal surfaces of a given topological type bounded by a smoothly 
embedded, suitably  bumpy curve is rather well understood. 
 For example, if the curve lies on the boundary
of a strictly convex set in $\RR^3$, the number is $1$ if the surface is a disk and is $0$ if not.
Of course the curve $\Gamma$ in theorem~\ref{main-bumpy-theorem} 
is neither smooth nor embedded, so to take advantage of such results,
we will round the corners of $\Gamma$ to make a smooth embedded curve, and we will 
use information
about the mod $2$ number of various surfaces bounded by the rounded curve to deduce
information about mod $2$ numbers of various surfaces bounded by the original curve $\Gamma$.

In this section, we define the notion of a ``rounding''. 
A rounding of $\Gamma$ is a one-parameter family 
  $t\in (0,\tau]\mapsto \Gamma(t)$ of smooth embedded curves (with
certain properties) that converge to $\Gamma$ as $t\to 0$.
Now if $\Gamma$ were smooth and bumpy, then by the implicit function theorem, any smooth
minimal surface $S(0)$ bounded by $\Gamma$ would extend uniquely to a one-parameter family
$t\in [0,\tau']\mapsto S(t)$ of minimal surfaces with $\partial S(t)\equiv \Gamma(t)$ (for some possibly smaller 
  $\tau'\in (0,\tau]$.)

It is natural to guess that this is also the case even in our situation, when $\Gamma$ is neither smooth
nor embedded.   
In fact, we prove that the guess is 
correct\footnote{The correctness of the guess can be viewed as a kind of bridge theorem. 
Though it does not quite follow from the bridge theorems in~\cite{smale-bridge} or  
in~\cites{white-stable-bridge,white-unstable-bridge}, we believe the proofs
there could be adapted to our situation.  However, the proof here is shorter and more elementary than
those proofs.
(It takes advantage of special properties of our surfaces.)}.
The proof is  still based on the implicit function
theorem, but the corners make the proof significantly more complicated.  However, the idea of the proof
is simple: we project the rounded curve $\Gamma(t)$ to a curve in the 
surface
\[  
  M:=\overline{S\cup \rho_ZS}
\] 
by the nearest point projection. We already have a minimal surface bounded by that projected curve:
it bounds a portion $\Omega(t)$ of $M$.  Now we smoothly isotope the projected curve back to $\Gamma(t)$,
and use the implicit function theorem to make a corresponding isotopy through minimal surfaces of $\Omega(t)$
to the surface $S(t)$ we want.  Of course we have to be careful to verify that we do not encounter
nontrivial jacobi fields on the way.

We also prove that, roughly speaking, the surfaces $S(t)$  (for the various $S$'s bounded by $\Gamma$)
account for {\em all} the minimal $Y$-surfaces bounded by
$\Gamma(t)$ when $t$ is sufficiently small.  
The precise statement (theorem~\ref{all-accounted-for-theorem}) is 
slightly more complicated because the larger the genus of the surfaces, the 
smaller one has to choose $t$.

Defining roundings, proving the existence of the associated one-parameter families $t\mapsto S(t)$
of minimal surfaces as described above, and proving basic properties of such families take up
the rest of this section and the following section.  Once we have those tools, 
the proof of theorem~\ref{main-bumpy-theorem}
 is not so hard: it is carried out
in section~\ref{counting-section}.  

To avoid losing track of the big picture, 
the reader may find it helpful initially to skip 
sections~\ref{Gamma(s,t)-definition}--\ref{eigenfunction-claim} 
(the proof of theorem~\ref{bridged-approximations-theorem})
as well as the proofs in section~\ref{additional-properties-section}, 
and then to read section~\ref{counting-section}, which contains the heart
of the proof of theorem~\ref{main-bumpy-theorem} and therefore also
(see remark~\ref{periodic-case-done-remark}) of the periodic case of theorem~\ref{main-S2xR-theorem}.

\begin{lemma}\label{TubularNeighborhoodLemma}
Suppose that  $S$   a minimal embedded $Y$-surface in $N=\overline{H^+}\cap\{|z|\le h\}$ with
 $\partial S=\Gamma$.   
Let 
\[
    V(S,\eps) = \{p\in \sS^2\times\RR: \dist(p,S) <\eps\}.
\]
For all sufficiently small $\eps>0$, the following hold:
\begin{enumerate}
\item\label{nearest-point-item} if $p\in V(S,\eps)$, then there is a unique point $\pi(p)$ in $\overline{S\cup\rho_Z S}$ nearest to $p$.
\item\label{unique-in-V-item}
if $S'$ is a $\rho_Y$-invariant minimal surface in $\overline{V(S,\eps)}$ with $\partial S'=\Gamma$,
  and if $S'$ is smooth except possibly at the corners of $\Gamma$,  then $S'=S$.
\end{enumerate}
\end{lemma}

\begin{proof}[Proof of Lemma~\ref{TubularNeighborhoodLemma}]
Assertion~\eqref{nearest-point-item} holds (for sufficiently small $\eps$) 
because $M:=\overline{S\cup \rho_ZS}$ is a
smooth embedded manifold-with-boundary.

Suppose assertion~\eqref{unique-in-V-item} fails.
Then there is a sequence of minimal $Y$-surfaces $S_n\subset \overline{V(S,\eps_n)}$ 
with $\partial S_n=\Gamma$ such that $S_n\ne S$ and such that $\eps_n\to 0$.
 Let $M_n$ be the closure of $S_n\cap \rho_Z S_n$ or (equivalently) of $S_n\cap \rho_X S_n$. 
 (Note that $\rho_Z S_n=\rho_X S_n$ by the $\rho_Y$-invariance of $S_n$.)
Then $M_n$ is a 
minimal surface with boundary $\partial M_n=\partial M$, $M_n$ is smooth away from $Y$ and from the corners of $\Gamma$, and 
\[
   \max_{p\in M_n}\dist(p,M)\to 0.
\]
Since $M$ is a smooth, embedded manifold with nonempty boundary,
 this implies 
that the convergence $M_n\to M$ is smooth by the extension of 
Allard's boundary regularity theorem in \cite{white-controlling-area}*{6.1}.

A  {\em normal graph} of $f:S\to \RR$ over a hypersurface $S$ in a Riemannian manifold is  the hypersurface
$\{\exp_p(f(p)n(p))\, |\,\, p\in S\}$, where n(p) is a unit normal
vector field on $S\subset N$ and $\exp _p$ is the exponential mapping  at 
$p$.
From the previous paragraph, it follows that for all sufficiently large $n$,  $M_n$ is the 
    normal graph of a function $f_n:M\to \RR$ 
with $f_n\vert\Gamma=0$ such that $f_n\to 0$
smoothly.  
But then 
\[
  \frac{f_n}{\|f_n\|_0}
\]
converges (after passing to a subsequence) to a nonzero jacobi field on $S$ that vanishes
on $\partial S=\Gamma$, contradicting the assumption (hypothesis~\eqref{bumpy-hypothesis}
of theorem~\ref{main-bumpy-theorem}) that 
the Riemannian metric is bumpy with respect to $\Gamma$.
\end{proof}


\stepcounter{theorem}

\addtocontents{toc}{\SkipTocEntry}
\subsection{Roundings of \texorpdfstring{$\Gamma$}{Lg}}\label{roundings-subsection}
 Let $t_0>0$ be less than half the
distance between any two corners of $\Gamma$. For $t$ satisfying
$0<t\leq t_0$, we  can form from $\Gamma$ a smoothly embedded $\rho_Y$-invariant
curve $\Gamma (t)$ in the portion of $H$ with $|z|\le h$ as follows:
\begin{enumerate}
\item If $q$ is a corner of $\Gamma$ other than $O$ or $O^*$, we
replace $\Gamma\cap\BB(q,t)$ by  a smooth curve in $H\cap
\BB(q,t)$ that has the same endpoints as $\Gamma\cap \BB(q,t)$ but
that is otherwise disjoint from $\Gamma\cap\BB(q,t)$. 
\item 
  If $q=O$
or $q=O^*$ we replace $\Gamma\cap \BB(q,t)$ by two smoothly embedded
curves in $H$ that have the same endpoints as $\Gamma\cap\BB(q, t)$ but that
are otherwise disjoint from $\Gamma\cap \BB(q,t)$. See
Figures~\ref{rounding-corners-figure} and \ref{signs-figure}.
\end{enumerate}
Note that $\Gamma(t)$ lies in the boundary of $\partial N$ of the region $N=\overline{H^+}\cap\{|z|\le h\}$.


\begfig
 \hspace{-1.0in}
  \vspace{-2.0in}
  \centerline{
\qquad\qquad\qquad \includegraphics[width=4.05in]{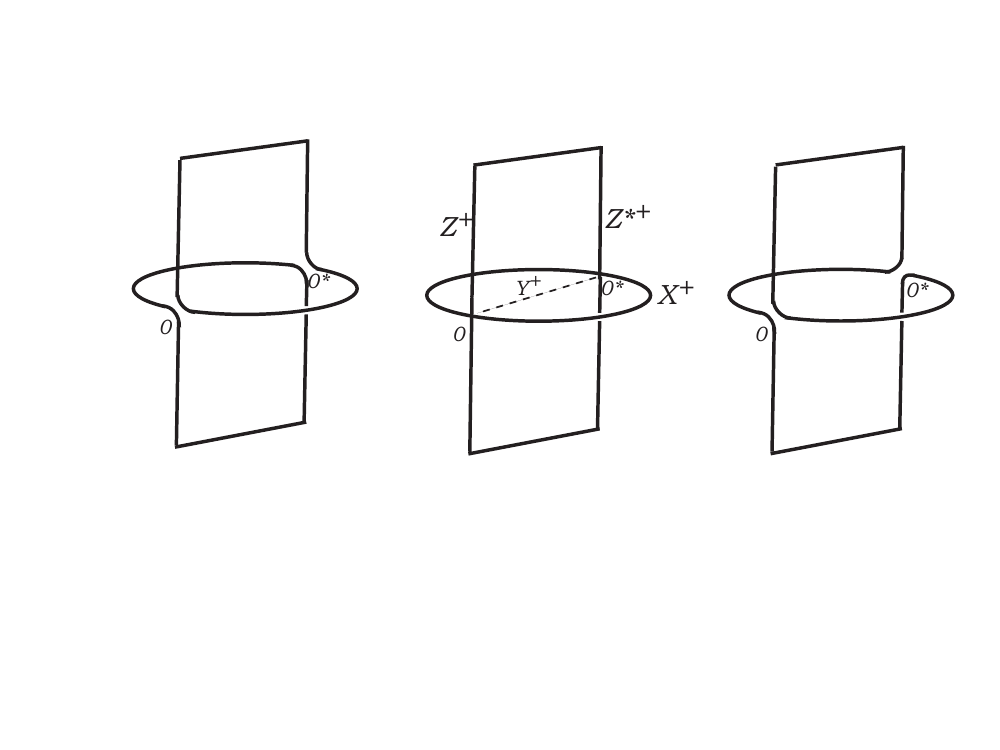}
   }
   \vspace{0.60 in}
 \begin{center}
 \parbox{5.1in}{
\caption{\label{rounding-corners-figure}  {\bf Rounding the corners  of $\Gamma$.}
{\bf Center:} The boundary curve $\Gamma$ as illustrated in Figure~\ref{GammaFigure}.
{\bf Left and Right:} Desingularizations of $\Gamma$. The corners at $O$ and $O^*$ are removed, following the conditions $(1)$ and $(2)$ of \ref{roundings-subsection}. In both cases we have desingularized near $O$ by joining $X^+$ to $Z^+$ and
$X^-$ to $Z^-$. In the language of Definition~\ref{PosNegRounding}, both desingularizations are {\em positive at $O$}.
On the left, the rounding is also positive  at $O^*$. On the right, the rounding is
{\em negative at $O^*$}. Note that when the signs of the rounding agree at $O$ and $O^*$, as they do on the left, 
the rounded curve  has two components; 
when the signs are different, as on the right, the rounded curve is  connected.
 }
 }
 \end{center}
 \endfig

\begin{definition}\label{Rounding1} Suppose $\Gamma(t)\subset H$ is a family of smooth embedded
$\rho _Y$-invariant curves created from $\Gamma$ according to the
recipe above. Suppose  we do this in such a way that that  for each
corner $q$  of $\Gamma$, the curve
\begin{equation}\label{expression}
       (1/t) ( \Gamma (t) - q)
\end{equation}
converges smoothly to a smooth, embedded planar curve $\Gamma'$ as $t\to 0$.
Then we say that the family $\Gamma (t)$ is a {\em rounding of}
$\Gamma$.   
\end{definition}

\begin{remark}\label{meaning-of-translation-remark} 
Since we are working in $\sS^2\times\RR$
 with some Riemannian metric, 
 it may not be immediately obvious what we mean by 
  translation and  by scaling in definition~\ref{Rounding1}. 
 However, there are various ways to make sense of it.  For example, by the Nash embedding
 theorem, we can regard $\sS^2\times\RR$ with the given Riemannian metric as 
 embedded isometrically 
 in some Euclidean space.  In that Euclidean space, the expression~\eqref{expression}
  is well defined, 
and its limit as $t\to 0$ lies in the $3$-dimensional tangent space (at $q$) to $\sS^2\times\RR$,
which is of course linearly isometric to $\RR^3$.
\end{remark}

\begin{remark}
In definition~\ref{Rounding1}, note that if the corner $q$ is $O$ or $O^*$, then $\Gamma'$ consists
of  two components, and $\Gamma'$ coincides with a pair of perpendicular lines outside a disk of radius $1$ about
the intersection of those lines.   In this case, $\Gamma'$ is the boundary of two regions in the plane: one
region is connected, and the other region (the complement of the connected region) consists of two connected components. 
 We refer to each of these regions as a {\bf rounded quadrant pair}.
 If $q$ is a corner other than $O$ or $O^*$, then $\Gamma'$ consists
of a single curve.  In this case, $\Gamma'$ bounds a planar region which, outside of a disk, coincides with a quadrant
of the plane.  We call such a region a {\bf rounded quadrant}.
\end{remark}

\stepcounter{theorem}
\addtocontents{toc}{\SkipTocEntry}
\subsection{The existence of bridged approximations to \texorpdfstring{$S$}{Lg}}
We  will assume until further notice that $\Gamma\subset H$ bounds an embedded
minimal $Y$-surface $S$ in $N=\overline{H^+}\cap\{|z|\le h\}$. As in the previous section we
define  $M= \overline {S\cup\rho _Z S}$. 
For $p\in \sS^2\times\RR$, let $\pi(p)=\pi_M(p)$ be the point in $M$ closest to $p$,
provided that point is unique.   Thus the domain of $\pi$ is the set of all points in $\sS^2\times\RR$
such that there is a unique nearest point in $M$.
Since $M$ is a smooth embedded 
manifold-with-boundary, the domain of $\pi$ contains $M$ in its interior.

Consider a rounding 
 $\Gamma (t)$ of $\Gamma$ with $t\in [0,t_0]$.
 By replacing $t_0$ by a smaller value, we may assume that for all $t\in [0,t_0]$,
 the curve $\Gamma(t)$ is in the interior of the domain of $\pi$ and
 $\pi(\Gamma(t))$ is a smooth embedded curve in $M$.  It follows that $\Gamma(t)$ is the 
 normal graph of a function
 \[
     \phi_t: \pi(\Gamma(t))\to \RR.
 \]
(Here normal means normal to $M$.)
We let $\Omega(t)$ be the domain in $M$ bounded by $\pi(\Gamma(t))$.

\begin{remark}\label{where-is-O-remark}
Suppose that $S$ is positive at $O$, i.e., that it is tangent to the positive quadrants of $H$
(namely the quadrant bounded by $X^+$ and $Z^+$ and the quadrant bounded by $X^-$ and $Z^-$.)
Note that $O$ is in $\Omega(t)$ if and only if
   $\Gamma(t)\cap \BB(O,t)$ lies in the {\em negative} quadrants of $H$, or, equivalently, 
 if and only if $\Gamma(t)\cap \BB(O,t)$ connects $Z^+$ to $X^-$ and $Z^-$ to $X^+$. 
 See figure~\ref{signs-figure}.
\end{remark}

\begin{theorem}\label{bridged-approximations-theorem}
There exists a $\tau>0$ and a smooth one-parameter family $t\in (0,\tau]\mapsto f_t$ of functions
\[  
    f_t: \Omega(t)\to \RR
\]
with the following properties:
\begin{enumerate}
\item The normal graph $S(t)$ of $f_t$ is a $Y$-nongenerate,
minimal embedded $Y$-surface with boundary $\Gamma(t)$,
\item $\|f_t\|_0 + \|Df_t\|_0\to 0$ as $t\to 0$, 
\item $S(t)$ converges smoothly to $S$ as $t\to 0$ except possibly at the corners of $S$,
\item\label{in-H^+-item} $S(t)$ lies in $\overline{H^+}$.
\end{enumerate}
\end{theorem}

Later (see theorem~\ref{strong-uniqueness-theorem}) we will prove that for small $t$, the surfaces $S(t)$ have a very strong uniqueness property.
In particular, given $S$, the rounding $t\mapsto \Gamma(t)$, and any sufficiently small if $\tau>0$,  
there is a unique family $t\in(0,\tau]\to S(t)$
having the indicated properties.

\begin{remark}\label{in-H^+-remark}
Assertion~\eqref{in-H^+-item} of the theorem follows easily from the preceding assertions, provided
we replace $\tau$ by a suitable smaller number.
To see this, 
note by the smooth convergence $S(t)\to S$ away from corners, each point of $S(t)\cap H^-$  must 
lie within distance $\eps_n$ of the corners of $S$, where $\eps_n\to 0$.
By the implicit function theorem,
each corner $q$ of $S$ has a neighborhood $U\subset \sS\times[-h,h]$
that is foliated by minimal surfaces, one of which is $\overline{M}\cap U$.  For $t$ sufficiently small, the set of
points of $S(t)\cap H^-$ that are near $q$ will be contained entirely in $U$, which violates the maximum
principle unless $S(t)\cap H^-$ is empty.   
\end{remark}

{\bf Idea of the proof of theorem~\ref{bridged-approximations-theorem}}. 
(The details will take up the rest of section~\ref{rounding-section}.)
   The rounding is a one-parameter family of curves $\Gamma(t)$.
We extend the one-parameter family to a two-parameter family $\Gamma(t,s)$ (with $0\le s\le1$)
in such a way 
that $\Gamma(t,1)=\Gamma(t)$ and $\Gamma(t,0)=\pi(\Gamma(t))$.  Now $\Gamma(t,0)$
trivially bounds a minimal $Y$-surface that is a normal graph over $\Omega(t)$, namely $\Omega(t)$
itself (which is the normal graph of the zero function).  We then use the implicit function
theorem to get existence for all $(t,s)$ with $t$ sufficiently small of a minimal embedded $Y$-surface
$S(t,s)$ with boundary $\Gamma(t,s)$.  Then $t\mapsto S(t,1)$ will be the desired one-parameter family
of surfaces.

  \begin{definition}\label{Gamma(s,t)-definition}
For $0\leq t< t_0$, each $\Gamma(t)$  is the normal graph over $\pi(\Gamma (t))$ of a function $\phi _t : \pi(\Gamma (t))\rightarrow \RR$.  For  $0\leq s\leq 1$, define
 \begin{equation}
 \Gamma (t,s): =\,\,\graph \,\, s\phi _t.
  \end{equation}
  
  Note that $\pi (\Gamma (t,s)) =\Gamma (t,0)$.
  \end{definition}

\begin{proposition}\label{two-parameter-proposition}
There is a $\tau>0$ and a smooth two-parameter family
\[
   (t,s)\in (0,\tau]\times[0,1]\mapsto S(t,s)
\]
of\, $Y$-nondegenerate, minimal embedded $Y$-surfaces such that each $S(t,s)$ has boundary $\Gamma(t,s)$
and is the normal graph of a function $f_{t,s}:\Omega(t)\to \RR$  such that
\[
   \|f_{t,s}\|_0 + \|Df_{t,s}\|_0 \to 0
\]
as $t\to 0$.  The convergence $f_{t,s}\to 0$ as $t\to 0$ is smooth away from the corners of $S$.
\end{proposition}

Theorem~\ref{bridged-approximations-theorem} follows from proposition~\ref{two-parameter-proposition} by setting $S(t):=S(t,1)$.
(See remark~\ref{in-H^+-remark}.)

\begin{proof}[Proof of proposition~\ref{two-parameter-proposition}]
Fix a $\eta>0$ and a $\tau>0$ and consider the following subsets of the domain $D:=(0,\tau]\times[0,1]$:
\begin{enumerate}
\item the relatively closed set $A$ of all $(t,s)\in D$ such that $\Gamma(t,s)$ bounds a minimal embedded $Y$-surface
 that is the normal graph of a function from $\Omega(t)\to \RR$ with Lipschitz constant $\le \eta$.
\item the subset $B$ of $A$ consisting
    of all $(t,s)\in D$ such that $\Gamma(t,s)$ bounds a minimal embedded $Y$-surface
 that is $Y$-nondegenerate and that is the normal graph of a function from $\Omega(t)$ to $\RR$
 with Lipschitz constant $< \eta$.
\item the subset $C$ of $A$ consisting of all 
   $(t,s)\in D$ such that there is exactly one function whose Lipschitz constant is $\le \eta$
  and whose normal graph is a minimal embedded $Y$-surface with boundary $\Gamma(t,s)$.
\end{enumerate}
By proposition~\ref{rounding-sequence-proposition} below, 
we can choose $\eta$ and $\tau$ so that these three sets are equal: $A=B=C$.
Clearly the set $A$ is a relatively closed subset of $(0,\tau]\times[0,1]$.  Also, $A$ is nonempty
since it contains $(0,\tau]\times\{0\}$. (This is because $\Gamma(t,0)$ is the boundary of the minimal
$Y$-surface $\Omega(t)$, which is the normal graph of the zero function on $\Omega(t)$).
By the implicit function theorem, the set $B$ is a relatively open subset of $(0,\tau]\times[0,1]$.

Since $A=B=C$ is nonempty and since it is both relatively closed and relatively open in $(0,\tau]\times [0,1]$,
we must have
\[
   A=B=C=(0,\tau]\times [0,1].
\]
For each $(t,s)\in (0,\tau]\times [0,1]=C$, let $f_{t,s}:\Omega(t)\to \RR$  be the unique function with Lipschitz
constant $\le \eta$ whose normal graph is a minimal embedded $Y$-surface $S(t,s)$ with boundary $\Gamma(t,s)$.
Since $B=C$, in fact $f_{t,s}$ has Lipschitz constant $<\eta$ and $S(s,t)$ is $Y$-nondegenerate.
By the $Y$-nondegeneracy and the implicit function theorem, $S(t,s)$ depends smoothly on $(t,s)$.
Also, 
\begin{equation}\label{C^1-convergence-to-0}
    \|f_{t,s}\|_0 + \|Df_{t,s}\|_0 \to 0
\end{equation}
as $t\to 0$ by proposition~\ref{rounding-sequence-proposition} below.  
Finally, the smooth convergence $S(t,s)\to S$ away from corners 
follows from~\eqref{C^1-convergence-to-0} by standard elliptic PDE.
\end{proof}

\begin{proposition}\label{rounding-sequence-proposition}
There is an $\eta>0$ with the following property.
Suppose $S_n$ is a sequence of minimal embedded $Y$-surfaces with $\partial S_n=\Gamma(t_n,s_n)$
where $t_n\to 0$ and $s_n\in [0,1]$.  Suppose also that each $S_n$ is the normal
graph of a function 
\[
   f_n: \Omega(t_n)\to\RR
\]
with Lipschitz constant $\le \eta$.  
Then 
\begin{enumerate}
\item\label{C^1-item} $\|f_n\|_0 + \|Df_n\|_0\to 0$. (In particular, $Lip(f_n)<\eta$ for all sufficiently large $n$.)
\item\label{nondegenerate-item} $S_n$ is $Y$-nondegenerate for all sufficiently large $n$,
\item\label{unique-function-item} If $g_n$ is a function with Lipschitz constant $\le \eta$
  and if the graph of $g_n$ is a minimal embedded $Y$-surface, then $g_n=f_n$ for all sufficiently large $n$.
\item\label{unique-surface-item} 
If $\Sigma_n$ is a sequence of minimal embedded $Y$-surfaces such that $\partial \Sigma_n = \partial S_n$
and such that $\Sigma_n\subset V(S,\eps_n)$ where $\eps_n\to 0$, then $\Sigma_n=S_n$ for all
sufficiently large $n$.
\end{enumerate} 
\end{proposition}

\begin{proof}[Proof of proposition~\ref{rounding-sequence-proposition}.]
By lemma~\ref{TubularNeighborhoodLemma}, there is an $\eps>0$
  such that $S$ is the only embedded minimal $Y$-surface in $V(S,\eps)$
with boundary $\Gamma$.
Choose $\eta>0$ small enough that if $f:S\to \RR$ is Lipschitz with Lipschitz constant $\le \eta$ and 
if $f|\Gamma=0$, then the normal graph of $f$ lies in $V(S,\eps)$.  In particular, if the graph of $f$
is a minimal embedded $Y$-surface, then $f=0$.

Since the $f_n$ have a common Lipschitz bound $\eta$, they converge subsequentially
to a Lipschitz function $f:S\to \RR$.  By the Schauder estimates, the convergence is smooth
away from the corners of $\Gamma$, so the normal graph of $f$ is minimal.  Thus by
choice of $\eta$, $f=0$.   This proves that
\[
   \|f_n\|_0 \to 0.
\]

Let
\[
  L = \limsup \|Df_n\|_0.
\]
We must show that $L=0$. By passing to a subsequence, we can
assume that the $\limsup$ is a limit, and we can choose a sequence
of points
 $p_n\in S_n\setminus\partial M_n=S_n\setminus\Gamma(t_n,s_n)$ such that
\[
    \lim | Df_n(p_n)| = L.
\]
By passing to a further subsequence, we can assume that the $p_n$
converge to a point $q\in S$.  If $q$ is not a corner of $S$, then
$f_n\to 0$ smoothly near $q$, which implies that $L=0$.

Thus suppose $q$ is a corner point of $S$, that is, one of the corners of $\Gamma$.  Let $R_n =
\dist(p_n,q)$.
Now translate $S_n$, $\Omega_n$, $Y$, and $p_n$ by $-q$ and dilate
by $1/R_n$ to get $S (t)'$, $\Omega_n'$, $Y_n'$ and $p_n'$.  Note
that $S(n')$ is the normal graph over $\Omega_n'$ of a function
$f_n'$ where the $\|Df_n'\|_0$ are bounded (independently of $n$).

By passing to a subsequence, we may assume that the $\Omega_n'$
converges to a planar region $\Omega'$, which must be one of the
following:
\begin{enumerate}
\item A quadrant 
\item a rounded quadrant. 
\item a quadrant pair.
\item a rounded quadrant pair. 
\item an entire plane.
\end{enumerate}
(If $q$ is $O$ or $O^*$, then (3), (4), and (5) occurs according
to whether $t_n/R_n$ tends to $0$, to a finite nonzero limit, or
to infinity.  If $q$ is one of the other corners, then (1) or (2)
occurs according to whether $t_n/R_n$ tends to $0$ or not.)  We
may also assume that the $f_n'$ converge to a Lipschitz function
$f:\Omega'\to \RR$ and that the convergence is smooth away from
the origin. Furthermore, there is a point $p\in \Omega'$ with
\begin{equation}\label{e:Lpoint}
  \text{$|p|=1$ and $|Df(p)|=L$.}
\end{equation}

Suppose first that $\Omega'$ is a plane, which means that $q$ is
$O$ or  $O^*$, and thus that $Y'$ is the line that intersects the
plane of $\Omega'$ orthogonally. Since $S'$ is a minimal graph
over $\Omega'$,    $S'$ must also be a plane (by Bernstein's
theorem).  Since $Y$ interesects each $S(t)$ perpendicularly, $Y'$
must intersect $S'$ perpendicularly.  Thus $S'$ is a plane
parallel to $\Omega'$, so $Df'\equiv 0$. In particular, $L=0$ as
asserted.

Thus we may suppose that $\partial \Omega'$ (which is also
$\partial S'$) is nonempty.

By Schwartz reflection, we can extend $S'$ to a surface
$S^\dagger$ such that $\partial S^\dagger$ is a compact subset of
the plane $P$ containing $\partial S'$ and such that $S^\dagger$
has only one end, which is a Lipschitz graph over that plane. Thus
the end is either planar or catenoidal.  It cannot be catenoidal
since it contains rays.  Hence the end is planar, which implies
that
\[
   \lim_{x\to\infty} f(x) = 0.
\]
But then $f\equiv 0$ by the maximum principle, so $Df\equiv 0$, and
therefore $L=0$ by \eqref{e:Lpoint}. This completes the proof that $\|Df_n\|_0\to 0$ and thus
the proof of assertion~\eqref{C^1-item}.

For the proofs of assertions~\eqref{nondegenerate-item}--\eqref{unique-surface-item}, 
it is convenient to make the following observation:

\begin{alt-claim}\label{planar-limit-domain-claim}
Suppose that $p_n\in S_n\setminus \partial S_n$ and that $\dist(p_n,\partial S_n)\to 0$.
Translate $S_n$ by $-p_n$ and dilate by $1/\dist(p_n, \partial S_n)$
to get a surface $S_n'$.  Then a subsequence of the $S_n'$ converges to one of the following
planar regions:
\begin{enumerate}[\upshape $\bullet$]
 \item a quadrant,
  \item a rounded quadrant,
\item a quadrant pair, 
 \item a rounded quadrant pair, or
 \item a halfplane.
\end{enumerate}
\end{alt-claim}

The claim follows immediately from the definitions (and the fact that $\|Df_n\|\to 0$) so we omit the proof.

Next we show assertion~\eqref{nondegenerate-item} of proposition~\ref{rounding-sequence-proposition}: 
 that $S_n$ is $Y$-nondegenerate for all sufficiently large $n$.
In fact, we prove somewhat more:

\begin{alt-claim}\label{eigenfunction-claim}
Suppose $u_n$ is an eigenfunction of the Jacobi operator on $S_n$ with eigenvalue $\lambda_n$,
normalized so that 
\[
    \|u_n\|_0 = \max  |u_n(\cdot)| = \max u_n(\cdot) = 1.
\]
Suppose also that the $\lambda_n$ are bounded.  Then (after passing to a subsequence)
the $S_n$ converge smoothly on compact sets to an eigenfunction $u$ on $S$ with
eigenvalue $\lambda=\lim_n\lambda_n$.
\end{alt-claim}

(With slightly more work, one could prove that for every $k$, the $k$th eigenvalue 
 of the jacobi operator on $S_n$ converges to the $k$th eigenvalue
of the jacobi operator on $S$.  However, we do not need that result.)

\begin{proof}
By passing to a subsequence, we can assume that 
the $\lambda_n$ converge to a limit $\lambda$, 
and that the $u_n$ converge smoothly
away from the corners of $S$ to a solution of
\[
    Ju = -\lambda u
\]
where $J$ is the jacobi operator on $S$. To prove the claim, it suffices to show
that $u$ does not vanish everywhere, and that $u$ extends continuously to the corners
of $S$.

Since $u$ is bounded, that $u$ extends continuously to the corners is a standard
removal-of-singularities result.
(One way to see it is as follows.  Extend $u$ by reflection to
the the smooth manifold-with-boundary $M= \overline{S\cup \rho_ZS}$.  Now $u$
solves $\Delta u = \phi u$ for a certain smooth function $\phi$ on $M$.  Let $v$ be the solution
of $\Delta v=\phi u$ on $M$ with $v|\partial M=0$ given by the Poisson formula.  Then $v$
is continuous on $M$ and smooth away from a finite set (the corners of $\Gamma$).
Away from the corners of $M$,  $u-v$ is a bounded harmonic function that vanishes
on $\partial M$.  But isolated singularities of bounded harmonic functions are removable,
so $u-v\equiv 0$.)

To prove that $u$ does not vanish everywhere, let $p_n$ be a point at which $u_n$
attains its maximum:
\[
   u_n(p_n) = 1 = \max_{S_n}|u_n(\cdot)|.
\]
By passing to a subsequence, we can assume that the $p_n$ converge to a point 
 $p\in \overline{S}$.
 We assert that $p\notin \partial S$.  For suppose $p\in\partial S$.
Translate\footnote{See Remark~\ref{meaning-of-translation-remark}.}
  $S_n$  by $-p_n$ and dilate by 
\[
  c_n := \frac1{\dist(p_n, \partial S_n)}= \frac1{\dist(p_n, \Gamma (t_n,s_n))}
\]
to get  $S_n'$.
Let $u_n'$ be the eigenfunction on $S_n'$ corresponding to $u_n$.
Note that $u_n'$ has eigenvalue $\lambda_n/c_n^2$.

We may assume (after passing to a subsequence) that the $S_n'$ converge to one of the 
planar regions $S'$ listed in lemma~\ref{planar-limit-domain-claim}. 
The convergence $S_n'\to S_n$ is smooth except possibly at the corner (if there
is one) of $S'$. 

By the smooth convergence of $S_n'$ to $S'$, the $u_n'$ converge
subsequentially to a jacobi field $u'$ on $S'$ that is smooth except
possibly at the corner (if there is one) of $S'$.   Since $S'$ is flat, $u'$ is a harmonic function.
Note that $u'(\cdot)$ attains its maximum
value of $1$ at $O$.  By the strong maximum principle for harmonic functions, $u'\equiv 1$ on the connected component of $S'\setminus \partial S'$
containing $O$.  But $u'\equiv 0$ on $\partial S'$, a contradiction.
Thus $p$ is in the interior of $S$, where the smooth convergence $u_n\to u$
implies that $u(p)=\lim u(p_n)=1$.

This completes
the proof of claim~\ref{eigenfunction-claim}
(and therefore also the proof of assertion~\eqref{nondegenerate-item} 
in proposition~\ref{rounding-sequence-proposition}.)
\end{proof}

To prove assertion~\eqref{unique-function-item} of proposition~\ref{rounding-sequence-proposition}, 
note that by assertion~\eqref{C^1-item} of the proposition applied to the $g_n$, 
\[
  \|g_n\|_0 + \|Dg_n\|\to 0.
\]
Thus if $\Sigma_n$ is the normal graph of $g_n$, then $\Sigma_n\subset V(S,\eps_n)$ for $\eps_n\to 0$.
Hence assertion~\eqref{unique-function-item} of the proposition is a special case 
of assertion~\eqref{unique-surface-item}.

Thus it remain only to prove assertion~\eqref{unique-surface-item}.  Suppose it is false.  
Then (after passing to a subsequence) there exist embedded minimal $Y$-surfaces
$\Sigma_n\ne S_n$ such that $\partial \Sigma_n=\partial S_n$ and such that
\[
  \text{$ \Sigma_n\subset \overline{V(S,\eps_n)}$ with $\eps_n\to 0$}. \tag{*}
\]
Now~\thetag{*} implies, by the extension of 
Allard's boundary regularity theorem in \cite{white-controlling-area},
that the $\Sigma_n$ converge smoothly to $S$ away from the corners of $S$.
(We apply theorem~6.1 of \cite{white-controlling-area} in 
the ambient space obtained by removing the 
corners of $S$ from $\sS^2\times \RR$.)

  Choose a point $q_n\in \Sigma_n$
that maximizes $\dist(\cdot, S_n)$.  Let $p_n$ be the point in $S_n$ closest to $q_n$.
Since $\partial S_n=\partial \Sigma_n$,
\begin{equation}\label{relative-distances}
    \dist(p_n,q_n) \le \dist(p_n, \partial \Sigma_n) = \dist(p_n, \partial S_n).
\end{equation}
By passing to a subsequence, we may assume that $p_n$ converges to a limit $p\in \overline{S}$.

If $p$ is not a corner of $S$,  then the smooth convergence $\Sigma_n\to S$ away
from the corners implies that there is a bounded $Y$-invariant jacobi field $u$ on
$S\setminus C$ such that $u$ vanishes on $\partial S\setminus C=\Gamma \setminus C$
and such that
\[
   \max |u(\cdot)| = u(p) = 1.
\]
By standard removal of singularities (see the second paragraph
of the proof of claim~\ref{eigenfunction-claim}), 
the function $u$ extends continuously to the corners.
By hypothesis, there is no such $u$. 
Thus $p$ must be one of the corners of $S$ (i.e., one of the corners of $\Gamma$.)
  Translate $S_n$, $\Sigma_n$, and $q_n$ by
$-p_n$ and dilate by $1/\dist(p_n, \partial S_n)$ to get $S_n'$, $\Sigma_n'$, and $q_n'$.

By passing to a subsequence, we can assume that the $S_n'$ converge to one
of the planar regions $S'$ listed in the statement of lemma~\ref{planar-limit-domain-claim}.
We can also assume that the $\Sigma'_n$ converge as sets to a limit set $\Sigma'$,
and that the points $q_n'$ converge to a limit point $q'$.
Note that 
\begin{equation}\label{bounded-distance-away}
  \sup_{\Sigma'} \dist(\cdot, S') = \dist(q',S') \le \dist(O, S') = 1
\end{equation}
by~\eqref{relative-distances}.

We claim that $\Sigma'\subset S'$.  
We prove this using catenoid barriers as follows.
Let $P$ be the plane containing $S'$ 
and consider a connected component $\mathcal{C}$ of the set of catenoids
whose waists are circles in $P\setminus S'$.
(There are either one or two such components according to whether
$P\setminus S'$ has one or two components.)   Note that the ends of each such
catenoid are disjoint from $\Sigma'$ since $\Sigma'$ lies with a bounded distance of $S'$.
By the strong maximum principle,
the catenoids in $\mathcal{C}$ either all intersect $S'$ or else are all disjoint from $S'$.  
Now $\mathcal{C}$ contains catenoids whose waists are unit circles that are arbitrarily far from $S'$.
Such a catenoid (if its waist is sufficiently far from $S'$) is disjoint from $\Sigma'$.  Thus all the
catenoids in $\mathcal{C}$ are disjoint from $\Sigma'$. We have shown that if the waist of catenoid
is a circle in $P\setminus S'$, then the catenoid is disjoint from $\Sigma'$.  The union of all such catenoids
is $\RR^3\setminus S'$, so $\Sigma'\subset S'$ as claimed.

Again by the extension of Allard's boundary regularity theorem
in \cite{white-controlling-area}*{theorem~6.1},
the $\Sigma_n'$ must converge smoothly to $S'$ except at the corner (if there
is one) of $S'$.

The smooth convergence of $\Sigma_n'$ and $S_n'$ to $S'$ implies
existence of a bounded jacobi field $u'$ on $S'$ that is smooth
except at the corner, that takes its maximum value of $1$ at $O$,
and that vanishes on $\partial S'$. 
Since $S'$ is flat, $u'$ is a harmonic function.  By the maximum principle, $u'\equiv 1$
on the connected component of $S'\setminus \partial S'$ containing $O$.
But that is a contradiction since $u'$ vanishes on $\partial S'$.
\end{proof}

 \section{Additional properties 
               of the family \texorpdfstring{$t\mapsto S(t)$}{Lg}}\label{additional-properties-section}
 
 We now prove that the surfaces $S(t)$ of theorem~\ref{bridged-approximations-theorem} 
 have a strong uniqueness property for small $t$:

\begin{theorem}\label{strong-uniqueness-theorem}
Let $t\in (0,\tau]\mapsto S(t)$ be the one-parameter family of minimal $Y$-surfaces given by 
theorem~\ref{bridged-approximations-theorem}.
For every sufficiently small $\eps>0$, there is a $\tau'>0$ with the following property.
For every $t\in (0,\tau']$, the surface $S(t)$ lies in $V(S,\eps)$ 
(the small neighborhood of $S$ defined in section~\ref{rounding-section}) 
and is the unique minimal embedded 
 $Y$-surface in $V(S,\eps)$ with boundary $\Gamma(t)$.
 \end{theorem}
 
\begin{proof}
Suppose the theorem is false.  Then there is a sequence of $\eps_n\to 0$
such that, for each $n$, either
\begin{enumerate}
\item there are arbitrarily large $t$ for which $S(t)$ is not contained in $V(S,\eps_n)$, or
\item there is a $t_n$ for which $S(t_n)$ is contained in $V(S,\eps_n)$ but such that
 $V(S,\eps_n)$ contains a second embedded minimal $Y$-surface $\Sigma_n$ with boundary $\Gamma(t_n)$.
 \end{enumerate}
 The first is impossible since $S(t)\to S$ as $t\to 0$.
 Thus the second holds for each $n$.  But (2) contradicts assertion~\eqref{unique-surface-item} of proposition~\ref{rounding-sequence-proposition}. 
\end{proof}

According to theorem~\ref{bridged-approximations-theorem}, 
for each embedded minimal $Y$-surface $S$ bounded by $\Gamma$,
we get a family of minimal surfaces $t\mapsto S(t)$ with $\partial S(t)=\Gamma(t)$.   The following theorem
says, roughly speaking, that as $t\to 0$,  those surfaces account for all
minimal embedded $Y$-surfaces bounded by $\Gamma(t)$.

\begin{theorem}\label{all-accounted-for-theorem}
Let $t\mapsto \Gamma(t)$ be a rounding of $\Gamma$.  Let $S_n$ be a sequence of embedded minimal $Y$-surfaces
in $H^+\cap \{|z|\le h\}$ such that $\partial S_n = \Gamma(t_n)$ where $t_n\to 0$.  Suppose the number of points in
$S_n\cap Y^+$ is bounded independent of $n$.  Then, after passing to a subsequence, the $S_n$ converge
to a smooth minimal embedded $Y$-surface $S$ bounded by $\Gamma$, and $S_n =S(t_n)$ for all sufficiently large $n$,
where $t\mapsto S(t)$ is the one-parameter family given by theorem~\ref{bridged-approximations-theorem}.
\end{theorem}

\begin{proof}
The areas of the $S_n$ are uniformly bounded by hypothesis 
on the Riemannian metric on $\sS^2\times\RR$:
see~($3''$) in remark~\ref{isoperimetric-equivalence-remark}. 
Using  the Gauss-Bonnet theorem, the minimality of the $S_n$, 
and the fact that the sectional curvatures 
of $\sS^2\times\RR$ are bounded, it follows that 
\[
  \int_{S_n} \beta(S_n, \cdot)\,dA
\]
is uniformly bounded, where $\beta(S_n,x)$ is the square of the norm of the second fundamental form of $S_n$ at $x$.
It follows (see~\cite{white-curvature-estimates}*{theorem~3}) that after passing to a subsequence, 
the $S_n$ converge smoothly
(away from the corners of $\Gamma$)
 to a minimal embedded $Y$-surface $S$ with boundary $\Gamma$. 
By the uniqueness theorem~\ref{strong-uniqueness-theorem}, $S_n=S(t_n)$ for all sufficiently large $n$.
\end{proof}



\section{Counting the number of points in \texorpdfstring{$Y\cap S(t)$}{Lg}}\label{counting-section}

Consider a rounding $t\rightarrow\Gamma(t)$ of a boundary curve $\Gamma$, as specified in 
definition~\ref{Rounding1}. 
There are two qualitatively different ways to do the rounding at the crossings $O$ and $O^*$. We describe
what can happen at $O$ (the same description holds at $O^*$):
\begin{enumerate}
\item Near $O$, each  $\Gamma(t)$ connects points of $Z^+$ to points
of $X^+$ (and therefore points of $Z^-$ to points of $X^-$), or
\item the curve $\Gamma(t)$ connects points of $Z^+$ to points of $X^-$
(and therefore points of $Z^-$ to points of $X^+$.)
\end{enumerate}
\begin{definition} \label{PosNegRounding}
 In case (1),  the rounding  $t\rightarrow\Gamma(t)$ is {\em positive} at $O$.
In case (2),   the rounding $t\rightarrow\Gamma(t)$ is {\em negative} at $O$. 
Similar statements hold at $O^*$.
\end{definition}
In what follows, we will use the notation $\|A\|$ to denote the number of elements in a finite set $A$.

\begin{proposition}\label{disagreements-proposition}
Let $S$ be an open minimal embedded $Y$-surface in 
$N:=\overline{H^+} \cap \{|z|\le h\}$ bounded by $\Gamma$.
Let $t\mapsto S(t)$ be the family given by theorem~ \ref{bridged-approximations-theorem}, and suppose $S\cap Y$ has
exactly $n$ points.  Then 
\[
     \|S(t)\cap Y\|  =  \|S\cap Y\|  + \delta(S, \Gamma(t))
\]
where $\delta(S,\Gamma(t))$ is $0$, $1$, or $2$ 
according to whether the signs of $S$ and $\Gamma(t)$ agree at both $O$ and $O^*$, 
at one but not both of $O$ and $O^*$, or at neither $O$ nor $O^*$.
(In other words, $\delta(S,\Gamma(t))$ is the number of sign disagreements of $S$ and $\Gamma(t)$.
 See Figure~\ref{signs-figure}.)
\end{proposition}

\begin{proof}
Recall that $S(t)$ is normal graph over $\Omega(t)$, the region in $M= \overline{S\cup \rho_ZS}$ bounded by
the image of $\Gamma(t)$ under the nearest point projection from a neighborhood of $M$ to $M$.
It follow immediately that 
\[
   \| Y\cap S(t)\| =  \|Y\cap \Omega(t)\|.
\]
Note that $Y\cap \Omega(t)$ consists of $Y\cap S$ 
together with one or both of  the points $O$ and $O^*$.  
(The points $O$ and $O^*$ in $\partial S=\Gamma$ do not belong to $S$ because $S$ is open.)
Recall also (see remark~\ref{where-is-O-remark}) that $O\in \Omega(t)$ if and only if $S$ and $\Gamma(t)$
have the same sign at $O$.  Likewise, $O^*\in \Omega(t)$ if and only if $S$ and $\Gamma(t)$ have
the same sign at $O^*$.  The result follows immediately.
\end{proof}

\begin{definition} 
Let $\Mm(\Gamma)$ be the set of all open, minimal embedded $Y$-surfaces $S\subset N$
  such that $\partial S=\Gamma$.
  (Here $\Gamma=\Gamma_C$ is the curve in the statement of theorem~\ref{main-bumpy-theorem}.)

Let $\Mm^s(\Gamma,n)$ be the set of surfaces $S$ in $\Mm(\Gamma)$ 
  such that $S\cap Y=n$ and such that $S$ has sign $s$ at $O$.

 If $\Gamma'$ is a smooth, $\rho_Y$-invariant curve (e.g., one of the rounded curves $\Gamma(t)$)
 in $\overline{H^+}$ such 
 that $\Gamma'/\rho_Y$ has
  exactly one component, we let $\Mm(\Gamma',n)$ be the set of 
  embedded minimal $Y$-surfaces $S$ in 
    $\overline{H^+}$ such that $\partial S=\Gamma'$ 
    and such that $S\cap Y$ has exactly $n$ points.
\end{definition}

\begin{proposition}\label{double-positive-rounding-proposition}
Suppose the rounding $t\mapsto \Gamma(t)$ is positive at $O$ and at $O^*$.
\begin{enumerate}[\upshape (1)]
\item If $n$ is even and $S\in \Mm^+(\Gamma,n)$, then $S(t)\in \Mm(\Gamma(t),n)$.
\item If $n$ is odd and $S\in \Mm^s(\Gamma,n)$, then $S(t)\in \Mm(\Gamma(t),n+1)$.
\item If $n$ is even and $S\in \Mm^-(\Gamma,n)$, then $S(t)\in \Mm(\Gamma(t),n+2)$.
\end{enumerate}
\end{proposition}

\begin{remark}\label{sign-switch-remark}
Of course, the statement remains true if we switch all the signs.
\end{remark}

\begin{proof}
If $n$ is odd and  $S\in \Mm^s(\Gamma,n)$, then 
 $S$ has different signs at $O$ and $O^*$ by lemma~\ref{parity-sign-lemma}, 
and thus $\delta(S,\Gamma(t))=1$.

Now suppose that $n$ is even and that $S\in \Mm^s(\Gamma,n)$.  
Then by lemma~\ref{parity-sign-lemma}, the surface $S$ has the same
sign at $O^*$ as at $O$, namely $s$.  Thus $\delta(S,\Gamma)$ is $0$ if $s=+$ and is
$2$ if $s=-$. Proposition~\ref{double-positive-rounding-proposition} now follows
immediately from proposition~\ref{disagreements-proposition}.
\end{proof}

\begin{proposition}\label{mixed-sign-rounding-proposition}
Suppose the rounding $t\mapsto \Gamma(t)$ has sign $s$ at $O$ and $-s$ at $O^*$.
\begin{enumerate}[\upshape (1)]
\item If $n$ is odd and $S\in\Mm^s(\Gamma,n)$, then $S(t)\in \Mm(\Gamma(t),n)$.
\item If $n$ is even and $S$ is in $\Mm^+(\Gamma,n)$ or $\Mm^-(\Gamma,n)$, then
 $S(t)\in \Mm(\Gamma(t),n+1)$.
\item If $n$ is odd and $S\in \Mm^{-s}(\Gamma,n)$, then $S(t)\in \Mm(\Gamma(t),n+2)$.
\end{enumerate}
\end{proposition}

The proof is almost identical to the proof of proposition~\ref{double-positive-rounding-proposition}.

\begin{theorem}\label{main-count-theorem}
For every nonnegative integer $n$ and for each sign $s$, the set $\Mm^s(\Gamma,n)$ has
an odd number of surfaces.
\end{theorem}


\begfig
\centerline{
         \includegraphics[width=3.0in]{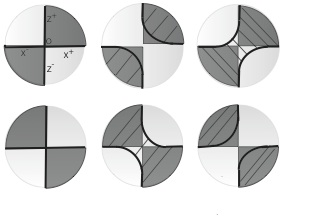}
}
 \caption{\label{signs-figure} 
  {\bf The sign of $S$ and $\Gamma (t)$ at $O$.}\newline 
   The behavior near $O$ of a surface $S\subset H^+$ with boundary $\Gamma$.\newline  {\bf First Column:}
  The surface $S$, here illustrated by
 the darker shading,   is tangent at
 $O$ to either the positive quadrants of $H$ (as illustrated on
 top) or the negative quadrants (on the bottom).  In the sense of section~\ref{sign-section},
  $S$ is positive at $O$ in the top illustration and negative  in the bottom illustration.  \newline {\bf Second column:} A curve $\Gamma (t)$ in a  positive rounding $t\rightarrow \Gamma(t)$ of $\Gamma$.   The  striped regions lie in the projections
 $\Omega (t)$ defined in theorem~\ref{bridged-approximations-theorem}. Note that on the top $O\not\in \Omega (t) $.
  On the bottom, $O\in \Omega (t)$.   
 \newline  {\bf Third Column:} A curve $\Gamma (t)$ in a negative rounding of $\Gamma$. The striped regions lie in
 $\Omega (t)$.  Note that on top we have  $O\in \Omega (t)$.
  On the bottom, $O\not \in \Omega(t)$. 
  }
\endfig


\begin{remark}\label{periodic-case-done-remark}
Note that theorem~\ref{main-count-theorem} is the same as theorem~\ref{main-bumpy-theorem},
because ever since section~\ref{bumpy-section}, we have been working with an arbitrary Riemannian metric
on $\sS^2\times\RR$ that satisfies the hypotheses of theorem~\ref{main-bumpy-theorem}.
By proposition~\ref{bumpy-suffices-proposition}, 
theorem~\ref{main-bumpy-theorem} implies 
theorem~\ref{special-existence-theorem-tilted},
 which by proposition~\ref{tilted-suffices-proposition}
  implies theorem~\ref{special-existence-theorem}, 
  which by proposition~\ref{reduction-to-H^+-proposition} 
  implies theorem~\ref{main-S2xR-theorem} for $h<\infty$.  
  Thus in proving theorem~\ref{main-count-theorem}, we complete the proof of the periodic case of 
  theorem~\ref{main-S2xR-theorem}.
\end{remark}

\begin{proof}
Let $f^s(n)$ denote the mod $2$ number of surfaces in $\Mm^s(\Gamma,n)$.
Note that $f^s(n)=0$ for $n<0$ since $Y\cap S$ cannot have a negative number of points.
The theorem asserts that $f^s(n)=1$ for every $n\ge 0$.

We prove the theorem by induction.  Thus we let $n$ be a nonnegative integer,
we assume that $f^s(k)=1$ for all nonnegative $k< n$ and $s=\pm$, and we must prove
that $f^s(n)=1$.

{\bf Case 1}: $n$ is even and $s$ is $+$.

To prove that $f^+(n)=1$, we choose a rounding
 $t\mapsto \Gamma(t)$ that is positive at both $O$ and $O^*$.

We choose $\tau$ sufficiently small that for every $S\in\Mm(\Gamma)$ with $\|Y\cap S\|\le n$,
the family $t\mapsto S(t)$ is defined for all $t\in (0,\tau]$.   We may also choose $\tau$ small
enough that if $S$ and $S'$ are two distinct such surfaces, then $S(t)\ne S'(t)$ for $t\le\tau$.
(This is possible since $S(t)\to S$ and $S'(t)\to S'$ as $t\to 0$.)

By theorem~\ref{all-accounted-for-theorem}, we can fix a $t$ sufficiently small that
 for each surface $\Sigma\in \Mm(\Gamma(t),n)$, there is a surface 
$S=S_\Sigma\in \Mm(\Gamma)$
such that $\Sigma=S_\Sigma(t)$.  Since all such $S(t)$ are $\rho_Y$-nondegenerate,
this implies
\begin{equation}\label{suitably-bumpy}
\text{The surfaces in $\Mm(\Gamma(t), n)$ are all $\rho_Y$-nondegenerate.}
\end{equation}

By proposition~\ref{double-positive-rounding-proposition}, $S_\Sigma$ belongs to the union $U$ of
\begin{equation}\label{gang-of-four}
 \text{$\Mm^+(\Gamma,n)$,  $\Mm^+(\Gamma,(n-1))$, $\Mm^-(\Gamma,(n-1))$, and $\Mm^-(\Gamma, (n-2))$.}
\end{equation}
By the same proposition, if $S$ belongs to the union $U$, then $S(t)\in \Mm(\Gamma(t),n)$.
Thus $\Sigma\mapsto S_\Sigma$ gives a bijection from $\Mm(\Gamma(t),n)$ to $U$,
so the number of surfaces in $\Mm(\Gamma(t),n)$ is equal to the sum of the numbers of surfaces
in the four sets in~\eqref{gang-of-four}.  Reducing mod $2$ gives
\begin{equation}
  \| \Mm(\Gamma(t), n) \|_\text{mod $2$} = f^+(n) + f^+(n-1) + f^-(n-1) + f^-(n-2).
\end{equation}
By induction, $f^+(n-1)=f^-(n-1)$ (it is $0$ for $n=0$ and $1$ if $n\ge 2$), so
\begin{equation}\label{recursion-formula-even}
  \| \Mm(\Gamma(t), n) \|_\text{mod $2$} = f^+(n)  + f^-(n-2).
\end{equation}

As mentioned earlier, we have good knowledge about the mod $2$ number of minimal surfaces
bounded by suitably bumpy smooth embedded curves.  In particular, $\Gamma(t)$ is smooth
and embedded and has the bumpiness property~\eqref{suitably-bumpy}, which implies
that (see theorem~\ref{Y-degree-theorem})
\begin{equation}\label{rounded-curve-count}
\|\Mm(\Gamma(t),n)\|_\text{mod $2$}
=
\begin{cases}
1 &\text{if $n=1$ and $\Gamma(t)$ is connected}, \\
1 &\text{if $n=0$ and $\Gamma(t)$ is not connected, and} \\
0 &\text{in all other cases.}
\end{cases}
\end{equation}
Combining~\eqref{recursion-formula-even} and~\eqref{rounded-curve-count} gives
    $f^+(n)=1$.

{\bf Case 2}: $n$ is even and $s$ is $-$.  The proof is exactly like the proof of case 1, except
that we use a rounding that is negative at $O$ and at $O^*$.  (See remark~\ref{sign-switch-remark}.)

{\bf Cases 3 and 4}: $n$ is odd and $s$ is $+$ or $-$.

The proof is almost identical to the proof in the even case, except that we use
a rounding $t\mapsto \Gamma(t)$ that has sign $s$ at $O$ and $-s$ at $O^*$.
In this case we still get a bijection $\Sigma\mapsto S_\Sigma$, but it is a bijection
from $\Mm(\Gamma(t),n)$ to the union $U$ of the sets
\begin{equation}\label{odd-gang-of-four}
 \text{$\Mm^s(\Gamma,n)$, $\Mm^+(\Gamma,n-1)$, $\Mm^-(\Gamma,n-1)$, and $\Mm^{-s}(\Gamma,n-2)$.}
\end{equation}
Thus $\Mm(\Gamma(t),n)$ and $U$ have the same number of elements mod $2$:
\[
  \|\Mm(\Gamma(t),n)\|_ \text{mod $2$} =  f^s(n) + f^+(n-1) + f^-(n-1) + f^{-s}(n-2).
\]
As in case 1, $f^+(n-1)=f^-(n-1)$ by induction, so their sum is $0$:
\[
  \|\Mm(\Gamma(t),n)\|_ \text{mod $2$} =  f^s(n) + f^{-s}(n-2).
\] 
Combining this with~\eqref{rounded-curve-count} gives
$f^s(n)=1$.
\end{proof}


\renewcommand{\thesubsection}{\thetheorem}
   
\section{Counting minimal surfaces bounded by smooth curves}\label{smooth-count-section}

In the previous section, we used certain facts about the mod $2$ numbers of minimal surfaces
bounded by smooth curves.  In this section we state those facts, and show that they apply
in our situation.   The actual result we need is theorem~\ref{Y-degree-theorem} below,
and the reader may go directly to that result.  However, we believe it may be helpful to 
first state a simpler result that has the main idea of theorem~\ref{Y-degree-theorem}:

\begin{theorem}\label{degree-theorem}
Suppose $N$ is compact, smooth, strictly mean-convex Riemannian $3$-manifold diffeomorphic to the
a ball.
Suppose also that $N$ contains no smooth, closed minimal surfaces.  
Let $\Sigma$ be any compact $2$-manifold
with boundary.   Let $\Gamma$ be a smooth embedded curve in $\partial N$, and
let $\MM(\Gamma, \Sigma)$ be the set of embedded minimal surfaces in $N$ that have boundary $\Gamma$ and
that are diffeomorphic to $\Sigma$.   Suppose all the surfaces in $\MM(\Gamma,\Sigma)$ are nondegenerate.
Then the number of those surfaces is odd if $\Sigma$ is a disk or union of disks, and is even if not.
\end{theorem}

See~\cite{hoffman-white-number}*{theorem~2.1} for the proof.

If we replace the assumption of strict mean convexity by mean convexity, then $\Gamma$ may bound
a minimal surface in $\partial N$.   In that case, theorem~\ref{degree-theorem}
 remains true provided (i) we assume that
no two adjacent components of $\partial N\setminus \Gamma$ are both minimal surfaces, 
and (ii) we count minimal surfaces
in $\partial N$ only if they are stable.
  Theorem~\ref{degree-theorem} also generalizes to the case of curves and surfaces
invariant under a finite group $G$ of symmetries of $N$.  If one of those symmetries is $180^\circ$ rotation
about a geodesic $Y$, then the theorem also generalizes to $Y$-surfaces:

\begin{theorem}\label{Y-degree-theorem}
Let $N$ be a compact region in a smooth Riemannian $3$-manifold such that $N$ is homeomorphic to the $3$-ball.
Suppose that $N$ has piecewise smooth, weakly mean-convex boundary,
and that $N$ contains no closed minimal surfaces.
Suppose also that  $N$ admits a $180^\circ$ rotational symmetry $\rho_Y$ about a geodesic $Y$.

Let $C$ be a $\rho_Y$-invariant smooth closed curve in $(\partial N)\setminus Y$ such that $C/\rho_Y$ is connected,
 and such that no two adjacent components of $(\partial N)\setminus C$ are both smooth minimal surfaces.
Let $\Mm^*(C,n)$ be the collection of $G$-invariant, 
minimal embedded $Y$-surfaces $S$ in $N$ with boundary $C$ such that
(i) $S\cap Y$ has exactly $n$ points, and
(ii) if $S\subset \partial N$, then $S$ is stable.
Suppose $C$ is $(Y,n)$-bumpy in the following sense:
all the $Y$-surfaces in  $\Mm^*(C,n)$ are $\rho_Y$-nondegenerate (i.e., have no nontrivial $\rho_Y$-invariant
jacobi fields.)
Then:
\begin{enumerate}
\item\label{union-of-disks-case}
     If $C$ has two components and $n=0$, the number of surfaces in $\Mm^*(C,n)$ is odd.
\item\label{disk-case}
     If $C$ has one component and $n=1$, the number of surfaces in $\Mm^*(C,n)$ is odd.
\item\label{more-complicated-case}
     In all other cases, the number of surfaces in $\Mm^*(C,n)$ is even.
\end{enumerate}
\end{theorem}

We remark (see corollary~\ref{$Y$-corollary})
 that in case~\eqref{disk-case}, each surface in $\Mm^*(C,n)$ is a disk, 
 in case~\eqref{union-of-disks-case}, each surface
in $\Mm^*(C,n)$ is the union of two disks, 
and in case~\eqref{more-complicated-case}, each surface is $\Mm^*(C,n)$ 
has more complicated topology (it is connected but not simply connected).

Theorem~\ref{Y-degree-theorem} is proved 
in \cite{hoffman-white-number}*{\S4.7}.

In the proof of theorem~\ref{main-count-theorem}, 
we invoked the conclusion of theorem~\ref{Y-degree-theorem}.
We now justify that. 
  Let $\Gamma(t)$ be the one of the curves formed by rounding $\Gamma$ in section~\ref{roundings-subsection}.
Note that $\Gamma(t)$ bounds a unique minimal surface $\Omega(t)$ that lies in the helicoidal portion of 
 $\partial N$, i.e, that lies in $H\cap \{|z|\le h\}$.
(The surface $\Omega(t)$ is topologically a disk, an annulus, or a pair of disks, depending 
on the signs of the rounding at $O$ and $O^*$.)
Note also that the complementary region $(\partial N)\setminus \Omega(t)$ is piecewise smooth, but
not smooth.
To apply theorem~\ref{Y-degree-theorem} as we did, we must check that:
\begin{enumerate}[\upshape (i)]
\item\label{no-closed-surface-item} $N$ contains no closed minimal surfaces.
\item\label{no-adjacent-item} No two adjacent components of $\partial N$ are smooth minimal surfaces.
\item\label{stable-item} The surface $\Omega(t)$ is strictly stable. 
    (We need this because in the proof of theorem~\ref{main-count-theorem}, we counted $\Omega(t)$,
     whereas theorem~\ref{Y-degree-theorem} tells us to count it only if it is stable.)
\end{enumerate}
Now~\eqref{no-closed-surface-item} is true by hypothesis on the Riemannian metric on $N$:
see~theorem~\ref{main-bumpy-theorem}\eqref{no-closed-minimal-surface-hypothesis}.
Also,~\eqref{no-adjacent-item} is true because (as mentioned above) the surface $(\partial N)\setminus \Omega(t)$
is piecewise-smooth but not smooth.

On the other hand,~\eqref{stable-item} need not be true in general. 
However, in the proof of theorem~\ref{main-count-theorem}, 
we were allowed to choose $t>0$ as small as we like,
and~\eqref{stable-item} is true if $t$ is sufficiently small:

\begin{lemma}\label{strictly-stable-lemma}
Let $t\mapsto \Gamma(t)\subset H$ be a rounding as in theorem~\ref{bridged-approximations-theorem}.
Then the region $\Omega(t)$ in $\partial N$ bounded by $\Gamma(t)$ is strictly stable
provided $t$ is sufficiently small.
\end{lemma}

(We remark that is a special case of a more general principle: if two strictly stable minimal
surfaces are connected by suitable thin necks, the resulting surface is also strictly stable.)

\begin{proof}
Let $\lambda(t)$ be the lowest eigenvalue of the jacobi operator on $\Omega(t)$.
Note that $\lambda(t)$ is bounded. (It is bounded below by the lowest
eigenvalue of a domain in $H$ that contains all the $\Omega(t)$ and above by the lowest
eigenvalue of a domain that is contained in all the $\Omega(t)$.)  
It follows that any subsequential limit $\lambda$ as $t\to 0$ of the $\lambda(t)$ is an eigenvalue
of the jacobi operator on $\Omega$, where $\Omega$ is the region in $H$ bounded
by $\Gamma$.
(This is a special case of claim~\ref{eigenfunction-claim}.)

By hypothesis\footnote{This is the only place where that hypothesis is used.}~\eqref{strictly-stable-hypothesis}
of theorem~\ref{main-bumpy-theorem}, $\Omega$ is strictly stable, so $\lambda>0$ and therefore
 $\lambda(t)>0$ for all sufficiently small $t>0$.
\end{proof}

 \section{General results on existence of limits}\label{general-results-section}
    
At this point, we have completed the proof of theorem~\ref{main-S2xR-theorem} in the case $h<\infty$.
That is,
we have established the existence of periodic genus-$g$ helicoids in $\sS^2(R)\times \RR$.
During that proof (in sections~\ref{bumpy-section}--\ref{smooth-count-section}), 
we considered rather general Riemannian metrics on $\sS^2(R)\times\RR$.
However, from now on we will always use the standard product metric.
In the remainder of \hyperref[part1]{part~\ONE} of the paper, 
\begin{enumerate}
\item We prove existence of nonperiodic genus-$g$ helicoids in $\sS^2(R)\times \RR$ 
by taking limits of periodic examples as the period tends to $\infty$.
\item We prove existence of
helicoid-like surfaces in $\RR^3$ by taking suitable
limits of nonperiodic examples in $\sS^2(R)\times \RR$ as $R\to\infty$.
\end{enumerate}
(We remark that one can also get periodic genus $g$-helicoids in $\RR^3$
as limits of periodic examples in $\sS^2(R)\times \RR$ as $R\to\infty$ with the
period kept fixed.)

Of course one could take the limit as sets in the Gromov-Hausdorff sense.
But to get smooth limits, one needs curvature estimates and local area bounds:
without curvature estimates, the limit need not be smooth, 
whereas with curvature estimates but without local area bounds, limits might be minimal 
laminations rather than smooth, properly embedded surfaces.

In fact, local area bounds are the key, because such bounds allow
one to use the following compactness theorem 
(which extends similar results in~\cite{choi-schoen}, \cite{anderson}, 
and~\cite{white-curvature-estimates}):


\begin{theorem}[General Compactness Theorem]\label{general-compactness-theorem}
Let $\Omega$ be an open subset of a Riemannian $3$-manifold.
Let $g_n$ be a sequence of smooth Riemannian metrics on $\Omega$ converging smoothly
to a Riemannian metric $g$.
Let $M_n\subset \Omega$ be a sequence of properly embedded surfaces 
such that $M_n$ is minimal with respect to $g_n$.
Suppose also that  the area and the genus of $M_n$ are bounded independently of $n$.
Then (after passing to a subsequence) the $M_n$ converge to a  smooth, properly embedded
$g$-minimal surface $M'$.  For each connected component $\Sigma$ of $M'$,
either
\begin{enumerate}
\item the convergence to $\Sigma$ is smooth with multiplicity one, or
\item the convergence is smooth (with some multiplicity $>1$) away from a discrete set $S$.
\end{enumerate}
In the second case, if $\Sigma$ is two-sided, then it must be stable.

Now suppose $\Omega$ is an open subset of $\RR^3$.  (The metric $g$ need not be flat.)
If $p_n\in M_n$ converges to $p\in M$, then (after passing to a further subsequence)
either
\[
    \Tan(M_n,p_n)\to \Tan(M,p)
\]
or there exists constants $c_n>0$ tending to $0$ such that the surfaces
\[
    \frac{M_n - p_n}{c_n}
\]
converge to a non flat complete embedded minimal surface $M'\subset \RR^3$ of finite total 
curvature with ends parallel to $\Tan(M,p)$.

\end{theorem}


See~\cite{white-embedded} for the proof.

When we apply theorem~\ref{general-compactness-theorem}, 
in order to get smooth convergence everywhere (and not just away from
a discrete set), we will prove that the limit surface has no stable components.   For that,
 we will use the following theorem of Fischer-Colbrie and Schoen. (See theorem~3 on page 206
and paragraph 1 on page 210 of \cite{fischer-colbrie-schoen}.)

\begin{theorem*}
Let $M$ be an orientable, complete, stable minimal surface in a complete, orientable Riemannian $3$-manifold
of nonnegative Ricci curvature.  Then $M$ is totally geodesic, and its normal bundle is Ricci flat.
(In other words, if $\nu$ is a normal vector to $M$, then $\operatorname{Ricci}(\nu,\nu)=0$.)
\end{theorem*}

\begin{corollary}\label{stable-spheres-corollary}
If $M$ is a connected, stable, properly embedded, minimal surface in $\sS^2\times \RR$, then
$M$ is a horizontal sphere.
\end{corollary}

To prove the corollary, note that since $\sS^2\times\RR$ is orientable and simply connected
and since $M$ is properly embedded, $M$ is orientable.  
Note also that if $\operatorname{Ricci}(\nu,\nu)=0$,
then $\nu$ is a vertical vector.

In section~\ref{area-bounds-section}, 
we prove the area bounds we need to get nonperiodic examples in $\sS^2\times\RR$.
In section~\ref{nonperiodic-section}, 
we prove area and curvature bounds in $\sS^2(R)\times \RR$ as $R\to\infty$.
In section~\ref{R3-section}, we get examples in $\RR^3$ by letting $R\to\infty$.


\section{Uniform Local Area Bounds 
         in \texorpdfstring{$\sS^2\times\RR$}{Lg}}\label{area-bounds-section}

Let 
$
    \theta: \overline{H^+}\setminus (Z\cup Z^*)\to \RR
$
be the natural angle function which, if we identify $\sS^2\setminus\{O^*\}$ with $\RR^2$ by stereographic
projection, is given by $\theta(x,y,z)=\arg(x+iy)$.  Note that since $\overline{H^+}$ is simply connected,
we can let $\theta$ take values in $\RR$ rather than in $\RR$ modulo $2\pi$.

\begin{proposition}\label{flux-bounds-proposition}
Suppose $H$ is a helicoid in $\sS^2\times\RR$ with axes $Z$ and $Z^*$.
Let $M$ be a minimal surface in $\overline{H^+}$ with compact, piecewise-smooth boundary,
and let
\[
   S = M \cap \{a \le z \le b\} \cap \{\alpha\le \theta \le \beta\}.
\]
Then
\begin{align*}
  \area(S)
   \le
  (b-a)\int_{(\partial M)\cap\{z>a\}}|\vv_z \cdot \nupt |\,ds  
  +  
  (\beta-\alpha)\int_{(\partial M)\cap \{\theta>\alpha\}} |\vv_\theta\cdot \nupt | \,ds
\end{align*}
where $\vv_z=\pdf{}{z}$, $\vv_\theta=\pdf{}{\theta}$, and where $\nupt$
is the unit normal to $\partial M$ that points out of $M$.
\end{proposition}

\begin{proof}
Let $u:\sS^2\times \RR$ be the function $z(\cdot)$ or the function $\theta(\cdot)$.
In the second case, $u$ is well-defined as a single-valued function only on $H^+$.
But in both cases, $\vv=\vv_u:=\pdf{}{u}$ is well-defined Killing field on all of $\sS^2\times \RR$.
(Note that $\vv_\theta\equiv 0$ on $Z\cup Z^*$.)  

Now consider the vectorfield $w(u)\vv$, where $w:\RR\to \RR$ is given by
\[
  w(u) 
  =
  \begin{cases}
   0 &\text{if $u<a$}, \\
   u-a &\text{if $a \le u \le b$, and} \\
   b-a &\text{if $b< u$.}
  \end{cases}
\]
Then\footnote{The reader may find it helpful to note that in the proof, we are 
  expressing
 $\frac{d}{dt}\area(M_t)$ in two different ways (as a surface integral and as a boundary integral), where
 $M_t$ is a one-parameter family of surfaces with $M_0=M$ and with initial velocity vectorfield $w(u)\vv$.}
\begin{align*}
   \int_M \Div_M(w\vv)\,dA
   &=
    \int_M\left( \grad_M(w(u))\cdot \vv + w(u) \Div_M\vv \right)\,dA  
   \\
   &=
    \int_M \left( w'(u) \grad_Mu \cdot \vv + 0 \right) \,dA  
    \\
    &=
    \int_{M\cap (u^{-1}[a,b]) } \nabla_Mu \cdot \vv \, dA
\end{align*}
since $\Div_M\vv\equiv 0$ (because $\vv$ is a Killing vectorfield.)

Let $\ee=\ee_u$ be a unit vectorfield in the direction of $\nabla u$.
Then $\nabla u = |\nabla u|\,\ee$ and $\vv = \pdf{}{u} = |\nabla u|^{-1}\ee$, so
\begin{align*}
 \nabla_Mu \cdot \vv
 &=
 (\nabla u)_M \cdot (\vv)_M \\
 &=
 (|\nabla u|\,\ee)_M \cdot (|\nabla u|^{-1}\ee)_M \\
 &=
 |(\ee)_M|^2  \\
 &=
 1 - (\ee\cdot \nu_M)^2
\end{align*}
where $(\cdot)_M$ denotes the component tangent to $M$ and
where $\nu_M$ is the unit normal to $M$.

Hence we have shown
\begin{equation}\label{e:StretchM}
    \int_M \Div_M(w\vv)\,dA = \int_{M \cap \{a\le u \le b \}} (1 - (\ee_u\cdot \nu_M)^2)\, dA.
\end{equation}

Since $M$ is a minimal surface, 
\[
  \Div_M(V) = \Div_M(V^{\rm tan})
\]
for any vectorfield $V$ (where $V^{\rm tan}$ is the component of $V$ tangent to $M$), so
\begin{equation}\label{e:FirstVariationM}
\begin{aligned}
   \int_M \Div_M(w\vv)\,dA
  &=
     \int_M \Div_M(w\vv)^{\rm tan}\,dA
  \\
  &=
  \int_{\partial M} (w\vv)\cdot \nupt  
  \\
  &\le
  (b-a)\int_{(\partial M) \cap \{u>a\}} \left|\vv_u\cdot \nupt \right|
 \end{aligned}
\end{equation}
Combining~\eqref{e:StretchM} and~\eqref{e:FirstVariationM} gives
\[
 \int_{M \cap \{a \le u \le b\}} (1 - (\ee_u\cdot \nu_M)^2)\, dA
 \le
 (b-a)\int_{(\partial M) \cap \{u>a\}} \left|\vv_u\cdot \nupt \right|
 \]
Adding this inequality
 for $u=z$ to the same inequality for $u=\theta$ (but with $\alpha$ and $\beta$ in place of $a$ and $b$)
 gives
\begin{equation}\label{almost-there}
\begin{aligned}
  &\int_S (2 - (\ee_z\cdot\nu_M)^2 - (\ee_\theta\cdot\nu_M)^2)\,dA
  \\
  &\quad \le
 (b-a)\int_{(\partial M)\cap\{z>a\}}|\vv_z \cdot \nupt|\,ds  \\
  &\quad\quad+  (\beta-\alpha)\int_{(\partial M)\cap \{\theta>\alpha\}} |\vv_\theta\cdot \nupt| \,ds
\end{aligned}
\end{equation}

Let $\ee_\rho$ be a unit vector orthogonal to $\ee_z$ and $\ee_\theta$.  Then
for any unit vector $\nu$, 
\[
   1 = (\ee_z\cdot\nu)^2 + (\ee_\theta\cdot\nu)^2 + (\ee_\rho\cdot \nu)^2,
\]
so the integrand in the left side of~\eqref{almost-there} 
is $\ge 1 + (\ee_\rho\cdot\nu_M)^2 \ge 1$.
\end{proof}

\newcommand{\diam}{\operatorname{diam}}
\begin{corollary}\label{flux-bounds-corollary}
Let $M$ be a compact minimal surface in $\overline{H^+}$ and let $L$ be the 
the length of $(\partial M)\setminus (Z\cup Z^*)$.
Then
\[
    \area(M\cap K)\le c_H L \diam(K)
\]
for every compact set $K$, where $\diam(K)$ is the diameter of $K$ and
where $c_H$ is a constant depending on the helicoid $H$.
\end{corollary}

The corollary follows immediately from proposition~\ref{flux-bounds-proposition} because 
    $\vv_z\cdot\nu_M=0$ and $\vv_\theta=0$ on $(\partial M)\cap (Z\cup Z^*)$.


 \section{Nonperiodic genus-\texorpdfstring{$g$}{Lg} 
              helicoids in \texorpdfstring{$\sS^2\times\RR$}{Lg}: 
              theorem~\ref{main-S2xR-theorem} for \texorpdfstring{$h=\infty$}{Lg}}
\label{nonperiodic-section}

Fix a helicoid $H$ in $\sS^2\times \RR$ with axes $Z$ and $Z^*$ and fix a genus $g$.
For each $h\in (0,\infty]$, consider the class $\Cc(h)=\Cc_g(h)$ of 
embedded, genus--$g$ minimal surfaces $M$ in $\sS^2\times [-h,h]$ such that
\begin{enumerate}
\item If $h<\infty$, then $M$ is bounded by two great circles at heights $h$ and $-h$.
   If $h=\infty$, then $M$ is properly embedded with no boundary.
\item $M\cap H \cap \{|z|<h\} = (X\cup Z\cup Z^*)\cap \{|z|<h\}$.
\item $M$ is a $Y$-surface.
\end{enumerate}

By the $h<\infty$ case of theorem~\ref{main-S2xR-theorem} (see section~\ref{construction-outline-section}),
the collection $\Cc(h)$ is nonempty for every $h<\infty$.
Here we prove the same is true for $h=\infty$:

\begin{theorem}\label{h-to-infinity-theorem}
Let $h_n$ be a sequence of positive numbers tending to infinity.
Let $M_n\in \Cc(h_n)$.
Then a subsequence of the $M_n$ converges smoothly and with multiplicity one
to a minimal surface $M\in \Cc(\infty)$.  The surface $M$ has bounded curvature,
and each of its two ends is asymptotic to a helicoid having the same pitch as $H$.
\end{theorem}

\begin{proof}[Proof of theorem~\ref{h-to-infinity-theorem}]
Note that $M_n\cap H^+$ is bounded by two vertical line segments, by the horizontal great circle $X$,
and by a pair of great semicircles at heights $h_n$ and $-h_n$.  It follows that vertical flux is uniformly
bounded.  Thus by corollary~\ref{flux-bounds-corollary},
 for any ball $B$, the area of
\[
   M_n\cap H^+ \cap B
\]
is bounded by a constant depending only on the radius of the ball.
Therefore
the areas of the $M_n$ (which are obtained from the $M_n\cap H^+$ by Schwarz reflection) are also
uniformly bounded on compact sets.
 By the compactness theorem~\ref{general-compactness-theorem},
we can (by passing to a subsequence) assume that the $M_n$ converge as sets to a smooth, properly embedded limit minimal surface $M$.   
According to~\cite{rosenberg2002}*{theorem~4.3}, every properly embedded minimal
 surface in $\sS^2\times\RR$
is connected unless it is a union of horizontal spheres.  
Since $M$ contains $Z\cup Z^*$, it is not a union of horizontal spheres, and thus it is connected.
By corollary~\ref{stable-spheres-corollary}, $M$ is unstable.
Hence by the general compactness theorem~\ref{general-compactness-theorem}, the convergence
$M_n\to M$ is smooth with multiplicity one.

Now suppose that each $M_n$ is a $Y$-surface, i.e., that
\begin{enumerate}
\item $\rho_Y$ is an orientation-preserving involution of $M_n$,
\item $M_n/\rho_Y$ is connected, and
\item Each $1$-cycle $\Gamma$ in $M_n$ is homologous (in $M_n$) to $-\rho_Y\Gamma$.
\end{enumerate}

The smooth convergence implies that $\rho_Y$ is also an orientation-preserving involution of $M$.
Since $M$ is connected, so is $M/\rho_Y$.  Also, if $\Gamma$ is a cycle in $M$,
then the smooth, multiplicity one convergence implies that $\Gamma$ is a limit of cycles $\Gamma_n$
in $M_n$.  Thus $\Gamma_n$ together with $\rho_Y\Gamma_n$ bound a region, call it $A_n$,
in $M_n$.  Note that the $\Gamma_n \cup \rho_Y\Gamma_n$ lie in a bounded region in $\sS^2\times \RR$.
Therefore so do the $A_n$ (by, for example, the maximum principle applied to the minimal surfaces $A_n$.)
Thus the $A_n$ converge to a region $A$ in $M$ with boundary $\Gamma + \rho_Y\Gamma$.
This completes the proof that $M$ is a $Y$-surface.

Recall that $Y$ intersects any $Y$-surface transversely, and the number of intersection points is equal
to twice the genus plus two.  It follows immediately from the smooth convergence (and the compactness of $Y$)
that $M$ has genus $g$.

The fact that $M\cap H=X\cup Z\cup Z^*$ follows immediately from smooth convergence together
with the corresponding property of the $M_n$.

Next we show that $M$ has bounded curvature.
Let $p_k\in M$ be a sequence of points such that the curvature of $M$ at $p_k$ tends
to the supremum.  Let $f_k$ be a screw motion such that $f_k(H)=H$ and such that $z(f(p_k))=0$.
The surfaces $f_k(M)$ have areas that are uniformly bounded on compact sets. (They inherit those
bounds from the surfaces $M_n$.)  Thus exactly as above, by passing to a subsequence, we get
smooth convergence to a limit surface.  It follows immediately that $M$ has bounded curvature.

Since $M$ is a minimal embedded surface of finite topology containing $Z\cup Z^*$, each of its
two ends is asymptotic to a helicoid by~\cite{hoffman-white-axial}.
Since $M\cap H=Z\cup Z^*$, those limiting helicoids must have the same pitch as $H$.
(If this is not clear, observe that the intersection of two helicoids with the same axes but different pitch
contains an infinite collection of equally spaced great circles.)
\end{proof}

\addtocontents{toc}{\SkipTocEntry}
\subsection{Proof of theorem~\ref{main-S2xR-theorem} for \texorpdfstring{$h=\infty$}{Lg}}
The non-periodic case of theorem~\ref{main-S2xR-theorem} follows immediately from the 
periodic case together with theorem~\ref{h-to-infinity-theorem}.
The various asserted properties of the non-periodic examples follow from the corresponding
properties of the periodic examples together with the smooth convergence in
 theorem~\ref{h-to-infinity-theorem}, except for the noncongruence properties, which are
 proved in section~\ref{noncongruence-appendix}. \qed


\section{Convergence to Helicoidal Surfaces in \texorpdfstring{$\RR^3$}{Lg}}\label{R3-section}

In the section, we study the behavior of genus-$g$ helicoidal surfaces in
 $\sS^2(R)\times\RR$ as $R\rightarrow\infty$.
 The results in this section will be used in section~\ref{Proof_of_DIVERGING-RADII-THEOREM} 
 to prove theorem~\ref{DIVERGING-RADII-THEOREM} of section~\ref{section:main-theorems}.

 We will identify $\sS^2(R)$ with $\RR^2 \cup\{\infty\}$ by stereographic projection, and therefore
  $\sS^2(R)\times \RR$ with 
  \[
     (\RR^2\cup \{\infty\})\times \RR = \RR^3 \cup (\{\infty\}\times\RR) = \RR^3\cup Z^*.
  \]

Thus we are working with $\RR^3$ together with a vertical axis $Z^*$ at infinity.
The Riemannian metric is
\begin{equation}\label{metric}
    \left(\frac{4R^2}{4R^2+ x^2 + y^2}\right)^2(dx^2+dy^2) + dz^2.
\end{equation}
In particular, the metric coincides with the Euclidean metric along the $Z$ axis.
Inversion in the cylinder
\begin{equation} \label{C2r}
C_{2R}=\{(x,y,z)\,:\, x^2+y^2=(2R)^2\}
\end{equation}
 is an isometry of \eqref{metric}. 
Indeed, $C_{2R}$ corresponds to $E\times\RR$, 
where $E$ is the equator of $\sS^2\times\{0\}$ with respect 
to the antipodal points $O$ and $O^*$.
 We also note for further use that \begin{equation} \label{Cyldist}
\dist_R(C_{2R}\,,\,Z)= \pi R/2,
\end{equation}
where $\dist_R(\cdot,\cdot)$ is the distance function associated to the metric \eqref{metric}.

We fix a genus $g$ and choose a helicoid
 $H\subset \RR^3$  with axis $Z$ and containing $X$.
(Note that it is a helicoid for all choices of $R$.)   As usual, let $H^+$ be the component of $\RR^3\setminus H$
containing $Y^+$, the positive part of the $y$-axis.
Let $M$ be one of the nonperiodic, genus-$g$ examples described in 
 theorem~\ref{main-S2xR-theorem}.  Let 
 \[
     S=\interior (M\cap H^+).
 \]
According to theorem~\ref{main-S2xR-theorem}, $M$ and $S$ have the following properties:
\begin{enumerate}
\item\label{example-def-one} $S$  is a smooth, embedded $Y$-surface in $H^+$
that intersects $Y^+$ in exactly $g$ points, 
\item\label{example-def-two} The boundary\footnote{Here we are regarding $M$ and $S$ 
   as  subsets of $\RR^3$
with the metric~\eqref{metric}, so $\partial S$ is $X\cup Z$ and not $X\cup Z\cup Z^*$.}
of $S$ is $X\cup Z$.
\item\label{example-def-three} $M=\overline{S}\cup \rho_Z\overline{S} \cup Z^*$ is a 
  smooth surface that is minimal
  with respect to the metric~\eqref{metric}.
\end{enumerate}

\newcommand{\hh}{\eta}   
\begin{definition}\label{example-definition}
An {\em example} is a triple $(S,\hh,R)$ with $\hh>0$ and $0<R<\infty$
such that $S$ satisfies~\eqref{example-def-one}, 
\eqref{example-def-two}, and~\eqref{example-def-three}, where
$H$ is the helicoid in $\RR^3$ that has axis $Z$, that contains $X$, and
that has vertical distance between successive sheets equal to $\hh$.
In the terminology of the previous sections, $H$ is the helicoid of pitch $2\hh$.
\end{definition}

\stepcounter{theorem}
\addtocontents{toc}{\SkipTocEntry}
\subsection{Convergence Away from the Axes}

Until section~\ref{nearZ} it will  be convenient to work not in $\RR^3$ but rather
in the universal cover of $\RR^3\setminus Z$,  still with the Riemannian metric~\eqref{metric}.
Thus the angle function $\theta(\cdot)$ will be well-defined and single valued. 
However, we normalize the angle function so that $\theta(\cdot)=0$ on $Y^+$.
(In the usual convention for cylindrical coordinates, $\theta(\cdot)$ would be $\pi/2$ on $Y^+$.)
Thus $\theta=-\pi/2$ on $X^+$ and $\theta=\pi/2$ on $X^-$.

Of course $Z$ and $Z^*$ are not in the universal cover, but $\dist(\cdot, Z)$
and $\dist(\cdot, Z^*)$ still make sense.

Since we are working in the universal cover, each vertical line intersects $H^+$ in a single segment
of length $\hh$.

\begin{theorem}[First Compactness Theorem]\label{first-compactness-theorem} 
Consider a sequence $(S_n,\hh_n,R_n)$ of examples with $R_n$ bounded away from $0$ 
and with $\hh_n\to 0$.
Suppose that 

\begin{enumerate}[\upshape (*)]
\item each $S_n$ is graphical in some nonempty, open cylindrical region $U\times \RR$ such that
$\theta(\cdot)>\pi/2$ on $U\times\RR$. In other words, every vertical line in $U\times\RR$ intersects
$M_n$ exactly once.
\end{enumerate}
Then after passing to a subsequence, the $S_n$ converge smoothly away from a discrete set
$K$ to the surface $z=0$.  The convergence is with multiplicity one where $|\theta(\cdot)|>\pi/2$ and with
multiplicity two where $|\theta(\cdot)|<\pi/2$.  

Furthermore, the singular set $K$ lies in the region $|\theta(\cdot)|\le \pi/2$.
\end{theorem}

\begin{remark}\label{compactness-remark}
Later (corollary~\ref{hypothesis-satisfied-corollary} and corollary~\ref{singularities-in-Y-corollary}) 
we will show the hypothesis \thetag{*} is not needed 
and  that the singular set $K$ lies in $Y^+$.
\end{remark}

\begin{proof}
By passing to a subsequence and scaling,  we can assume that the $R_n$ converge to a limit $R\in [1,\infty]$.
Note that the $H_n^+$ converge as sets to the surface $\{z=0\}$ in the universal cover of $\RR^3\setminus Z$.
Thus, after passing to a subsequence, the $S_n$ converge as sets to a closed subset 
of the surface $\{z=0\}$.    By standard estimates for minimal graphs, the convergence is smooth
(and multiplicity one) in $U\times\RR$.   Thus the area-blowup set
\[
  Q: = \{ q: \text{$\limsup_n \area(S_n\cap \BB(q,r)) =\infty$ for all $r>0$} \}
\]
is contained in $\{z=0\}\setminus U$ and is therefore a proper subset of $\{z=0\}$.
The constancy theorem for area-blowup sets~\cite{white-controlling-area}*{theorem~4.1} 
states that the area-blowup set of a sequence
of minimal surfaces cannot be a nonempty proper subset of a smooth, connected two-manifold, provided
the lengths of the boundaries are uniformly bounded on compact sets.   Hence $Q$ is empty.
   That is, the areas
of the $S_n$ are uniformly bounded on compact sets.

Thus by the general compactness theorem~\ref{general-compactness-theorem},
after passing
to a subsequence, the $S_n$ converge smoothly away from a discrete set $K$
to a limit surface $S'$ lying in $\{z=0\}$.
The surface $S'$ has some constant multiplicity  in the region where $\theta(\cdot)>\pi/2$.  
Since the $S_n\cap (U\times\RR)$ are graphs,
that multiplicity must be $1$.  By $\rho_Y$ symmetry, the multiplicity is also $1$
where $\theta<-\pi/2$.
  Since each $S_n$ has boundary $X$, the multiplicity
of $S'$ where $|\theta(\cdot)| <\pi/2$ must be $0$ or $2$.

Note that $\tilde S_n:=S_n\cap \{|\theta(\cdot)|<\pi/2\}$ is nonempty and lies in the solid cylindrical region
\begin{equation}\label{the-region}
    \{|\theta(\cdot)|\le \pi/2\} \cap \{ |z|\le 2\hh_n\}
\end{equation} 
and that $\partial \tilde S_n$ lies on the cylindrical, vertical edge of that region.  
It follows (by theorem~\ref{all-or-nothing-theorem} in section~\ref{hemisphere-appendix}) 
that for $\hh_n/R_n$ sufficiently small,
every vertical line that intersects the region~\eqref{the-region} is at distance at most $4\hh_n$
from $S_n$.
Thus the limit of the $\tilde S_n$ as sets is all of $\{z=0\}\cap \{|\theta(\cdot)|\le \pi/2\}$,
and so the multiplicity there is two, not zero.

Since the convergence $S_n\to S'$ is smooth wherever $S'$ has multiplicity $1$ (either
by the General Compactness Theorem~\ref{general-compactness-theorem} 
or by the Allard Regularity Theorem~\cite{allard}), 
     $|\theta(\cdot)|$ must be $\le \pi/2$ at
each point of $K$. 
\end{proof}


\begin{theorem}\label{second-compactness-theorem}
Let $(S_n, \hh_n, R_n)$ be a sequence of examples with $R_n\ge 1$ and with $\hh_n\to 0$.
Let $f_n$ be the screw motion through angle $\alpha_n$ that maps $H_n^+$ to itself, 
and assume that $|\alpha_n|\to \infty$.
Let $S_n' = f_n(S_n)$.   Suppose each $S_n'$ is graphical in some nonempty open cylinder $U\times\RR$
(as in \thetag{*} in Theorem~\ref{first-compactness-theorem}).
Then the $S_n'$ converge smoothly (on compact sets) with multiplicity one to the surface $\{z=0\}$.
\end{theorem}

The proof is almost identical to the proof of theorem~\ref{first-compactness-theorem}, so we omit it.

\begin{theorem}\label{graphical-theorem}
For every genus $g$ and  angle $\alpha>\pi/2$, there is a $\lambda<\infty$ with the following
property.  If $(S,\hh,R)$ is a genus-$g$ example (in the sense of definition~\ref{example-definition}) with 
\[
    \dist(Z,Z^*) = \pi R > 4\lambda \hh,
\]
then $S$ is graphical in the region
\[
   Q(\lambda \hh, \alpha):= \{|\theta(\cdot)| \ge \alpha, \, \dist(\cdot, Z\cup Z^*) \ge \lambda \hh \}.
\]
\end{theorem}

\begin{proof}
Suppose the result is false for some $\alpha>\pi/2$, and let $\lambda_n\to\infty$.
Then for each $n$, there is an example $(S_n,\hh_n,r_n)$ such that
\begin{equation}
\dist_n(Z,Z^*) > 4\lambda_n \hh_n
\end{equation}
and such that $S_n$  is not a graphical in $Q(\lambda_n\hh_n,\alpha)$.
Here $\dist_n(\cdot,\cdot)$ denotes distance with 
respect the metric that comes from $\sS^2(R_n)\times \RR$.
However, henceforth we will write $\dist(\cdot, \cdot)$ instead of $\dist_n(\cdot,\cdot)$ to reduce
notational clutter.

Since the ends of $M_n=\overline{S_n\cup\rho_ZS_n}$ 
are asymptotic to helicoids as $z\to\pm\infty$, 
note that $S_n$ is graphical in $Q(\lambda_n \hh_n, \beta)$ for all sufficiently large $\beta$.
Let $\alpha_n\ge \alpha$ be the largest angle such that $S_n$ is 
not graphical in $Q(\lambda_n\hh_n, \alpha_n)$.
Note that there must be a point $p_n\in S_n$
such that
\begin{align}
 \theta(p_n)&=\alpha_n, \\ 
 \dist(p_n,Z\cup Z^*) &> \lambda_n \hh_n, 
 \end{align}
 and such that $\Tan(S_n,p_n)$ is vertical.
 Without loss of generality, we may assume (by scaling) that $\dist(p_n, Z\cup Z^*)=1$.
 In fact, by symmetry of $Z$ and $Z^*$, we may assume that 
 \[
   1 = \dist(p_n,Z) \le \dist(p_n, Z^*),
 \]
 which of course implies that $\pi R_n = \dist(Z,Z^*) \ge 2$, and therefore that
\[
  \lambda_n\hh_n \le \frac12.
\]
By passing to a further subsequence, we can assume that
\[
  \lambda_n \hh_n \to \mu\in [0,\frac12].
\]
Since $\lambda_n\to\infty$, this forces
\[
  \hh_n\to 0.
\]
We can also assume that
\[
\alpha_n\to \talpha \in [\alpha, \infty].
\]

Case 1:  $\talpha<\infty$.
Then the $p_n$ converge to a point $p$ with $\theta(p)=\talpha$
and with $\dist(p,Z)=\dist(p,Z\cup Z^*)=1$.

Note that $M_n$ is graphical in the region $Q(\lambda_n \hh_n, \alpha_n)$,
and that those regions converge to $Q(\mu, \alpha)$.  Thus by the Compactness
Theorem~\ref{first-compactness-theorem}, 
the $S_n$ converge smoothly and with multiplicity one to $\{z=0\}$
in the region $|\theta(\cdot)|>\pi/2$.   But this is a contradiction since $p_n\to p$,
which is in that region, and since $\Tan(S_n,p_n)$ is vertical.

Case 2: Exactly as in case 1, except that we apply a screw motion $f_n$ to $M_n$
such that $\theta(f_n(p_n))=0$. (We then use theorem~\ref{second-compactness-theorem}
 rather than theorem~\ref{first-compactness-theorem}.)
\end{proof}

\begin{corollary}\label{hypothesis-satisfied-corollary}
The hypothesis~(*) in theorems~\ref{first-compactness-theorem} and~\ref{second-compactness-theorem}
 is always satisfied provided $n$ is sufficiently large.
\end{corollary}

\stepcounter{theorem}
\addtocontents{toc}{\SkipTocEntry}
\subsection{Catenoidal Necks} The next theorem shows that, in the compactness theorem~\ref{first-compactness-theorem}, any point away from $Z\cup Z^*$
where the convergence is not smooth must lie on $Y$, and that near such a point,
the  $S_n$ have small catenoidal necks.

\begin{theorem}\label{limit-is-a-catenoid-theorem}
Let $(S_n,\hh_n,R_n)$ be a sequence of examples and
 $p_n\in S_n$ be a sequence of points  such that
\begin{equation}\label{slope-bigger-than-delta-hypothesis}
\slope(S_n,p_n)\ge \delta>0, 
\end{equation}
(where $\slope(S_n,p_n)$ is the slope of the tangent plane to $S_n$ at $p_n$) and such that
\begin{equation}\label{away-from-axes-hypothesis}
\frac{\dist(p_n,{Z\cup Z^*})}{\hh_n}  \to \infty.
\end{equation}
Then there exist positive numbers $c_n$ such that (after passing to a subsequence) the surfaces
\begin{equation}\label{rescaled-catenoidish}
   \frac{S_n-p_n}{c_n}
\end{equation}
converge to a catenoid in $\RR^3$.  
The waist of the catenoid is a horizontal circle, and the line $(Y^+-p_n)/c_n$ converges to a line
that intersects the waist  in two diametrically opposite points.

Furthermore, 
\begin{equation}\label{h_n-relatively-large}
\frac{\hh_n}{c_n} \to \infty. 
\end{equation}
and
\begin{equation}\label{h_n-relatively-large2}
\frac{\dist(p_n, S_n\cap Y^+)}{\hh_n} \to 0.  
\end{equation}
\end{theorem}

\begin{proof}
By scaling and passing to a subsequence, we may assume that
\begin{equation}\label{distance-to-Z-normalized}
   1 = \dist(p_n, Z) \le \dist(p_n, Z^*).
\end{equation}
and that $\theta(p_n)$ converges to a limit $\alpha\in [-\infty,\infty]$.
By~\eqref{away-from-axes-hypothesis} (a statement that is scale invariant)
 and by~\eqref{distance-to-Z-normalized}, $\hh_n\to 0$.
Thus by~theorem~\ref{graphical-theorem} (and standard estimates for minimal graphs), $|\alpha|\le \pi/2$.

First we prove that there exist $c_n\to 0$ such that the surfaces $(S_n-p_n)/c_n$
converge subsequentially to a catenoid with horizontal ends.

{\bf Case 1}: $|\alpha|=\pi/2$.  By symmetry, it suffices to consider the case $\alpha=\pi/2$.
Let $\widetilde S_n$ be obtained from $S_n$ by Schwartz reflection about $X^-$.

By the last sentence of the general compactness theorem~\ref{general-compactness-theorem}, 
there exist numbers $c_n\to 0$
such that (after passing to a subsequence) the surfaces
\[
  \frac{\widetilde{S}_n - p_n}{c_n}
\]
converge smoothly to a complete, non-flat, properly embedded minimal surface $\widetilde{S}\subset \RR^3$ of finite
total curvature whose ends are horizontal.   
By proposition~\ref{genus-0-proposition}, $\widetilde{S}$ has genus $0$.  
By  a theorem of Lopez and Ros\cite{lopez-ros}, the only 
nonflat, properly embedded minimal surfaces in $\RR^3$ with genus zero and finite total curvature are the catenoids.
Thus $\tilde S$ is a catenoid.   Note that 
\[
   \frac{S_n-p_n}{c_n}
\]
converges to a portion $S$ of $\widetilde{S}$.  Furthermore, $S$ is either all of $\widetilde{S}$, or 
it is a portion of $\widetilde S$ bounded by a horizontal line $\widetilde{X}= \lim_n ((X^- - p_n)/c_n$ 
in $\widetilde{S}$.  
Since catenoids contain no lines, 
in fact $S=\widetilde{S}$ is a catenoid.

{\bf Case 2}:  $|\alpha|<\pi/2$.
By the last statement of the general compactness theorem~\ref{general-compactness-theorem}, 
there are $c_n>0$ tending to $0$ such
that (after passing to a subsequence) the surfaces
\[
 \frac{S_n-p_n}{c_n}
\]
converge smoothly to a complete, nonflat, embedded minimal surface $S\subset \RR^3$ of finite
total curvature with ends parallel to horizontal planes.  By monotonicity, 
\begin{equation}\label{monotonicity}
  \limsup_n \frac{\area\left(\frac{S_n-p_n}{c_n} \cap \BB(0,\rho)\right)}{\pi \rho^2} \le 2
\end{equation}
for all $\rho>0$.  Thus $S'$ has density at infinity $\le 2$, so it has at most two ends.  
If it had just one end, it would be a plane. But it is not flat, so that is impossible. Hence it has two ends. 
By a theorem of Schoen~\cite{schoen-uniqueness}, 
a properly embedded minimal surface in $\RR^3$ of 
with finite total curvature and two ends must be  a catenoid.

This completes the proof that (after passing to a subsequence) the surfaces $(S_n-p_n)/c_n$
converge to a catenoid $S$ with horizontal ends.

Note that for large $n$, there is a simple closed geodesic $\gamma_n$ in $S_n$
such that $(\gamma_n-p_n)/c_n$ converges to the waist of the catenoid $S$.
Furthermore, $\gamma_n$ is unique in the following sense: if $\gamma_n'$ is a
simple closed geodesic in $S_n$ that converges to the waist of the catenoid $S$, then
$\gamma_n'=\gamma_n$ for all sufficiently large $n$.  (This follows from the implicit
function theorem and the fact that the waist $\gamma$ of the catenoid is non-degenerate
as a critical point of the length function.)

\begin{claim} $\rho_Y\gamma_n=\gamma_n$ for all sufficiently large $n$.
\end{claim}

\begin{proof}[Proof of claim]
 Suppose not.  
Then (by passing to a subsequence) we can assume that $\gamma_n\ne\rho_Y\gamma_n$ for all $n$.
Thus (passing to a further subsequence) the curves $(\rho_Y\gamma_n-p_n)/c_n$ do one of the following:
(i) they converge to $\gamma$, (ii) they converge to another simple closed geodesic in $S$ having the same
length as $\gamma$, or (iii) they diverge to infinity.  
Now (i) is impossible by the uniqueness of the $\gamma_n$.  Also, (ii) is impossible because the waist
$\gamma$ is the only simple closed geodesic in the catenoid $S$.
Thus (iii) must hold: the curves $(\rho_Y\gamma_n - p_n)/c_n$ diverge to infinity.

Since $S_n$ is a $Y$-surface, $\gamma_n$ together with $\rho_Y\gamma_n$ bound a region $A_n$
in $S_n$.  By the maximum principle, $\theta(\cdot)$ restricted to $A_n$ has its maximum on one of the
two boundary curves $\gamma_n$ and $\rho_Y\gamma_n$ and (by symmetry) its minimum on the other.
(Note that the level sets of $\theta$ are totally geodesic and therefore minimal.)

\newcommand{\hA}{\hat{A}}
By passing to a subsequence, we can assume that the regions $(A_n-p_n)/c_n$ converges to a subset $\hA$
of the catenoid $S$.
Note that $\hA$ is the closure of one of the components of $S\setminus \Gamma$.  (This is because one of the two boundary components of $(A_n-p_n)/c_n$,
namely $(\gamma_n-p_n)/c_n$, converges to the waist of the catenoid, 
whereas the other boundary component, namely $(\rho_Y\gamma_n-p_n)/c_n$, diverges to infinity.)
The fact that $\theta | \overline{A_n}$ attains it maximum on $\gamma_n$ implies that
there is a linear function $L$ on $\RR^3$ with horizontal gradient such that $L|  \hA$
attains its maximum on the waist $\gamma=\partial \hA$.  But that is impossible since the catenoid $S$
has a horizontal waist.  This proves the claim.
\end{proof}

Since each $\gamma_n$ is $\rho_Y$-invariant (by the claim), it follows that 
the waist $\gamma$ is invariant under $180^\circ$ rotation about the line $Y'$,
where $Y'$ is a subsequential limit of the curves $(Y_n-p_n)/c_n$.
Since $\gamma$ is a horizontal circle, $Y'$ must be a line that bisects the circle.
Thus
\begin{equation}\label{normalized-distance-to-Y}
   \frac{\dist(p_n, S_n\cap Y)}{c_n} \to \dist(O, S\cap Y') < \infty.
\end{equation}

Note that $\hh_n/c_n \to \infty$, since if it converged to a finite limit, 
then the regions $(H_n^+ - p_n) /c_n$ would converge to a horizontal slab of finite
thickness and the catenoid $S$ would be contained in that slab, a contradiction.
This completes the proof of~\eqref{h_n-relatively-large}.

Finally,~\eqref{h_n-relatively-large2} follows immediately 
from~\eqref{h_n-relatively-large} and~\eqref{normalized-distance-to-Y}.
\end{proof}

\begin{corollary}\label{singularities-in-Y-corollary}
The singular set $K$ in theorem~\ref{first-compactness-theorem} is a finite subset of $Y^+$.
In fact (after passing to a subsequence), $p\in K$ if and only if there is a sequence $p_n\in Y^+\cap S_n$
such that $p_n\to p$.
\end{corollary}

The following definition is suggested by theorem~\ref{limit-is-a-catenoid-theorem}:

\begin{definition}\label{neck-definition}
Let $(S,\hh,R)$ be an example (as in definition~\ref{example-definition}).  
Consider the set of points of $S$ at which the tangent plane is vertical.
A {\em neck} of $S$ is a connected component of that set consisting of a simple closed
curve that intersects $Y^+$ in exactly two points.
The {\em radius}
of the neck is half the distance between those two points, and the {\em axis} of the neck
is the vertical line that passes through the midpoint of those two points.
\end{definition}


\begin{theorem}\label{graph-form-theorem}
Suppose that $(S,\hh,R)$ is an example (as in definition~\ref{example-definition}) and that $V$ is a vertical line.
If $V$ is not too close to $Z\cup Z^*$ and also not too close to any neck
axis, then $V$ intersects $M$ in at most two points, and the tangent planes
to $M$ at those points are nearly horizontal.   
Specifically, 
for every $\eps>0$, there is a $\lambda$ (depending only on genus and $\eps$) 
with the following properties.
Suppose that
\[
   \frac{\dist(V,Z\cup Z^*)}{\hh} \ge \lambda,
\]
and that for every neck axis $A$, either
\[
    \frac{\dist(V,A)}{r(A)} \ge \lambda
\]
(where $r(A)$ is the neck radius)
or
\[
    \frac{\dist(V,A)}{\hh} \ge 1.
\]
Then
\begin{enumerate}[\upshape (i)]
\item\label{slope-small-item}The slope of the tangent plane at each point in $V\cap M$ is $<\eps$, and
\item\label{one-or-two-item} $V$ intersects $\overline{S}$ in exactly one point if $\theta(V)>\pi/2$ and in exactly
two points if $\theta(V)\le \pi/2$.
\end{enumerate}
\end{theorem}

\begin{proof} Let us first prove that there is a value $\Lambda<\infty$  of $\lambda$ such that 
assertion~\eqref{slope-small-item} holds.  Suppose not.   
Then there exist examples $(S_n,\hh_n,R_n)$
and vertical lines $V_n$ such that
\begin{equation}\label{thing-one}
   \frac{\dist(V_n,Z\cup Z^*)}{\hh_n} \ge \lambda_n\to \infty,
\end{equation}
and such that
\begin{equation}\label{thing-two}
   \frac{\dist(V_n,A)}{r(A)} \ge \lambda_n \quad\text{or}\quad  \frac{\dist(V_n,A)}{\hh_n} \ge 1
\end{equation}
for every neck axis $A$ of $S_n$, 
but such that $V_n\cap S_n$ contains a point $p_n$ at which the slope of $M_n$ is $\ge \eps$.

Note that~\eqref{thing-one} and~\eqref{thing-two} are scale invariant.  We can can choose coordinates
so that $p_n$ is at the origin and, by theorem~\ref{limit-is-a-catenoid-theorem}, 
we can choose scalings so that the $S_n$
converge smoothly to a catenoid in $\RR^3$.  Let $A'$ be the axis of the catenoid, $r(A')$ be the
radius of the waist of the catenoid, and $V'$ be the vertical line through the origin.
Then $\dist(A',V')$ is finite, $r(A')$ is finite and nonzero, and $\hh_n\to\infty$ 
by~\eqref{h_n-relatively-large}.  Thus if $A_n$ is the neck axis of $S_n$ that
converges to $A'$, then
\[
   \lim_n \frac{\dist(V_n,A_n)}{r(A_n)}  <\infty
   \quad\text{and}\quad
   \lim_n\frac{\dist(V_n,A_n)}{\hh_n}  =  0,
\]
contradicting~\eqref{thing-two}.  This proves
that there is a value
of $\lambda$, call it $\Lambda$, that makes assertion~\eqref{slope-small-item} of the theorem true.

Now suppose that there is no $\lambda$ that makes assertion~\eqref{one-or-two-item} true.
Then there is a sequence $\lambda_n\to\infty$, a sequence of examples $(S_n,\hh_n,R_n)$,
and a sequence of vertical lines $V_n$ such that~\eqref{thing-one} and~\eqref{thing-two}
hold, but such that $V_n$ does not intersect $M_n$ in the indicated number of points.
By scaling, we may assume that 
\[
   1 = \dist(V_n, Z)\le \dist(V_n, Z^*),
\]
which implies that $R_n$ is bounded below and (by~\eqref{thing-one}) that $\hh_n\to 0$.
We may also assume that $\theta(V_n)\ge 0$, and
that each $\lambda_n$ is greater than $\Lambda$.
Thus $V_n$ intersects $S_n$ transversely.
For each fixed $n$, if we move $V_n$ in such a way that $\dist(V_n,Z)=1$ stays constant and
that $\theta(V_n)$ increases, then~\eqref{thing-one} and~\eqref{thing-two} remain true, 
so $V_n$ continues to be transverse to $M_n$.
Thus as we move $V_n$ in that way, the number of points in 
   $V_n\cap \overline{S_n}$ does not change unless 
$V_n$ crosses $X$, so we may assume that $\theta(V_n)\ge \pi/4$.
But now theorem~\ref{first-compactness-theorem} and Remark~\ref{compactness-remark} 
imply that $V_n\cap \overline{S_n}$ has
the indicated number of intersections, contrary to our assumption that it did not.
\end{proof}

\begin{corollary}\label{form-of-graph-corollary}
Let $\eps>0$ and $\lambda>1$ be as in theorem~\ref{graph-form-theorem}, and 
let $(S,\hh,R)$ be an example.
Consider the following cylinders: 
vertical solid cylinders of radius $\lambda \hh$ about $Z$ and $Z^*$, 
and for each neck 
axis\,\footnote{See~\ref{neck-definition} for the definition of ``neck axis".} $A$ of $S$
with $\dist(A,Z\cup Z^*)>(\lambda-1) \hh$, a vertical solid cylinder with axis $A$ 
and radius $\lambda r(A)$.
 Let $J$ be the union of those 
cylinders.
Then $S\setminus J$ consists of two components, one
of which can be parametrized as
\[
   \{ (r\cos\theta, r\sin\theta, f(r,\theta)):   r>0,\, \theta\ge -\pi/2\}  \setminus J
\]
where $f(r,-\pi/2) \equiv 0$ and where
\begin{equation}\label{trapped}
       \hh\left( \frac{\theta}{\pi} - \frac{1}2\right) \le f(r,\theta) \le \hh\left( \frac{\theta}{\pi} + \frac{1}2\right).
\end{equation}
\end{corollary}

Of course, by $\rho_Y$ symmetry, the other component of $S\setminus J$ can be written
\[
   \{ (r\cos\theta, r\sin\theta, -f(r,-\theta)):   r>0,\quad \theta\le \pi/2\}  \setminus J.
\]

The inequality~\eqref{trapped} expresses the fact that $S$ lies in $H^+$.
Note that in corollary~\ref{form-of-graph-corollary}, because we are working in the universal cover of
 $\RR^3\setminus Z$,
each vertical cylinder about a neck axis in the collection $J$ 
intersects $H^+$ in a single connected component. (If we were working in $\RR^3$, it would intersect
$H^+$ in infinitely many components.)   Thus the portion of $S$ that lies in such a cylinder is a single
catenoid-like annulus.  If we were working in $\RR^3$, the portion of $S$ in such a cylinder would
be that annulus together with countably many disks above and below it.

\begin{remark}\label{form-of-graph-remark}
In Corollary~\ref{form-of-graph-corollary}, 
the function $f(r,\theta)$ is only defined for $\theta\ge -\pi/2$. 
It is positive
for $\theta> -\pi/2$ and it vanishes where $\theta=\pi/2$.   Note that we can extend $f$ by Schwarz reflection
to get a function $f(r,\theta)$ defined for all $\theta$:
\begin{align*}
  f(r,\theta) &= f(r,\theta) \qquad\text{for $\theta\ge -\pi/2$, and} \\
  f(r,\theta) &= -f(r, -\pi-\theta)  \qquad\text{for $\theta < -\pi/2$.}
\end{align*}
Corollary~\ref{form-of-graph-corollary} 
states that, after removing the indicated cylinders, we can express $S$ (the portion of $M$
in $H^+$) as the union of two multigraphs: the graph of the original, unextended $f$ together with the image of that graph under
$\rho_Y$.  Suppose we remove from $M$ those cylinders together with their images under $\rho_Z$.
Then the remaining portion of $M$ can be expressed as the the union of two multigraphs:
the graph of the extended function $f$ (with $-\infty<\theta<\infty$) together with the image of that graph 
under $\rho_Y$.
\end{remark}

\begin{remark}
Note that $H\setminus (Z\cup X)$ consists of four quarter-helicoids, two of which are
described in the universal cover of $\RR^3\setminus Z$ by
\[
    z =   \frac{\hh}{\pi} \left( \theta + \frac{\pi}2\right), \quad (\theta\ge  -\pi/2)
\]
and
\[
   z =  \frac{\hh}{\pi}\left( \theta - \frac{\pi}2\right), \quad (\theta \le \pi/2).
\]
(As in the rest of this section, we are measuring $\theta$ from $Y^+$ rather than from $X^+$.)
These two quarter-helicoids overlap only in the region $-\pi/2<\theta<\pi/2$:
a vertical line in that region intersects both quarter-helicoids in points that are
distance $\hh$ apart, whereas any other vertical line intersects only one
of the two quarter-helicoids.
Roughly speaking, theorem~\ref{graph-form-theorem} 
and corollary~\ref{form-of-graph-corollary} say
that if $(S,\hh,R)$ is an example with $R/\hh$ large, then $S$ must be obtained from these two
quarter-helicoids by joining them by catenoidal necks away from $Z$ and 
in some possibly more complicated way near $Z$.   The catenoidal necks lie
along the $Y$-axis.

Figure~\ref{unrolled-figure} illustrates the intersection of $M=S\cup\rho_Z S$ with a vertical cylinder with axis  $Z$.  The shaded region is the intersection of the cylinder with $H^+$. The intersections of the cylinder with the quarter-helicoids are represented by halflines  on the boundary of  the shaded region:  $\theta\geq - \pi/$2 on top of the shaded region, and  $\theta\leq \pi/2$ on the bottom. The radius of the cylinder is chosen so that the cylinder passes though a catenoidal neck of S that can be thought of as joining the quarter-helicoids, allowing $S$ to make a transition from approximating one quarter helicoid to approximating to the other. The transition takes place in the region $-\pi/2\leq\theta\leq \pi/2$.
\end{remark}


\begfig
\begin{center}
\includegraphics[height=40mm]{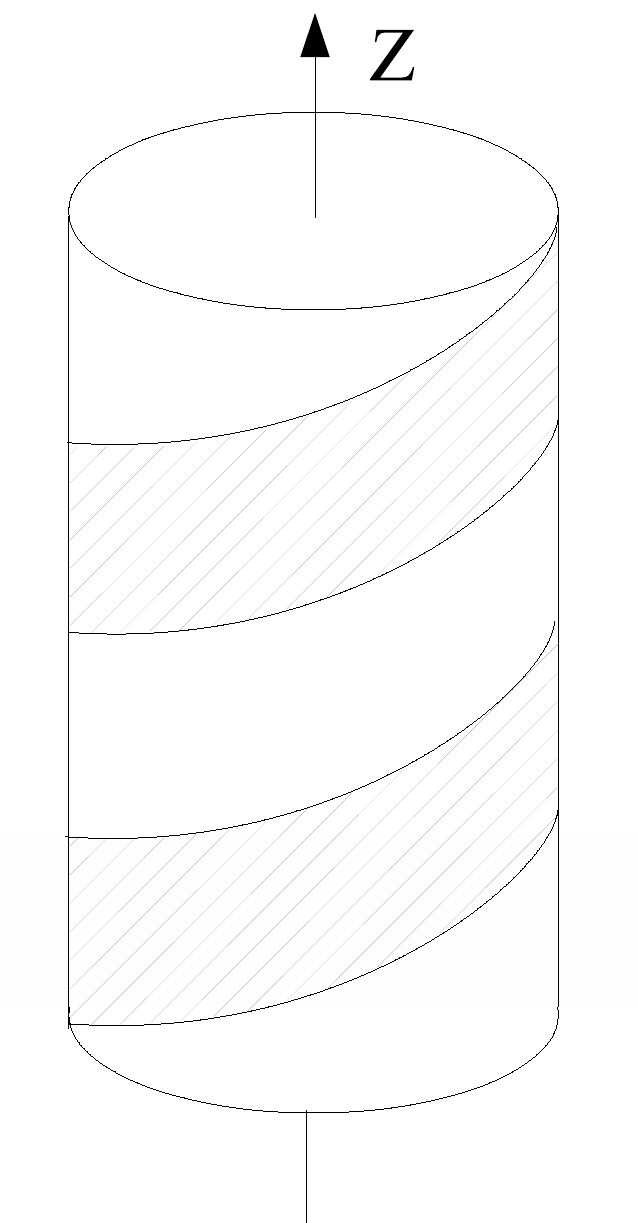}
\hspace{1cm}
\includegraphics[height=40mm]{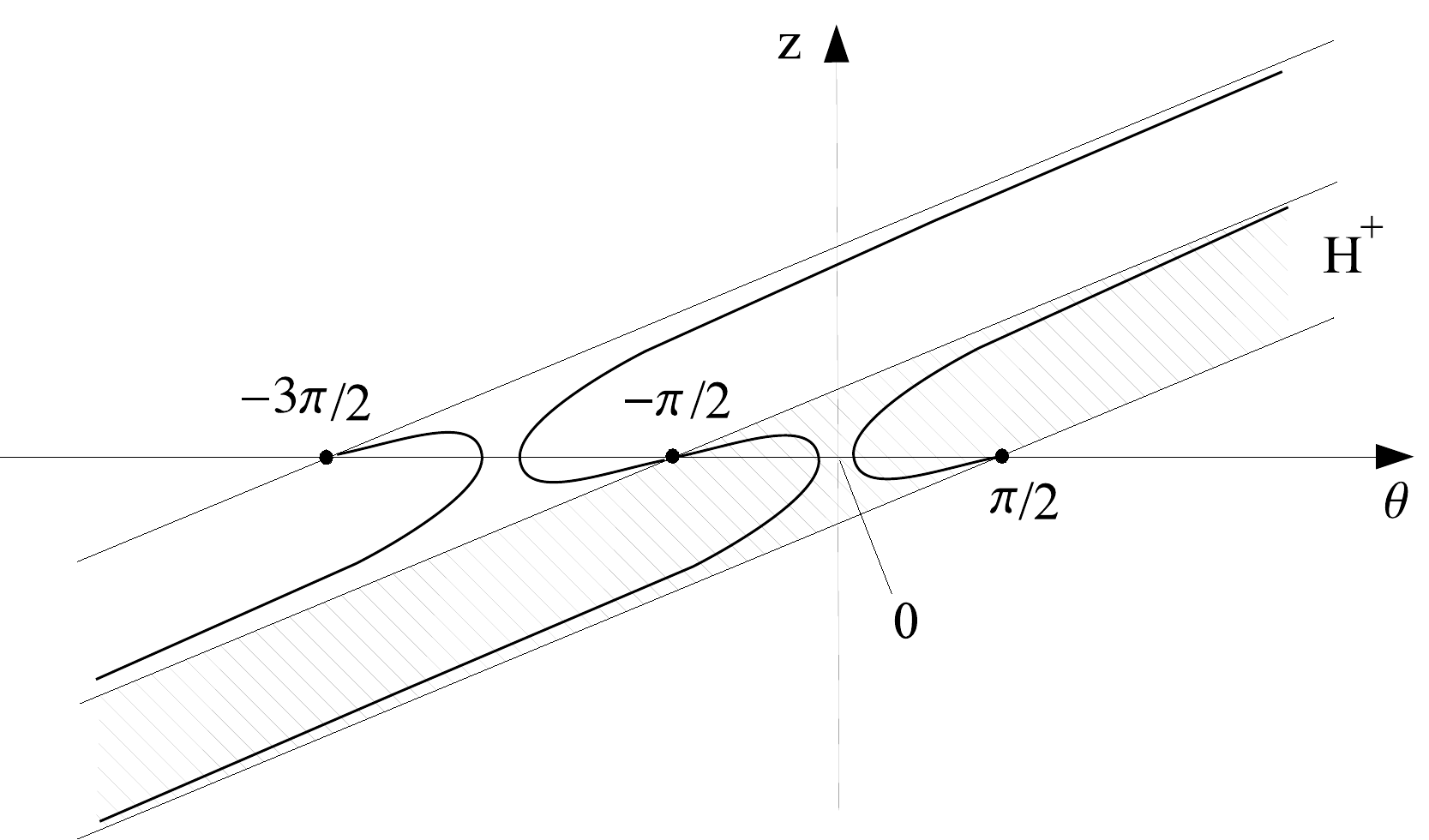}
\end{center}
\caption{{\bf Left}: the shaded region is the intersection of $H^+$ with the vertical cylinder of axis $Z$ and
radius $r$. {\bf Right}: intersection of $M$ with the same cylinder, unrolled in the plane. We use cylindrical coordinates $(r,\theta,z)$, with $\theta=0$ being the positive $Y$-axis.
The radius $r$ is chosen
so that the cylinder intersects a catenoidal neck. The positive $X$-axis intersects the cylinder at the point $(\theta,z)=(-\pi/2,0)$. The negative $X$-axis intersects the cylinder at the point $(\theta,z)=(\pi/2,0)$, which is the same as the point $(-3\pi/2,0)$ on the cylinder.}
\label{unrolled-figure}
\endfig


\stepcounter{theorem}
\addtocontents{toc}{\SkipTocEntry}
\subsection{Behavior near \texorpdfstring{$Z$}{Lg}}\label{nearZ}

In this section, 
we consider examples (see definition~\ref{example-definition}) $(S_n, 1, R_n)$ 
with $\eta=1$ fixed and with $R_n\to\infty$.
We will work in $\RR^3$ (identified with $(\sS^2(R_n)\times \RR)\setminus Z^*$ by stereographic
projection as described at the beginning of section~\ref{R3-section}), 
rather than in the universal cover of $\RR^3\setminus Z$.

\newcommand{\tM}{\tilde M}
\begin{theorem}\label{R-to-infinity-theorem}
Let $(S_n,1, R_n)$ be a sequence of examples with $R_n\to\infty$.
Let $\sigma_n$ be a sequence of screw motions of $\RR^3$ that map $H$ to itself.
Let
\[
  M_n = \sigma_n( \overline{S_n \cup \rho_Z S_n}).
\]
In other words, $M_n$ is the  full genus-$g$ example (of which  $S_n$ is the subset in the interior of $H^+$)  followed by 
the screw motion $\sigma_n$.
Then (after passing to a subsequence), the $M_n$ converge
smoothly on compact sets to a properly embedded, multiplicity-one minimal surface $M$ in $\RR^3$.
 Furthermore, there is a solid cylinder $C$ about $Z$ such that $M\setminus C$ is the union
of two multigraphs.
\end{theorem}

Thus the family $\FF$ of all such subsequential limits~$M$ (corresponding to arbitrary sequences of $M_n$ and 
$\sigma_n$) is compact with respect to smooth
convergence.  It is also closed under screw motions that leave $H$ invariant.  Those two
facts immediately imply the following corollary:

\begin{corollary}\label{R-to-infinity-corollary} 
Let $\FF$ be the family of all such subsequential limits.
For each solid cylinder $C$ around $Z$, each $M\in \FF$, and each $p\in C\cap M$, the
curvature of $M$ at $p$ is bounded by a constant $k(C)<\infty$ depending
only on $C$ (and genus).
\end{corollary}

\begin{proof}[Proof of theorem~\ref{R-to-infinity-theorem}]
Let $0<d_1^n < d_2^n < \dots < d_g^n$ be the distances of the points in $S_n\cap Y^+$ to the origin.
By passing to a subsequence, we may assume that the limit
\[
 d_k:= \lim_{n\to\infty} d_k^n  \in [0,\infty]
\]
exists for each $k$.  Let $d$ be the largest finite element of $\{d_k: k=1,\dots, g\}$.
By passing to a further subsequence, we may assume that the $\sigma_nM_n$ converge
as sets to a limit set $M$.

Let $C$ be a solid cylinder of radius $\ge (\lambda+1)(1+d)$ around 
  $Z$ where $\lambda$ is as in
corollary~\ref{form-of-graph-corollary}
for $\eps=1$.   Let $\hat C$ be any larger solid cylinder around $Z$.
By corollary~\ref{form-of-graph-corollary} (see also remark~\ref{form-of-graph-remark}),
for all sufficiently large $n$,  $M_n\cap (\hat C\setminus C)$ is the union of two
 smooth multigraphs, and for each vertical line $V$
in $\hat C\setminus C$, each connected component of $V\setminus H$ intersects $\tM_n$ at most twice.
In fact, all but one such component must intersect $M_n$ exactly once.

By standard estimates for minimal graphs, the convergence $M_n\to M$ is smooth
and multiplicity $1$ in the region $\RR^3\setminus C$.  
It follows immediately that 
$M\setminus C$
is the union of two multigraphs, and that
 the area blowup set
\[
    Q:= \{ p:\text{$ \limsup_{n\to\infty} \area(\sigma_nM_n\cap \BB(p,r))=\infty$ for every $r>0$} \}
\]
is contained in $C$.

The halfspace theorem for area blowup sets \cite{white-controlling-area}*{7.5} says that if an area blowup set is contained in a halfspace
of $\RR^3$, then that blowup set must contain a plane.  Since $Q$ is contained in the cylinder $C$,
it is contained in a halfspace but does not contain a plane.
Thus $Q$ must be empty.
Consequently, the areas of the $M_n$ are uniformly bounded locally.  Since the genus is also bounded,
we have, by the General Compactness Theorem~\ref{general-compactness-theorem}, that $M$ is a smooth embedded minimal hypersurface, and that 
either
\begin{enumerate}
\item the convergence $M_n\to M$ is smooth and multiplicity $1$, or
\item the convergence $M_n\to M$ is smooth with some multiplicity $m>1$ away from a discrete set.
   In this case, $M$ must be stable.
\end{enumerate}

Since the multiplicity is $1$ outside of the solid cylinder $C$, it follows that the convergence
$M_n\to M$ is everywhere smooth with multiplicity $1$.
\end{proof}

\begin{theorem}\label{asymptotic-to-H-theorem}
Suppose that in theorem~\ref{R-to-infinity-theorem}, the screw motions $\sigma_n$ are all the identity map.
Let $M$ be a subsequential limit of the $M_n$, and suppose that $M\ne H$.
Then $M\cap H=X\cup Z$ and $M$ is asymptotic to $H$ at infinity.
\end{theorem}

\begin{proof}
Since $M_n\cap H=X\cup Z$ for each $n$, the smooth convergence implies that $M$ cannot
intersect $H$ transversely at any point not in $X\cup Z$.  It follows from the strong maximum principle
that $M$ cannot touch $H\setminus (X\cup Z)$.

Since $M$ is embedded, has finite topology, and has infinite total curvature, it follows
from work by Bernstein and Breiner~\cite{bernstein-breiner-conformal} 
 or by Meeks and Perez~\cite{meeks-perez-end} that $M$ is asymptotic to some helicoid $H'$
at infinity.  The fact that $M\cap H=X\cup Z$ implies that $H'$ must be $H$.

The works of Bernstein-Breiner and Meeks-Perez quoted in the previous
paragraph rely on many deep results of Colding and Minicozzi.
We now give a more elementary proof that $M$ is asymptotic to a helicoid at infinity.

According to theorem~4.1 of~\cite{hoffman-white-geometry}, a properly immersed nonplanar
minimal surface in $\RR^3$ with finite genus, one end, and bounded curvature must be asymptotic
to a helicoid and must be conformally a once-punctured Riemann surface provided it contains $X\cup Z$
and provided it intersects some horizontal plane $\{x_3=c\}$ in a set that, outside of a compact region
in that plane, consists of two disjoint smooth embedded curves tending to $\infty$.
Now $M$ contains $X\cup Z$ and has 
bounded curvature (by corollary~\ref{R-to-infinity-corollary}).   
Thus to prove theorem~\ref{asymptotic-to-H-theorem}, it suffices to prove lemmas~\ref{one-end-lemma} 
and~\ref{z=0-level-lemma} below.
\end{proof}

\begin{lemma}\label{one-end-lemma}
Let $M$ be as in theorem~\ref{asymptotic-to-H-theorem}.  Then $M$ has exactly one end.
\end{lemma}

\begin{proof}
Let $Z(R)$ denote the solid cylinder with axis $Z$ and radius $R$.
By theorem~\ref{R-to-infinity-theorem}, for all sufficiently large $R$, the set 
\[
   M \setminus Z(R)
\]
is the union of two connected components (namely multigraphs) that are related to each other by $\rho_Z$.
We claim that for any such $R$, the set
\[
   M\setminus (Z(R)\cap \{|z|\le R\}) \tag{*}
\]
contains exactly one connected component. 
To see that is has exactly one component, 
let $\mathcal{E}$ be the component of \thetag{*} containing $Z^+\cap\{z>R\}$. 
 Note that $\mathcal{E}$ is invariant under
$\rho_Z$.  Now $\mathcal{E}$ cannot be contained in $Z(R)$ by the maximum principle (consider
catenoidal barriers).    Thus $\mathcal{E}$ contains one of the two connected components of 
  $M\setminus Z(R)$.
By $\rho_Z$ symmetry, it must then contain both components of $M\setminus Z(R)$. 
 It follows that if the set~\thetag{*} had
a connected component other than $\mathcal{E}$, that component would have to lie
in $Z(R) \cap \{z<-R\}$.  But such a component would violate
the maximum principle.
\end{proof}

\begin{lemma}\label{z=0-level-lemma}
Let $M$ be as in theorem~\ref{asymptotic-to-H-theorem}.  Then $M\cap \{z=0\}$ is the union
of $X$ and a compact set.
\end{lemma}

\begin{proof}
In the following argument, it is convenient to choose the angle function $\theta$
on $H^+$ so that $\theta=0$ on $X^+$, $\theta=\pi/2$ on $Y^+$, and $\theta=\pi$
on $X^-$.  

By theorem~\ref{R-to-infinity-theorem}, for all sufficiently large $R$, the set 
\[
   M \setminus Z(R)
\]
is the union of two multigraphs that are related to each other by $\rho_Z$.

By the smooth convergence $M_n\to M$ together with
 corollary~\ref{form-of-graph-corollary} and remark~\ref{form-of-graph-remark}, 
 one of the components of $M\setminus Z(R)$ 
 can be parametrized
as
\[
   (r\cos\theta, r\sin\theta, f(r,\theta)) \qquad (r\ge R,\,  \theta \in \RR)
\]
where
\begin{equation}\label{boundary-condition}
f(r,0)\equiv 0
\end{equation}
and
\begin{equation}\label{new-bound}
    \theta-\pi < f(r,\theta) < \theta + \pi.
\end{equation}
(The bound~\eqref{new-bound} looks different from the bound~\eqref{trapped}
 in corollary~\ref{form-of-graph-corollary} 
because there we were measuring $\theta$ from $Y^+$ whereas here we are measuring it from $X^+$.)

Of course $f$ solves the minimal surface equation in polar coordinates.

For $0<s<\infty$, define a function $f_s$ by
\[
    f_s(r,\theta) =  \frac1{s} f( s r, \theta).
\]
Going from $f$ to $f_s$ corresponds to dilating $S$ by $1/s$.  
(To be more precise, $(r,\theta) \mapsto (r\cos\theta, r\sin\theta, f_s(r,\theta))$
parametrizes the dilated surface.)
Thus the function $f_s$ will also solve the polar-coordinate minimal surface equation.
By~\eqref{new-bound},
\begin{equation}\label{f_s-bound}
   \theta - \pi \le  s f_s(r,\theta) \le \theta+ \pi.
\end{equation}
By the Schauder estimates for $f_s$ and by the bounds~\eqref{f_s-bound},
the functions $s f_s$ converge smoothly (after passing to a subsequence) as $s\to 0$ to a harmonic
function $g(r,\theta)$ defined for all $r>0$ and satisfying
\begin{equation}\label{g-bound}
   \theta - \pi \le g \le \theta + \pi.
\end{equation}
Here ``harmonic'' is with respect to the standard conformal structure on $\sS^2$ (or equivalently
on $\RR^2$), so $g$ satisfies the equation
\[
  g_{rr} + \frac1r g_{r} + \frac1{r^2}g_{\theta\theta} = 0.
\]
Now define $G: \RR^2\to \RR$ by
\[
  G(t,\theta) = g(e^t, \theta).
\]
Then $G$ is harmonic in the usual sense: $G_{tt} + G_{\theta\theta}=0$.

By~\eqref{g-bound}, $G(t,\theta)-\theta$ is a bounded, entire harmonic function, and therefore is constant.
Also, $G-\theta$ vanishes where $\theta=0$, so it vanishes everywhere.  Thus
\[
  g(r,\theta)\equiv \theta,
\]
and therefore $\pdf{}{\theta}g \equiv 1$.

The smooth convergence of $s f_s$ to $g$ implies that
\[
  \lim_{r\to\infty} r \pdf{}{\theta}f(r,\theta) = 1,
\]
where the convergence is uniform given bounds on $\theta$.
Thus there is a $\rho<\infty$ such that for each $r\ge \rho$,
the function
\[
    \theta\in [-2\pi, 2\pi] \mapsto f(r,\theta)
\]
is strictly increasing.  Thus it has exactly one zero in this interval,
namely $\theta=0$.  But by the bounds~\eqref{new-bound}, $f(r,\theta)$ never
vanishes outside this interval.   Hence for $r\ge \rho$, $f(r,\theta)$
vanishes if and only if $\theta=0$.

So far we have only accounted for one component of $M\setminus Z(R)$.
But the behavior of the other component follows by $\rho_Z$ symmetry.
\end{proof}


\section{The proof of theorem~\ref{DIVERGING-RADII-THEOREM}}
 \label{Proof_of_DIVERGING-RADII-THEOREM}

We now prove theorem~\ref{DIVERGING-RADII-THEOREM} 
in section~\ref {section:main-theorems}.

\begin{proof}
Smooth convergence to a surface asymptotic to $H$ was proved in
theorems~\ref{R-to-infinity-theorem}
and~\ref{asymptotic-to-H-theorem}.
The other geometric properties of the surfaces $M_s$ in
statements~\eqref{2:helicoidlike} and \eqref{2:intersection-property} 
follow from the smooth convergence and 
the corresponding properties of the surfaces $M_s(R_n)$ in theorem~\ref{DIVERGING-RADII-THEOREM}.

Next we prove statement~\eqref{2:Y-surface-property}, i.e., that $M_s$ is a $Y$-surface.
The $\rho_Y$ invariance of $M_s$ follows immediately
from the smooth convergence and the $\rho_Y$ invariance of the $M_s(R_n)$.
We must also show that $\rho_Y$
acts on the first homology group of $M_s$ by multiplication by $-1$.  
Let $\gamma$  be a closed curve in $M_s$.  We must  show
that $\gamma\cup \rho_Y \gamma$ bounds a region of
$M^\infty _s$.  The curve $\gamma$ is approximated by curves
$\gamma_n \subset M_s(R_n)$. Since each $M_s(R_n)$ is a
$Y$-surface, $\gamma_n \cup \rho_Y \gamma_n$ bounds a compact
region in $W_n\subset M_s(R_n)$. These regions converge uniformly
on compact sets to a region $W\subset M_s$ with boundary
$\gamma \cup \rho_Y \gamma$ but a priori that region might not be compact.
By the maximum principle, each $W_n$ is contained in the smallest slab
of the form $\{|z|\le a\}$ containing $\gamma_n\cup\rho_Y\gamma_n$.
Thus $W$ is also contained in such a slab.  Hence $W$ is compact, since
 $M_s$ contains only one end and that end is helicoidal (and therefore is not
 contained in a slab.)  This completes the proof of 
 statement~\eqref{2:Y-surface-property}.


Next we prove statement~\eqref{2:point-count}:
$\|M_s\cap Y\| = 2\, \|M_s\cap Y^+\| +1 = 2\,\genus(M_s)+1$,
where $\|\cdot\|$ denotes the number of points in a set.
Because
$M_s$ is a $Y$-surface,  
\[
   \|  M_s \cap Y \| =2 -\chi(M_s)
\]
by proposition~\ref{Y-surface-topology-propostion}\eqref{exactly-fixed-points}.
Also,
\[
   \chi(M_s)= 2-2\genus(M_s)-1
\]
since $M_s$ has exactly one end.  Combining the
last two identities gives statement~\eqref{2:point-count}.
(Note that $M_s\cap Y$ consists of the points of $M_s\cap Y^+$,
the corresponding points in $M_s\cap Y^-$, and the origin.)

Statement~\eqref{2:genus-bound} gives bounds on the genus of $M_+$ and of $M_-$
 depending on the parity of $g$.  To prove these bounds, note
that $M_{+}(R_n)\cap Y^+$ contains exactly $g$ points.
By passing to a subsequence, we can assume (as $n\to\infty$)
that $a$ of those points stay a bounded distance from $Z$,
that $b$ of those points stay a bounded distance from $Z^*$, and
that for each of the remaining $g-a-b$ points, the distance from the point to $Z\cup Z^*$
  tends to infinity.

By smooth convergence,
\[
  \|M_{+}\cap Y^+\| = a.
\]
so the genus of $M_{+}$ is $a$ by statement~\eqref{2:point-count}.

If $g$ is even, then $M_{+}(R_n)$ is symmetric by reflection $\mu_E$ in the totally
geodesic cylinder $E\times \RR$ of points equidistant from $Z$ and $Z^*$.
It follows that $a=b$, so $\genus(M_{+})\le a \le g/2$.
The proof for $M_{-}$ and $g$ even is identical.

If $g$ is odd, then $M_{-}(R_n)$ is obtained from $M_{+}(R_n)$ by the reflection $\mu_E$.
 Hence exactly $b$ of the points
of $M_{-}(R_n)\cap Y^+$ stay a bounded distance from $Z$.
It follows that $M_{-}\cap Y^+$ has exactly $b$ points, and therefore
that $M_{-}$ has genus $b$.  Hence
\[
 \genus(M_{+}) + \genus(M_{-}) = a + b \le g,
\]
which completes the proof of statement~\eqref{2:genus-bound}.

It remain only to prove statement~\eqref{2:genus-parity}: the genus of $M_{+}$
is even and the genus of $M_{-}$ is odd. 
By statement~\eqref{2:point-count}, this is equivalent to showing that $\|M_s\cap Y^+\|$ is even or odd
according to whether $s$ is $+$ or $-$.
Let
\[
   S = S_s = M_s\cap H^+.
\]
Then $M_s\cap Y^+= S\cap Y$, so it suffices to show that $\|S\cap Y\|$
is even or odd according to whether $s$ is $+1$ or $-1$.
By proposition~\ref{Y-surface-topology-propostion}, 
this is equivalent to showing that $S$ has two ends
if $s$ is $+$ and one end if $s$ is $-$.

Let $Z(R)$ be the solid cylinder of radius $R$ about $Z$.  
By~corollary~\ref{form-of-graph-corollary} (see also the first three paragraphs 
of the proof of Lemma~\ref{z=0-level-lemma}), we can choose
$R$ sufficiently large that $S\setminus Z(R)$ has two components.
One component is a multigraph on which $\theta$ goes from $\theta(X^+)$ to $\infty$, and on which $z$ is unbounded above.  The other component is a multigraph
on which $\theta$ goes from $\theta(X^-)$ to $-\infty$ and on which $z$ is unbounded
below.    In particular, if we remove a sufficiently large finite solid cylinder
\[
   C: = Z(R) \cap \{|z|\le A\}  
\]
from $S$, then the resulting surface $S\setminus C$ has two components.
(We choose $A$ large enough that $C$ contains all points of $H\setminus Z$
at which the tangent plane is vertical.)
One component has in its closure $X^+\setminus C$ and $Z^+\setminus C$,
and the other component has in its closure $X^-\setminus C$ 
and $Z^-\setminus C$.   Consequently $Z^+$ and $X^+$ belong to an end
of $S$, and $X^-$ and $Z^-$ also belong to an end of $S$.

Thus $S$ has one or two ends according to whether $Z^+$ and $X^-$ belong
to the same end of $S$ or different ends of $S$.  Note that they belong to the same
end of $S$ if and only if every neighborhood of $O$ contains a path in $S$
with one endpoint in $Z^+$ and the other endpoint in $X^-$, i.e., if and only
if $M_s$ is negative at $O$.  
By the smooth convergence, $M_s$ is negative at $O$ if and only if the $M_s(R_n)$
are negative at $O$, i.e., if and only if $s=-$.
\end{proof}


\section{Minimal surfaces in a thin hemispherical shell}\label{hemisphere-appendix}

If $D$ is a hemisphere of radius $R$ and $J$ is an interval with 
 $|J|/R$ is sufficiently small, we show that the projection of a minimal surface $M\subset D\times J$ with 
$\partial M\subset \partial D\times J$ onto $D$ covers most of $D$. This result is used in the proof of
theorem~\ref{first-compactness-theorem} to show that a certain
multiplicity is $2$, not $0$.

\begin{theorem}\label{all-or-nothing-theorem}
Let $D=D_R$ be an open hemisphere in the sphere $\sS^2(R)$
and let $J$ be an interval with length $|J|$.  If $|J|/R$ is sufficiently small,
the following holds.
Suppose that $M$
is a nonempty minimal surface in $D\times J$ with $\partial M$ in $(\partial D)\times J$.
Then every point $p\in D$ is within distance $4 |J|$ of $\Pi(M)$,
where $\Pi: D\times J \to D$ is the projection map.
\end{theorem}

\begin{proof}
By scaling, it suffices to prove that if $J=[0,1]$ is an interval of length $1$, then every point
$p$ of $D_R$ is within distance $4$ of $\Pi(M)$ provided $R$ is sufficiently large.

For the Euclidean cylinder $\BB^2(0,2)\times [0,1]$, 
the ratio of the area of the cylindrical side (namely $4\pi$) to the sum of the areas
of top and bottom (namely $8\pi$) is less than $1$.  Thus the same is true
for all such cylinders of radius $2$ and height $1$ in $\sS^2(R)\times \RR$, provided $R$ is sufficiently large.
It follows that there is a catenoid in $\sS^2(R)\times \RR$ whose boundary consists of
the two circular edges of the cylinder.

Let $\Delta\subset D_R$ be a disk of radius $2$ containing the point $p$.
By the previous paragraph, we may suppose that $R$ is sufficiently large
that there is a catenoid $C$ in $\Delta\times[0,1]$ whose boundary consists of the two circular
edges of $\Delta\times [0,1]$.  In other words, $\partial C= (\partial \Delta)\times(\partial [0,1])$.

Now suppose $\dist(p, \Pi(M))> 4$.  Since $\Delta$ has diameter $4$,
it follows that $\Pi(M)$ is disjoint from $\Delta$ and therefore that $M$ is disjoint from $C$.
 If we slide $C$ around in $D_R\times [0,1]$, it can never
bump into $M$ by the maximum principle.  That is, $M$ is disjoint from the union $K$ of
all the catenoids in $D_R\times [0,1]$ that are congruent to $C$.  Consider all surfaces
of rotation with boundary $(\partial D_R)\times (\partial [0,1])$ that are disjoint from $M$ and from $K$.
The surface $S$ of least area in that collection is a catenoid, and (because it is disjoint from $K$)
it lies with distance $2$ of $(\partial D_R)\times [0,1]$.

We have shown: if $R$ is sufficiently large and if the theorem is false for $D_R$ and $J=[0,1]$,
then there is a catenoid $S$ such that $S$ and $\Sigma:=(\partial D_R)\times [0,1]$ have the same boundary
and such that $S$ lies within distance $2$ of $\Sigma$.

Thus if the theorem were false, there would be a sequence of radii $R_n\to \infty$
and a sequence of catenoids $S_n$ in $D_{R_n}\times[0,1]$ such that
\begin{equation*}
   \partial S_n = \partial \Sigma_n
\end{equation*}
and such that $S_n$ lies within distance $2$ of $\Sigma_n$, where
\[
 \Sigma_n = (\partial D_{R_n})\times [0,1].
\]
We may use coordinates in which the origin is in $(\partial D_{R_n})\times \{0\}$.
As $n\to \infty$, the $\Sigma_n$ converge smoothly to an infinite flat strip $\Sigma=L\times[0,1]$ in $\RR^3$
(where $L$ is a straight line in $\RR^2$),
and the $S_n$ converges smoothly to a minimal surface $S$ such that (i) $\partial S=\partial \Sigma$,
(ii) $S$ has the translational invariance that $\Sigma$ does, and (iii) $S$ lies within a bounded
distance of $\Sigma$.   It follows that $S=\Sigma$.

Now both $\Sigma_n$ and $S_n$ are minimal surfaces.  
(Note that $\Sigma_n$ is minimal, and indeed totally
geodesic, since $\partial D$ is a great circle.)
Because $\Sigma_n$ and $S_n$ have the boundary and converge smoothly to the same limit $S$,
it follows that there is a nonzero Jacobi field $f$ on $S$ that vanishes at the boundary. Since $S$ is flat,
$f$ is in fact a harmonic function.
The common rotational symmetry of $\Sigma_n'$ and $S_n'$ implies that the Jacobi field is translationally
invariant along $S$, and thus that it achieves its maximum somewhere.  (Indeed, it attains its maximum
on the entire line $L\times\{1/2\}$.)  But then $f$ must be constant,
which is impossible since $f$ is nonzero and vanishes on $\partial S$.
\end{proof}


\section{Noncongruence Results}\label{noncongruence-appendix}

In this section, we prove the noncongruence results in theorem~\ref{main-S2xR-theorem}:
that $M_+$ and $M_-$ are not congruent by any orientation-preserving isometry of $\sS^2\times\RR$,
and that if the genus is even, they are not congruent by any isometry of $\sS^2\times\RR$.
For these results, we have to assume that $H$ does not have infinite pitch, i.e., that $H\ne X\times \RR$.
For simplicity, we consider only the non-periodic case ($h=\infty$).  The periodic case is similar.

Let $s\in \{+,-\}$.
Since $M_s\cap H= Z\cup Z^*\cup X$ and since every vertical line intersects $H$ infinitely many
times, we see that $Z$ and $Z^*$ are the only vertical lines contained in $M_s$. 

Let $C$ be a horizontal great circle in $M_s$ that intersects $Z$ and therefore also $Z^*$.
We claim that $C=X$.  To see this, first note that $C$ must lie in $\sS^2\times \{0\}$,
since otherwise $M_s$ would be invariant under the screw motion $\rho_X\circ \rho_C$
and would therefore have infinite genus, a contradiction.  It follows that $C$ contains $O$
and $O^*$.  But then $C$ must be $X$, since otherwise the tangents to $C$, $X$, and $Z$
at $O$ would be three linearly independent vectors all tangent to $M_s$ at $O$.

Now suppose $\phi$ is an isometry of $\sS^2\times \RR$ that maps $M_+$ to $M_-$.
We must show that the genus is odd and that $\phi$ is orientation reversing on $\sS^2\times\RR$.
By composing with $\rho_Y$ if necessary, we can assume that $\phi$ does not switch the two
ends of $\sS^2\times \RR$.
Also, by the discussion above, $\phi$ must map $Z\cup  Z^* \cup X$ to itself.

The symmetries of $Z\cup Z^*\cup X$ that do not switch up and down are:
$I$, 
$\mu_E$, 
$\mu_Y$ (reflection in $Y\times \RR$), and 
$\mu_E\circ \mu_Y$.

The ends of $M_+$ and $M_-$ are asymptotic to helicoids of positive pitch
with axes $Z$ and $Z^*$,
which implies that the ends of $\mu_Y(M_+)$ and of $(\mu_E\circ \mu_Y)(M_+)$
are asymptotic to helicoids of negative pitch.  Thus $\phi$ cannot be $\mu_Y$ or $\mu_E\circ \mu_Y$.
Since $M_+$ and $M_-$ have different signs at $O$, $\phi$ cannot be the identity.
Thus $\phi$ must be $\mu_E$, which is orientation reversing on $\sS^2\times\RR$.
 Since $M_-=\phi(M_+)$ and $M_+$ have different signs at $O$, 
 we see that $\phi(M_+)$ and $M_+$ are not equal, which implies
 (by~\eqref{even-genus-congruence-assertion} of theorem~\ref{main-S2xR-theorem}) 
 that the genus is odd. \QED

\bigskip

\section*{\bf Part~\TWO: Genus-\texorpdfstring{$g$}{Lg} Helicoids in \texorpdfstring{$\RR^3$}{Lg}}
\label{part2}

\bigskip

\renewcommand{\Im}{\operatorname{Im}}  
\renewcommand{\Re}{\operatorname{Re}}

\newcommand{\Res}{\operatorname{Res}}

\renewcommand{\div}{\operatorname{div}}

\def\wtu{\widetilde{u}}
\def\whM{\widehat{M}}
\def\whc{\widehat{c}}
\def\whf{\widehat{f}}
\def\whu{\widehat{u}}
\def\whp{\widehat{p}}
\def\wht{\widehat{t}}
\def\whU{\widehat{U}}
\def\whphi{\widehat{\phi}}
\def\whOmega{\widehat{\Omega}}

\def\wtC*{\widetilde{\CC^*}}

%

By theorem~\ref{main-S2xR-theorem} (\hyperref[part1]{part~\ONE}, section~\ref{section:main-theorems}),  
for every radius $R>0$ and positive integer $g$, there exist
two distinct helicoidal minimal surfaces of genus $2g$ in $\sS^2(R)\times\rR$,
with certain additional properties described below (see theorem~\ref{properties-theorem}).
We denote those surfaces
$M_+(R)$ and $M_-(R)$.
In \hyperref[part1]{part~\ONE},  the genus of $M_s(R)$ was denoted by $g$ and could be even or 
odd.  
In \hyperref[part2]{part~\TWO}, we
only use the examples of even genus, and it is more convenient to denote that genus by $2g$.

Let $R_n$ be a sequence of radii diverging to infinity.
By theorem~\ref{DIVERGING-RADII-THEOREM} (\hyperref[part1]{part~\ONE}, section~\ref{section:main-theorems}),
for each $s\in\{+,-\}$, a subsequence of $M_s(R_n)$ converges to a helicoidal
minimal surface
 $M_s$ in $\rR^3$.
The main result of \hyperref[part2]{part~\TWO} is the following:

\begin{THEOREM}
\label{PART2-MAIN-THEOREM}
If $g$ is even, $M_+$ has genus $g$. If $g$ is odd,  $M_-$ has genus $g$.
\end{THEOREM}
Thus $\RR^3$ contains genus-$g$ helicoids for every $g$.

\section{Preliminaries}
In this section, we recall the notations and results that we need from \hyperref[part1]{part~\ONE}.
Some notations are slightly different and better
suited to the arguments of \hyperref[part2]{part~\TWO}.

Our model for $\sS^2(R)$ is $\CC\cup\{\infty\}$ with the conformal metric obtained by stereographic projection:
\begin{equation}
\label{equation-metric-S2}
\lambda^2 |dz|^2 \quad\mbox{ with }
\lambda=\frac{2 R^2}{R^2+|z|^2},
\end{equation}
In this model, the equator is the circle $|z|=R$.
Our model for $\sS^2(R)\times\rR$ is $(\CC\cup\{\infty\})\times\rR$ with the 
metric 
\begin{equation}
\label{equation-metric-S2xR}
\lambda^2 |dz|^2+dt^2,\quad (z,t)\in(\CC\cup\{\infty\})\times\rR.
\end{equation}
When $R\to\infty$, this metric converges to the Euclidean metric 
$4|dz|^2+dt^2$ on $\CC\times\rR=\rR^3$. (This metric is isometric to
the standard Euclidean metric by the map $(z,t)\mapsto (2z,t)$.)

The real and imaginary axes in $\CC$ are denoted $X$ and $Y$.
The circle $|z|=R$ is denoted $E$ (the letter $E$ stands for ``equator").
Note that $X$, $Y$ and $E$ are geodesics for the metric \eqref{equation-metric-S2}.
We identify $\sS^2$ with $\sS^2\times\{0\}$, so $X$, $Y$ and
$E$ are horizontal geodesics in $\sS^2\times\rR$.
The antipodal points $(0,0)$ and $(\infty,0)$ are denoted $O$ and $O^*$, respectively.
The vertical axes through $O$ and $O^*$ in $\sS^2\times\rR$ are denoted $Z$ and
$Z^*$, respectively.
If $\gamma$ is a horizontal or vertical geodesic in $\sS^2\times\rR$, the $180^{\circ}$
rotation around $\gamma$ is denoted $\rho_{\gamma}$. This is an isometry of
$\sS^2\times\rR$.
The reflection in the vertical cylinder $E\times\rR$ is denoted $\mu_E$. This is an
isometry of $\sS^2\times\rR$. In our model,
\[
\mu_E(z,t)=\left(\frac{R^2}{\overline{z}},t\right).
\]

Let $H$ be the standard helicoid in $\rR^3$, defined by the equation
\[
x_2\cos x_3=x_1\sin x_3.
\]
It turns out that $H$ is minimal for the metric \eqref{equation-metric-S2xR} for any value of $R$, although not complete anymore (see section~\ref{part-1-preliminaries} in \hyperref[part1]{part~\ONE}).
We complete it by adding the vertical line $Z^*=\{\infty\}\times\rR$, and still denote it $H$.
This is a complete, genus zero, minimal surface in $\sS^2\times\rR$.
It contains the geodesic $X$, the axes $Z$ and $Z^*$ and
meets the geodesic $Y$ orthogonally at the points $O$ and $O^*$.
 It is invariant by
$\rho_X$, $\rho_Z$, $\rho_{Z^*}$, $\mu_E$ (which reverse its orientation) and
$\rho_Y$ (which preserves it).

In the following two theorems, we summarize what we need to know about
the genus-$2g$ minimal surfaces $M_+(R)$ and $M_-(R)$ in $\sS^2(R)\times\rR$
(see  theorems~\ref{main-S2xR-theorem} 
and \ref{DIVERGING-RADII-THEOREM} in section~\ref{section:main-theorems} of \hyperref[part1]{part~\ONE}):
\begin{THEOREM}
\label{properties-theorem}
Let $s\in\{+,-\}$. Then:
\begin{enumerate}[\upshape (1)]
\item $M_s(R)$ is a complete, properly embedded  minimal surface with genus $2g$, and  it 
has a top end and a bottom end, each asymptotic to $H$ or a vertical translate of $H$,
\item $M_s(R)\cap H=X\cup Z\cup Z^*$. In particular, $M_s(R)$ is invariant by $\rho_X$,
$\rho_Z$ and $\rho_{Z^*}$, each of which reverses its orientation.
\item $M_s(R)$ is invariant by the reflection $\mu_E$, which reverses its orientation,
\item $M_s(R)$ meets the
geodesic $Y$
orthogonally at $4g+2$ points and is invariant under $\rho_Y$, which preserves its orientation. Moreover, 
$(\rho_Y)_*$ acts on $H_1(M_s(R),\ZZ)$ by multiplication by $-1$.
\end{enumerate}
\end{THEOREM}

\begin{THEOREM}
\label{PART2-DIVERGING-RADII}
Let $s\in\{+,-\}$. Let $R_n$ be a sequence of radii diverging to infinity.
Let $M_s(R_n)$ be a surface having the properties
 listed in theorem~\ref{properties-theorem}.
Then a subsequence of $M_s(R_n)$ converges to a minimal surface $M_s$ in $\rR^3$
asymptotic to the helicoid $H$. The convergence is smooth convergence on compact
sets.
Moreover,
\begin{itemize}
\item the genus of $M_s$ is at most $g$,
\item the genus of $M_+$ is even,
\item the genus of $M_-$ is odd,
\item the number of points in $M_s\cap Y$ is $2\,\text{\rm genus}(M_s)+1$.
\end{itemize}
\end{THEOREM}

%

\section{The setup}
\label{section-setup}
Let $R_n \to  \infty$ be a sequence such that $M_s(R_n)$ converges smoothly to a limit $M_s$
as in theorem~\ref{PART2-DIVERGING-RADII}.
Let $g'$ be the
genus of $M_s$.
By the last point of theorem~\ref{PART2-DIVERGING-RADII},
$M_s\cap Y$ has exactly $2g'+1$ points.
It follows that
$2g'+1$ points of $M_s(R_n)\cap Y$ stay at bounded distance from the origin $O$.
By $\mu_E$-symmetry, $2g'+1$ points of
$M_s(R_n)\cap Y$ stay at bounded distance from the antipodal point $O^*$.
 There remains $4(g-g')$ points in $M_s(R_n)\cap Y$ whose distance to $O$ and $O^*$ is 
 unbounded. Let
 \[
 N=g-g'.
 \]

 We shall prove
 \begin{theorem}
 \label{maintheorem-N}
 In the above setup, $N\leq 1$.
 \end{theorem}
 Theorem~\ref{PART2-MAIN-THEOREM} is a straightforward consequence of this theorem:
 Indeed if $g$ is even and $s=+$, we know by theorem~\ref{PART2-DIVERGING-RADII} that
 $g'$ is even so $N=0$ and $g=g'$. If $g$ is odd and $s=-$, then $g'$ is odd
 so again $N=0$.



 To prove theorem~\ref{maintheorem-N},
 assume that $N\geq 1$. We want to prove that $N=1$ by studying the $4N$ points whose
 distance to $O$ and $O^*$ is unbounded.
To do this, it is necessary to work on a different scale. 
Fix a sign $s\in \{+,-\}$ and define
 \[
 M_n=\frac{1}{R_n} M_s(R_n)\subset \sS^2(1)\times\rR.  
 \]
 This is a minimal surface in $\sS^2(1)\times\rR$.
 Each end of $M_n$ is asymptotic to a vertical translate of a helicoid of pitch
\[
t_n=\frac{2\pi}{R_n}.
\]
(The pitch of a helicoid with counterclockwise rotation is twice the distance between consecutive sheets.
The standard helicoid has pitch $2\pi$.)
Observe that $t_n\to 0$.
By the definition of $N$, the intersection $M_n\cap Y$ has $4N$ points whose distance to $O$ and $O^*$ is $\gg t_n$.
 Because $M_n$ is symmetric with respect to $180^{\circ}$ rotation $\rho_X$ around $X$, there are $2N$ points on the positive $Y$-axis.
 We order these by increasing
 imaginary part:
 \[
 p'_{1,n}, p''_{1,n}, p'_{2,n}, p''_{2,n}, \dots,p'_{N,n}, p''_{N,n}.
 \]
 Because of the $\rho_X$-symmetry,
 the $2N$ points on the negative $Y$-axis are the conjugates of these points.
 Define $p_{j,n}$ to be the midpoint of the interval $[p'_{j,n},p''_{j,n}]$ and
 $r_{j,n}$ to be half the distance from $p'_{j,n}$ to
 $p''_{j,n}$, both with respect to the spherical metric.
   We have
  \[
  0 < \Im p_{1,n} < \Im p_{2,n} < \dots < \Im p_{N,n}.
  \]
By $\mu_E$-symmetry, which corresponds to inversion in the unit circle,
  \begin{equation}
  \label{equation**}
  p_{N+1-i,n}=\frac{1}{\overline{p_{i,n}}}.
  \end{equation}
  In particular, in case $N$ is odd, $p_{\frac{N+1}{2},n}=i$.
  


\begin{figure}
 \begin{center}
 \includegraphics[height=35mm]{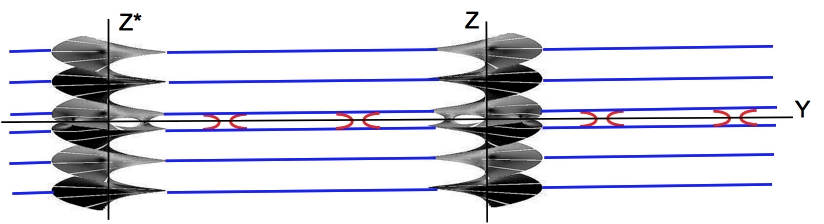}
 \end{center}
 \caption{
 Here is a schematic illustration of what $M_s(R)$ could look like for large values of $R$
($g=4$ and $s=+$ in this picture).
The distance between the two vertical axes $Z$ and $Z^*$ is $\pi R$ and should be
thought of as being very large.
Inside a vertical cylinder of large but fixed radius around the $Z$-axis,
the surface is very close to the limit helicoidal surface $M_s$ (a genus-$2$
helicoid in this picture). By $\mu_E$-symmetry, the same happens around the
$Z^*$-axis. Inside these two cylinders, we see the handles that stay
at bounded distance from the axes as $R\to\infty$.
Outside the cylinders, the surface is close to the helicoid (represented
schematically by horizontal lines) whose two sheets
are connected by $2N$ small necks placed along the $Y$-axis ($N=2$ in this picture).
The distance of each neck to the axes $Z$ and $Z^*$ is diverging
as $R\to\infty$. These are the handles that are escaping from both $Z$
and $Z^*$. These necks are getting smaller and smaller as $R\to\infty$
and have asymptotically catenoidal shape.
Note that the handles that stay at bounded distance from the axes do not
converge to catenoids as $R\to\infty$. (They do look like catenoids in
this picture.)
}
\end{figure}

  For $\lambda>1$ sufficiently large,  let  $\mathcal Z_n(\lambda)$ be the part of $M_n$ lying inside of the vertical cylinders  of radius $\lambda t_n$ around $Z$ and $Z^*$:
 \begin{equation}\mathcal Z_n(\lambda)=\{q=(z,t)\in M_n\,:\, d(Z\cup Z^*, q)<\lambda t_n\}.
\end{equation}
Also define
 $D_{j,n}(\lambda) =\{z\,:\, d(z, p_{j,n})<\lambda r_{j,n}\}$. Consider the intersection of $M_n$ with the vertical cylinder over $D_{j,n}(\lambda)$, and let $C_{
j,n}(\lambda)$ denote the component of this intersection that contains
 the points $\{p'_{j,n}, p''_{j,n}\}$.
Define
\begin{equation}\label{All_Necks_1}
\mathcal C_n(\lambda)=\bigcup_{j=1}^N C_{j,n}(\lambda)\cup \overline{C_{j,n}(\lambda)}.
\end{equation}
The following proposition is key in showing that at most one handle is lost in taking the limit as $R_n\rightarrow\infty$. In broad terms, it says that  near the points $p_{j,n}$, catenoidal necks are forming on a small scale, and after removing these necks and a neighborhood of the axes,  what is left is a pair of symmetric surfaces which are vertical graphs over a half-helicoid.

 \begin{proposition}\label{Setup-proposition} Let  $N=g-g'$, $t_n$,  
 and  $M_n \subset \sS^2(1)\times \rR$  be as above.
Then
\begin{enumerate}[\upshape i.]
\item  For each $j$, $1 \leq j \leq N$,  the surface $\frac{1}{r_{j,n}}(M_n-p_{j,n})$ converges to  the standard  catenoid $\mathbf C$ with vertical
 axis  and waist circle of radius $1$ in $\rR^3$. In particular, the distance (in the spherical metric) $d(p_{j,n},p_{j+1,n})$ is $\gg r_{j,n}$. Moreover, $t_n\gg r_{j,n}$, and the $C_{j,n}(\lambda)$ are close to catenoidal necks with collapsing radii.        
 \item  Given $\epsilon>0$, there exists a $\lambda >0$ such that
  \[
  M_n'=M_n\setminus (\mathcal Z_n(\lambda)\cup\mathcal C_n(\lambda))
  \]
  has the following properties:
                 \begin{enumerate}[\upshape (a)]
                         \item The slope of the tangent plane at any point of $M'_n$ is less than $\epsilon$.
                         \item $M'_n$ consists of two components related by
 the symmetry $\rho _Y$, rotation by $180^\circ$ around $Y$.
                        \item $M'_n$ intersects $t_nH$ in a subset of the axis $X$ and nowhere else, with one of its components intersecting in a ray of the positive $X$-axis, the other in a ray of $X^-$. Each component is graphical over its projection onto the half-helicoid (a component of $t_nH\setminus (Z\cup Z^*)$) that it intersects.
\end{enumerate}
\end{enumerate}
 \end{proposition}
 This proposition is proved in theorem~\ref{limit-is-a-catenoid-theorem} 
 and 
 corollary~\ref{form-of-graph-corollary} of \hyperref[part1]{part~\ONE}  
 (with slightly different notation).

Passing to a subsequence, $p_j=\lim p_{j,n}\in i\rR^+\cup\{\infty\}$
exists for all $j\in[1,N]$. We have $p_1\in[0,i]$, and we will consider the following three cases:
\begin{equation}
\label{the-3-cases}  
\begin{array}{ll}
\bullet & \text{Case 1: } p_1\in(0,i),\\
\bullet & \text{Case 2: } p_1=0,\\
\bullet & \text{Case 3: }  p_1=i.
\end{array}
\end{equation}
We will see that Case 1 and Case 2 are impossible, and that $N=1$ in Case 3.

\stepcounter{theorem}
\addtocontents{toc}{\SkipTocEntry}
\subsection{The physics behind the proof of theorem~\ref{maintheorem-N}}
\label{section-physics}
Theorem~\ref{maintheorem-N} is proved by evaluating the surface tension in the
$Y$-direction on each catenoidal neck.
Mathematically speaking, this means the flux of the horizontal Killing field tangent to the 
$Y$-circle in $\sS^2\times\rR$.
On one hand, this flux vanishes at each neck by $\rho_Y$-symmetry
(see Lemma~\ref{flux-zero-lemma}).
On the other hand, we can compute the limit $F_i$ of the surface tension on the $i$-th catenoidal neck (corresponding to $p_i=\lim p_{i,n}$) as $n\to\infty$, after suitable scaling.


Assume for simplicity that the points $O$, $p_1, \dots, p_N$ and $O^*$ are
distinct. Recall that the points $p_1,\dots p_N$ are on the positive imaginary $Y$-axis.
For $1\leq j\leq N$, let $p_j=i y_j$, with $0<y_j<\infty$.
Then we will compute that
\[
F_i=c_i^2\frac{1-y_i^2}{1+y_i^2}+\sum_{j\neq i} c_i c_j f(y_i,y_j)
\]
where the numbers $c_i$ are positive and proportional to the size of
the catenoidal necks and
\[
f(x,y)=\frac{-2\pi^2}{(\log x-\log y)|\log x-\log y+i\pi|^2}.
\]
Observe that $f$ is antisymmetric and $f(x,y)>0$ when $0<x<y$.
We can think of the point $p_i$ as a particle with mass $c_i$ and interpret $F_i$ as a force of gravitation type.
The particles $p_1,\dots,p_N$ are attracted to each other
and we can interpret the first term by saying that each particle
$p_i$ is repelled from the fixed antipodal points $O$ and $O^*$.
All forces $F_i$ must vanish. It is physically clear that no equilibrium is possible unless
$N=1$ and $p_1=i$. Indeed in any other case, $F_1>0$.


This strategy is similar to the one followed in \cite{traizet-balancing} and \cite{traizet-convex}.
The main technical difficulty is that we cannot guarantee that the 
points $O$, $p_1,\dots, p_N$ and $O^*$ are distinct. 
The distinction between Cases 1, 2 and 3 in \eqref{the-3-cases} stems from this problem.

\stepcounter{theorem}
\addtocontents{toc}{\SkipTocEntry}
\subsection{The space \texorpdfstring{$\wtC*$}{Lg}}
To compute forces we need to express $M_n$ as a graph.
For this, we need to express the helicoid itself as a graph, away from its axes $Z$ and
$Z^*$.
Let $\wtC*$ be the universal cover of $\CC^*$. Of course, one can identify $\wtC*$ with $\CC$ by mean of the exponential function. It will be more convenient to see $\wtC*$ as the covering space obtained by analytical continuation of $\log z$, so each point of $\wtC*$ is a point of $\CC^*$ together with a determination of its argument : points are couples $(z,\arg(z))$,
although in general we just write $z$.
The following two involutions of $\wtC*$ will be of interest:
\begin{itemize}
\item $(z,\arg(z))\mapsto (\overline{z},-\arg(z))$, which we write simply as $z\mapsto \overline{z}$. The fixed points are $\arg z=0$.
\item $(z,\arg(z))\mapsto (1/\overline{z},\arg(z))$, which we write simply as
$z\mapsto 1/\overline{z}$. The fixed points are $|z|=1$.
\end{itemize}
The graph of the function
$\frac{t}{2\pi}\arg z$ on $\wtC*$ is one half of a helicoid of pitch $t$.

\stepcounter{theorem}
\addtocontents{toc}{\SkipTocEntry}
\subsection{The domain \texorpdfstring{$\Omega_n$}{Lg} 
                and the functions \texorpdfstring{$f_n$}{Lg} and \texorpdfstring{$u_n$}{Lg}}
  By proposition~\ref{Setup-proposition}, away from the axes $Z\cup Z^*$ and the points
  $p_{j,n}$, we  may consider
   $M_n$ to be the union of two  multigraphs. We wish to express  this part of $M_n$ 
as a  pair of graphs over a subdomain of $\wtC*$.
We will allow ourselves the freedom to write  $z$ for a point $(z,\arg z)\in \wtC*$ when its argument is clear from the context. Thus we will write $p_{j,
n}$ for the point $(p_{j,n}, \pi/2)$ in $\wtC*$ corresponding to the points on $M_n\cap Y
$ in proposition~\ref{Setup-proposition}.  Define
\begin{equation}\label{D_n_Lambda}
D_n(\lambda)=\{\,(z,\arg z)\,:\, |z|<\lambda t_n \text{ or } |z|>\frac{1}{\lambda t_n}\},
\end{equation}
\begin{equation}\label{D_j_n}
D_{j,n}(\lambda)=\{\,(z,\arg z)\,:\, d(p_{j,n}, z)<\lambda r_{j,n}\text{ and } 0<\arg z <\pi\}
\end{equation}
and
\begin{equation}\label{Domain_1}
\Omega_n=\Omega_n(\lambda)=\wtC*\setminus \left(D_n(\lambda) \cup \bigcup_{j=1}^{N}D_{j,n}(\lambda)\cup\overline{D_{j,n}(\lambda)}\right ).
\end{equation}
According to statement~$ii.$ of proposition~\ref{Setup-proposition},  there exists a $\lambda>0$
such that for sufficiently large $n$,   
\[M'_n=M_n\cap (\Omega_n(\lambda)\times\rR)
\]
  is the union of two graphs related by $\rho_Y$-symmetry, and each graph intersects  the helicoid of pitch $t_n$ in a subset of the  $X$-axis. Only one of these graphs can contain points on the positive $X$-axis. We choose this component and write it
 as the graph of a function $f_n$ on the domain $\Omega_n$. We may write
 \begin{equation}\label{eq:def-of-u_n}  
 f_n(z)=\frac{t_n}{2\pi}\arg z-u_n(z).
 \end{equation}
The function $u_n$ has the following properties:
\begin{equation}
\label{eq:u_n-properties}
\begin{array}{ll}
\bullet & \text{ $u_n(\overline z) =-u_n(z)$. In particular, $u_n=0 $ on $\arg z=0$.}\\
\bullet & \text{$u_n(1/\overline z) =u_n(z)$ In particular, $\partial u_n/\partial \nu =0$ on $|z|=1$.}\\
\bullet & \text{ $0<u_n<t_n/2$ when $\arg z >0$.}
\end{array}
\end{equation}
The first two assertions follow from the symmetries of $M_n$. 
See theorem~\ref{properties-theorem} (statements 2 and 3),  and the discussion preceding it.  The third assertion follows from proposition~\ref{Setup-proposition}, statement ~$ii.c$, which implies that
\[
0<|u_n|<t_n/2
\]
  when $\arg z>0$, since the vertical distance between the sheets of $t_nH$ is equal to $t_n/2$.
Now choose a point $z_0$ in the domain of $f_n$  that is near a point $p_{j,n}$. Then $|f_n(z_0)|$ is small,
and $\arg z_0$ is near $\pi/2$. Hence $f_n(z_0)\sim t_n/4 -u_n(z_0)$, which implies that $u_n(z_0)>0$. We conclude that
$0<u_n<t_n/2$ when $\arg z >0$, as claimed.

\begin{figure}
\label{figure-domain}
\begin{center}
\includegraphics[width=60mm]{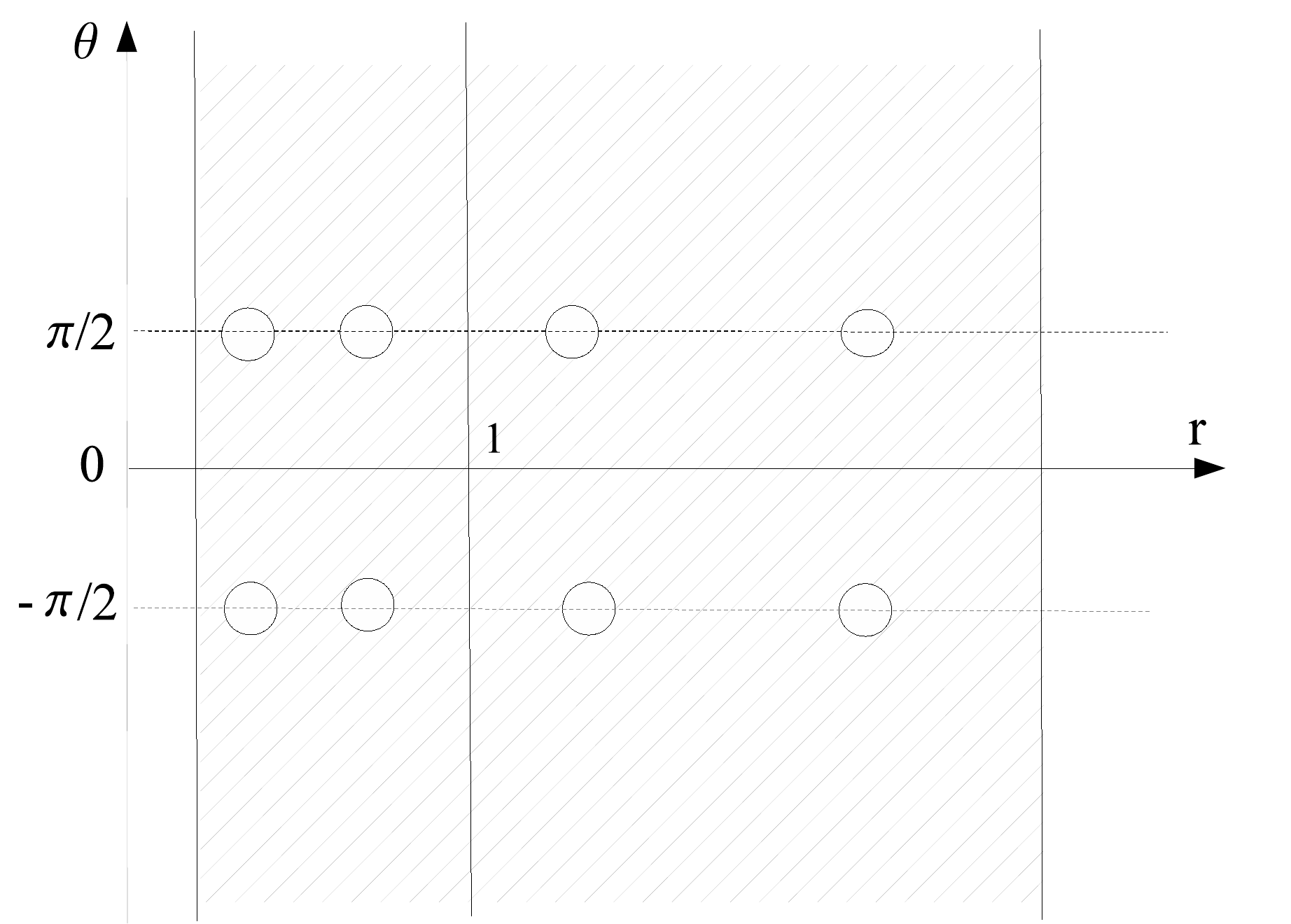}
\end{center}
\caption{The domain $\Omega_n$ in polar coordinates, $z=re^{i\theta}$.
The function $u_n$ is positive for $\theta>0$. 
 The line $r=1$ corresponds to the unit circle $|z|=1$. The white strip on the left corresponds to the projection of the vertical cylinder of radius $\lambda t_n$ about the $Z$-axis, and the region to the right of the shaded domain is its image by the inversion through the unit circle. The small disks correspond to the vertical cylinders of radius $\lambda r_{j,n}$ (in the spherical metric).
}
\end{figure}

\addtocontents{toc}{\SkipTocEntry}
\subsection*{Organization of \hyperref[part2]{part~\TWO}}
We deal with Cases 1, 2 and 3, as listed in \eqref{the-3-cases}, separately. In each case, we first state, without proof, a proposition
which describes the asymptotic behavior of the function $u_n$ defined by \eqref{eq:def-of-u_n} as $n\to\infty$.
We use this result to compute forces and obtain the required result (namely, $N=1$
or a contradiction). Then, we prove the proposition.
Finally, an appendix contains analytic and geometric results that are relevant to minimal surfaces
 in $\sS^2\times\rR$ and that are used in \hyperref[part2]{part~\TWO}.
\section{Case 1: \texorpdfstring{$p_1\in (0,i)$}{Lg}}  
\label{case-1-section}
For $p\in\wtC*$, let $h_p$ be the harmonic function defined on  
$\wtC*\setminus\{p,\overline{p}\}$ by
\begin{equation*}
h_p(z)=-\log\left|\frac{\log z-\log p}
{\log z-\log\overline{p}}\right|.
\end{equation*}
Note that since $p$ and $z$ are in $\wtC*$, both come with a determination of their logarithm, so the function $h_p$ is well defined.
This function has the same symmetries as $u_n$:
\begin{equation}
\label{eq:h_p-symmetries}
\begin{array}{ll}
\bullet & h_p(\overline{z})=-h_p(z),\\
\bullet & h_p(1/\overline{z})=h_{1/\overline{p}}(z).\\
\bullet &\text{Moreover, if $\arg p$ and $\arg z$ are positive then $h_p(z)>0$.}
\end{array}
\end{equation}

\begin{remark*}
The function $(z,p)\mapsto -h_p(z)$ is the Green's function for the domain $\arg z>0$ of $\wtC*$.
\end{remark*}
Recall that 
\begin{equation}\label{eq:limits}
p_i=\lim_n p_{i,n}. 
\end{equation}
It might happen that several points $p_j$ are equal to $p_i$.
In this case, we say that we have a cluster at $p_i$.
Let $m$ be the number of distinct points amongst $p_1, \dots, p_N$.
For each $n$, relabel the points $p_{i,n}$ 
(by permuting the indices)
so that the points $p_1,\dots,p_m$ defined by~\eqref{eq:limits}
are distinct and so that
\[
\Im p_1 < \Im p_2 < \dots < \Im p_m.
\]
(Consequently, for each $j$ with $1\le j\le N$, there is exactly
one $i$ with $1\le i\le m$ such that $p_j=p_i$.)

We define
\begin{equation}
\label{eq:def-of-tilde-u_n}
\wtu_n=
\frac{|\log t_n|}{t_n} u_n.
\end{equation}
\begin{proposition}
\label{first-case-1-proposition}
Assume that $p_1\neq 0$. Then, after passing to a subsequence, 
there exist non-negative real numbers
    $c_0,\dots,c_m$ such that
\begin{equation}
\label{equation-limit-un}
\wtu(z):=\lim \wtu_n(z)=c_0\arg z+\sum_{i=1}^m
c_ih_{p_i}(z).
\end{equation}
The convergence is the usual smooth uniform convergence on compact subsets of $\wtC*$ minus the points
$p_i$, $-p_i$ for $1\leq i\leq m$.
Moreover, for $1\leq i\leq m$,
\begin{equation}
\label{equation-ci}
c_i=\lim\frac{|\log t_n|}{t_n}\frac{\phi_{i,n}}{2\pi}
\end{equation}
where $\phi_{i,n}$ is the vertical flux of $M_n$ on
the graph of $f_n$ restricted to the circle $C(p_i,\varepsilon)$ for a fixed, small enough $\varepsilon$.
\end{proposition}
We allow $p_1=i$ as this proposition will be used in Case 3, section \ref{case-3-section}.

Note that for large $n$, $\phi_{i,n}$ is the sum of the vertical fluxes on the catenoidal necks
corresponding to the points $p_{j,n}$ such that $p_j=p_i$.

This proposition is proved in section~\ref{section:proof-of-first-case-1-proposition} 
by estimating the Laplacian of $u_n$ and
constructing an explicit barrier, from which we deduce that a subsequence converges to a limit harmonic function on $\wtC*$ with logarithmic singularities at $\pm p_1, \dots, \pm p_m$.

\begin{remark*}
In proposition~\ref{first-case-1-proposition},  it is easy to show using Harnack's inequality
that we can choose numbers
$\lambda_n>0$ so that $\lambda_n u_n$ converges subsequentially to a nonzero limit of the form \eqref{equation-limit-un}.
(One fixes a point $z_0$ and lets $\lambda_n= 1/u_n(z_0)$.)   However, for us it is crucial 
that we can choose $\lambda_n$ to be $\frac{|\log t_n|}{t_n}$; it means that in later calculations, we will be able to ignore terms that are $o(\frac{|\log t_n|}{t_n})$.
\end{remark*}

For all we know at this point, the limit $\wtu$ might be zero.
We will prove this is not the case:
\begin{proposition}
\label{second-case-1-proposition} 
For each $i\in[1,m]$, $c_i>0$.
\end{proposition}
This proposition is proved  in section \ref{section:proof-of-second-case-1-proposition} using a height estimate
(proposition~\ref{proposition-height}) to estimate the vertical flux of the catenoidal necks. 


From now on assume that $p_1\in (0,i)$.
Fix some small number $\varepsilon$.
Let $C_n$ be the graph of the restriction of $f_n$ to the circle $C(p_1,\varepsilon)$.
Let $F_n$ be the flux of the Killing field
  $\chi_Y$ on $C_n$.
The field $\chi_Y$ is the Killing field associated with rotations with respect to poles
whose equator is the $Y$-circle
(see proposition~\ref{proposition-flux1} in the appendix.)
On one hand, we have:
\begin{lemma}
\label{flux-zero-lemma}      
$F_n=0$.
\end{lemma}
\begin{proof}
By theorem~\ref{properties-theorem}, statement (4),
$C_n$ together with its image $\rho_Y (C_n)$
bound a compact region in $M_n$.  Thus the flux of the Killing field $\chi_Y$ on 
$C_n \cup \rho_Y (C_n)$ is $0$.   By $\rho_Y$-symmetry,
this flux is twice the flux $F_n$ of $\chi_Y$ on $C_n$.
Thus $F_n=0$.
\end{proof}


On the other
hand, $F_n$ can be computed using proposition~\ref{proposition-flux2} 
from the appendix:
\begin{align}
F_n&=-\Im\int_{C(p_1,\varepsilon)}
2\left(\frac{\partial}{\partial z}\left(\frac{t_n}{2\pi}\arg z-u_n\right)\right)^2\frac{i}{2}(1-z^2)\,dz
     +O(t_n^4)\nonumber\\
&=-\Re\int_{C(p_1,\varepsilon)}
\left(\frac{t_n}{4\pi i z} - u_{n,z}\right)^2 (1-z^2)\, dz+O(t_n^4)\nonumber\\
&=-\Re\int_{C(p_1,\varepsilon)} \left( \frac{-t_n^2}{16\pi^2 z^2}
-\frac{2 t_n}{4\pi i z} u_{n,z}+(u_{n,z})^2\right)(1-z^2)\,dz+O(t_n^4)\nonumber\\
\label{eq-integral}
&=\Re\int_{C(p_1,\varepsilon)} \left( 
\frac{2 t_n}{4\pi i z} u_{n,z}-(u_{n,z})^2\right)(1-z^2)\,dz+O(t_n^4).
\end{align}
The second equation comes from $\frac{\partial}{\partial z}\arg z=\frac{1}{2iz}$.
The fourth equation is a consequence of the fact that $\frac{1-z^2}{z^2}$ has no residue at
$p_1\neq 0$.
The first term in \eqref{eq-integral} (the cross-product) is a priori the leading term.
However we can prove that this term can be neglected:

\begin{proposition}
\label{third-case-1-proposition}
\begin{equation}
\label{equation*9}
\lim\left(\frac{\log t_n}{t_n}\right)^2 F_n=-\Re\int_{C(p_1,\varepsilon)} (\wtu_z)^2(1-z^2)\,dz
\end{equation}
where $\wtu$ is defined in  \eqref{equation-limit-un} as the limit of $\frac{|\log t_n|}{t_n}u_n$.
\end{proposition}
This proposition is proved in section~\ref{section:proof-of-third-case-1-proposition} using a Laurent series expansion to estimate the first term in \eqref{eq-integral}.

Assuming these results, we now prove
\begin{proposition}
\label{proposition-case1}
Case 1 is impossible.
\end{proposition}

\begin{proof}
According to Lemma~\ref{flux-zero-lemma}, the flux $F_n$ is zero.
Hence the limit in \eqref{equation*9} is zero. We compute that limit and show that it is nonzero.


Differentiating equation \eqref{equation-limit-un}, we get
\[
\wtu_z=
\frac{c_0}{2iz}-\sum_{i=1}^m\frac{c_i}{2z}\left(\frac{1}{\log z-\log p_i}-\frac{1}{\log z-\log\overline{p_i}}\right).
\]
Therefore,
\begin{eqnarray*}
\lefteqn{\Res_{p_1}(\wtu_z)^2(1-z^2)} \\
&=&
\Res_{p_1}\frac{1-z^2}{4z^2}\left[\frac{c_1^2}{(\log z-\log p_1)^2}
+2\frac{c_1}{\log z-\log p_1}
\right.\\
&&\left.\left(\frac{-c_0}{i}-\frac{c_1}{\log z-\log \overline{p_1}}
+\sum_{i=2}^m \frac{c_i}{\log z-\log p_i}-\frac{c_i}{\log z-\log\overline{p_i}}\right)\right]\\
&=& -\frac{c_1^2(1+p_1^2)}{4p_1}+\frac{c_1(1-p_1^2)}{2p_1}\left(
\frac{-c_0}{i}-\frac{c_1}{\log p_1-\log \overline{p_1}}\right.\\
&&\left.+\sum_{i=2}^m\frac{c_i}{\log p_1-\log p_i}-\frac{c_i}{\log p_1-\log \overline{p_i}}\right).
\end{eqnarray*}
(See proposition~\ref{proposition-residue} in 
the appendix
for the residue computations.)
Write $p_j=i y_j$ for $1\leq j\leq m$ so all $y_j$ are positive numbers.
By Lemma~\ref{flux-zero-lemma}, equation \eqref{equation*9} and the Residue Theorem,

\begin{eqnarray}
0&=&-\Re\int_{C(p_1,\varepsilon)}(\wtu_z)^2(1-z^2)\,dz\nonumber
\\
&=&-\Re\left[ 2\pi i \frac{y_1^2+1}{4iy_1}\left(c_1^2\frac{y_1^2-1}{y_1^2+1}
+2 c_1 \left(-\frac{c_0}{i}-\frac{c_1}{i\pi}
\right. \right. \right.\nonumber\\
& &\left. \left. \left.
+\sum_{i=2}^m\frac{c_i}{\log y_1-\log y_i}-\frac{c_i}{\log y_1-\log y_i+i\pi}
\right)\right)\right]\nonumber\\
\label{equation*10}
&=&\frac{\pi(y_1^2+1)}{2y_1}\left[
c_1^2\frac{1-y_1^2}{y_1^2+1}+\sum_{i=2}^m \frac{-2\pi^2\,c_1 c_i}{(\log y_1-\log y_i)|\log y_1-\log y_i+i\pi|^2}\right].
\end{eqnarray}
Now $y_1<1$ and $y_1<y_i$ for all $i\geq 2$, so all terms in \eqref{equation*10} are positive.
This contradiction proves proposition~\ref{proposition-case1}.
\end{proof}

We remark that the bracketed term in \eqref{equation*10} is precisely the expression for the force $F_1$ in
section \ref{section-physics}.

\addtocontents{toc}{\SkipTocEntry}
\section*{Barriers}

We now introduce various barriers that will be used to prove proposition~\ref{first-case-1-proposition}.
Fix some $\alpha\in(0,1)$. 
\begin{definition}
\label{def-of-A_n} 
$A_n$ is the set of points
$(z,\arg z)$ in $\wtC*$ which satisfy $t_n^{\alpha}<|z|<1$ and $\arg z>0$,
minus the disks $D(p_{i,n},t_n^{\alpha})$ for $1\leq i\leq N$.
\end{definition}
By the disk $D(p,r)$ in $\wtC*$ (for small $r$), we mean the points
$(z,\arg z)$ such that $|z-p|<r$ and $\arg z$ is close to $\arg p$.
It is clear that $A_n\subset\Omega_n$ for large $n$, since $t_n^{\alpha}\gg t_n$.
Moreover, if $z\in A_n$ then $d(z,\partial\Omega_n)\geq t_n^{\alpha}/2$.

We work in the hemisphere $|z|\leq 1$ where the conformal factor
of the spherical metric in \eqref{equation-metric-S2}
satisfies $1\leq\lambda\leq 2$. Hence Euclidean and spherical
distances are comparable. We will use Euclidean distance. Also the Euclidean and
spherical Laplacians are comparable. The symbol $\Delta$ will mean Euclidean Laplacian.

Let $\delta$ be the function on $A_n$ defined by
\[
\delta(z)=\left\{\begin{array}{ll}
\min \{|z|, |z-p_{1,n}|, \dots,  |z-p_{N,n}| \}
 &\mbox{ if }0<\arg z<\pi\\
 |z| &\mbox{ if }\arg(z)\geq\pi.
 \end{array}\right.
 \]

\begin{lemma} 
\label{lemma1}
There exists a constant $C_1$ such that
in the domain $A_n$, the function $u_n$ satisfies
\[
|\Delta u_n|\leq C_1\frac{t_n^3}{\delta^4}.
\]
\end{lemma}

\begin{proof}
The function $f_n(z)=\frac{t_n}{2\pi}\arg z-u_n(z)$ satisfies the minimal surface
equation, and $|\Delta f_n|=|\Delta u_n|$.
The proposition then follows from 
proposition~\ref{proposition-schauder} in the appendix, a straightforward application of the Schauder estimates.
More precisely:
\begin{enumerate}[\upshape $\bullet$]
\item If $0<\arg z<\pi$, we apply
proposition~\ref{proposition-schauder} on the domain
\[
A'_n=\{w\in\Omega_n\,:\,-\pi/2<\arg w<3\pi/2, |w|<2\}.
\]
The distance $d(z,\partial A'_n)$ is  comparable
to $\delta(z)$. The function $f_n$ is bounded
by $3t_n/4$.
\item If $k\pi\leq\arg z<k\pi+\pi$ for some $k\geq 1$, we apply proposition~\ref{proposition-schauder} to the function $f_n-\frac{k}{2}t_n$ and the domain 
\[
A'_n=\{w\in\Omega_n\,:\, k\pi-\pi/2<\arg w<k\pi+3\pi/2,|w|<2\}.
\]
 The distance $d(z,\partial A'_n)$
 is comparable to $|z|$. The function $f_n-\frac{k}{2}t_n$ is again bounded by
$3t_n/4$.
\end{enumerate}
\end{proof}

Next, we need to construct a function whose Laplacian is greater than $1/\delta^4$,
in order
to compensate for the fact that $u_n$ is not quite harmonic.
Let $\chi:\rR^+\to[0,1]$ be a fixed, smooth function such that $\chi\equiv 1$ on
$[0,\pi]$ and $\chi\equiv 0$ on $[2\pi,\infty)$.

\begin{lemma}
\label{lemma2}
There exists a constant $C_2\geq 1$ such that the function $g_n$ defined on $A_n$ by
\[
g_n(z)=\frac{C_2}{|z|^2}+\chi(\arg z)
\sum_{i=1}^N
\frac{1}{|z-p_{i,n}|^2}
\]
satisfies
\begin{equation}
\label{equation*12}
\Delta g_n\geq \frac{4}{\delta^4}.
\end{equation}
Moreover $\partial g_n/\partial \nu\leq 0$ on $|z|=1$ and
\begin{equation}
\label{equation*13}
g_n\leq \frac{C_2+
N
}{t_n^{2\alpha}}\quad\mbox{ in } A_n.
\end{equation}
\end{lemma}

\begin{proof}
The inequality \eqref{equation*13} follows immediately from the definitions of $g_n$
and $A_n$.
The function $f$ defined in polar coordinate by $f(r,\theta)=1/r^2$ satisfies
\[
|\nabla f|=\frac{2}{r^3},\qquad \Delta f=\frac{4}{r^4}.
\]
Hence for $\arg z\geq 2\pi$, \eqref{equation*12} is satisfied for any $C_2\geq 1$.
Suppose $0<\arg z <\pi$. Then
\[
\Delta g_n=\frac{4C_2}{|z|^4}+
\sum_{i=1}^N
\frac{4}{|z-p_{i,n}|^4}\geq \frac{4}{\delta^4}
\]
so again, \eqref{equation*12} is satisfied for any $C_2\geq 1$.
If $\theta=\arg z\in [\pi,2\pi]$, we have $|z-p_{i,n}|\geq |z|=r$ and
\[
|\nabla \chi(\arg z)|\leq \frac{C}{r},\qquad |\Delta\chi(\arg z)|\leq \frac{C}{r^2}.
\]
Hence
\[
\left|\Delta \frac{\chi(\arg z)}{|z-p_{i,n}|^2}\right|\leq
\frac{C}{r^2}\frac{1}{r^2}+2\frac{C}{r}\frac{2}{r^3}+\frac{4}{r^4}.
\]
Therefore, $\Delta g_n\geq 4/r^4$ provided $C_2$ is large enough.
(The constant $C_2$ only depends on $N$ and a bound on $\chi'$ and $\chi''$.)
This completes the proof of \eqref{equation*12}.
\end{proof}


We need a harmonic function on $\wtC*$ that is greater than $|\log t|$ on
$|z|=t$. A good candidate would be $-\log |z|$. However this function has the wrong 
Neumann data on the unit circle. We propose the following:

\begin{lemma}
\label{H_t-properties-lemma}
For $0<t<1$, the harmonic function $H_t(z)$ defined for $z\in\wtC*$, $\arg z>0$ by
\[
H_t(z)=\Im\left(\frac{\log t\log z}{\log t +i\log z}\right)
\]
has the following properties :
\begin{enumerate}[\upshape (1)]
\item $H_t(z)>0$ if $\arg z>0$,
\item $H_t(1/\overline{z})=H_t(z)$, hence $\partial H_t/\partial\nu=0$ on $|z|=1$,
\item $H_t(z)\geq |\log t|/2$ if $|z|=t$,
\item for fixed $t$, $H_t(z)\geq |\log t|/2$ when $\arg z\to\infty$, uniformly with respect to $|z|$
in $t\leq |z|\leq 1$,
\item for fixed $z$, $H_t(z)\to \arg z$ when $t\to 0$,
\item $H_t(z)\leq |\log z|$ if $\arg z>0$.
\end{enumerate}
\end{lemma}

\begin{proof}
It suffices to compute $H_t(z)$ in polar coordinates
$z=r e^{i\theta}$ :
\[
H_t(z)=\frac{(\log t)^2\theta+|\log t|((\log r)^2+\theta^2)}{(\log t-\theta)^2+(\log r)^2}.
\]
The first two points follow.
If $r=t$ then
\[
H_t(z)=\frac{|\log t|}{2}\left(1+\frac{\theta^2}{2(\log t)^2+2|\log t|\theta +\theta^2}\right)
\geq \frac{|\log t|}{2}
\]
which proves point 3. If $t\leq r\leq 1$ then
\[
H_t(z)\geq \frac{ (\log t)^2\theta+|\log t|\theta^2}{(\log t-\theta)^2+(\log t)^2}
\]
which gives point 4. Point 5 is elementary. For the last point, write
\[
\left|\frac{\log t\log z}{\log t+i\log z}\right|\leq \left|\frac{\log t\log z}{\log t}\right|=|\log z|.
\]
\end{proof}

\stepcounter{theorem}
\addtocontents{toc}{\SkipTocEntry}
\subsection{Proof of proposition~\ref{first-case-1-proposition}}
\label{section:proof-of-first-case-1-proposition}
The function $\wtu_n$ defined in~\eqref{eq:def-of-tilde-u_n} has the following properties in $A_n$:
\begin{equation}
\label{equation*14}
\begin{array}{ll}
\bullet & \displaystyle|\Delta \wtu_n|\leq  C_1 \frac{t_n^2|\log t_n|}{\delta^4}
\text{ by Lemma~\ref{lemma1}},\\
\bullet & \wtu_n\leq |\log t_n|/2,\\
\bullet & \wtu_n=0 \text{ on $\arg z=0$,}\\
\bullet & \partial \wtu_n/\partial\nu=0 \text{ on $|z|=1$.}
\end{array}
\end{equation}
The last three properties follow from \eqref{eq:u_n-properties} and the fact that $A_n\subset\Omega_n$.
Consider the barrier $v_n=v_{1,n}+v_{2,n}+v_{3,n}$ where
\[
v_{1,n}(z)= - C_1 t_n^2|\log t_n| g_n(z)
+C_1(C_2+
N
)t_n^{2-2\alpha}|\log t_n|,
\]
\[
v_{2,n}(z)=\frac{1}{\alpha}\sum_{i=1}^N
h_{p_{i,n}}(z)
=-\frac{1}{\alpha}\sum_{i=1}^N\log\left|\frac{\log z-\log p_{i,n}}
{\log z-\log\overline{p_{i,n}}}\right|,
\]
\[
v_{3,n}(z)=\frac{1}{\alpha}H_{t_n^{\alpha}}(z).
\]
The function $v_{1,n}$ is positive in $A_n$ by the estimate \eqref{equation*13} of
Lemma~\ref{lemma2}.
Observe that the second term in the expression for $v_{1,n}$ tends to $0$ as $n\to\infty$ since $\alpha<1$.
The functions $v_{2,n}$ and $v_{3,n}$ are harmonic and positive in $A_n$
(see point (1) of Lemma~\ref{H_t-properties-lemma} for $v_{3,n}$).


By \eqref{eq:h_p-symmetries} and the symmetry of the set $\{p_{1,n}, \dots, p_{N,n}\}$
(see \eqref{equation**}), 
the function $v_{2,n}$ satisfies $v_{2,n}(1/\overline{z})=v_{2,n}(z)$.
Hence $\partial v_{2,n}/\partial\nu=0$ on the unit circle.
By point (2) of Lemma~\ref{H_t-properties-lemma}, $\partial v_{3,n}/\partial\nu=0$ on the unit circle.
Therefore by Lemma~\ref{lemma2},
\[
\frac{\partial v_n}{\partial\nu}=\frac{\partial v_{1,n}}{\partial\nu}\geq 0
\text{ on $|z|=1$.}
\]
Because $p_{i,n}\to p_i\neq 0$, we have on the circle $C(p_{i,n},t_n^{\alpha})$
\[
\log |\log z - \log p_{i,n}|\simeq \log |z-p_{i,n}|
\]
Hence for large $n$ and for $1\leq i \leq N$
\[
v_{2,n}\geq \frac{1}{2\alpha}|\log t_n^{\alpha}|\geq \wtu_n\quad \mbox{ on } C(p_{i,n},t_n^{\alpha}).
\]
Using point (3) of Lemma~\ref{H_t-properties-lemma} and the second statement of \eqref{equation*14},
we have
$v_{3,n}\geq \wtu_n$ on the boundary component
$|z|=t_n^{\alpha}$.
So we have
\begin{equation}
\label{equation*15}  
\begin{array}{ll}
\bullet &\Delta \wtu_n\geq \Delta v_n \text{ in $A_n$},\\
\bullet &\wtu_n\leq v_n \text{ on the boundaries $\arg z=0$, $|z|=t_n^{\alpha}$ and
$C(p_{i,n},t_n^{\alpha})$,}\\
\bullet &\partial \wtu_n/\partial \nu\leq \partial v_n/\partial \nu\text{ on the boundary $|z|=1$},\\
\bullet &\wtu_n\leq v_n \text{ when $\arg z\to\infty$.}
\end{array}
\end{equation}
(The first statement follows from \eqref{equation*12}
and the first statement of \eqref{equation*14}.)


By the maximum principle, we have $\wtu_n\leq v_n$ in $A_n$.


For any compact subset $K$ of the 
set 
\[
     \{z\in\wtC*\,:\,|z|\leq 1,\arg z\geq 0\}\setminus\{p_1,\dots,p_{m}\},
\]
the function $v_n$ is bounded by $C(K)$ on $K$.
(For $v_{3,n}$, use the last point of Lemma~\ref{H_t-properties-lemma}.)
Then by symmetry, $u_n$ is bounded by $C(K)$ on $K\cup\overline{K}\cup
\sigma(K)\cup\sigma(\overline{K})$, where $\sigma$ denotes the inversion
$z\mapsto \overline{z}$.
Let
\[
\Omega_{\infty}=\lim_{n\to\infty}\Omega_n=\wtC*\setminus\{\pm p_1,\dots,\pm p_m\}.
\]
Then $\wtu_n$ is bounded on compact subsets of $\Omega_{\infty}$.
By standard PDE theory, passing to a subsequence, $\wtu_n$ has a limit $\wtu$.
The convergence is the uniform smooth convergence on compact subsets of $\Omega_{\infty}$.
The limit has the following properties
\begin{itemize}
\item $\wtu$ is harmonic in $\Omega_{\infty}$. This follows from the first point of \eqref{equation*14}.
\item $\wtu(\overline{z})=-\wtu(z)$ and $\wtu(1/\overline{z})=\wtu(z)$.
\item $\wtu(z)\geq 0$ if $\arg z\geq 0$.
\end{itemize}
Note that either $\wtu\equiv 0$ or $\wtu$ is positive in $\arg z>0$.
Using the fact that $\log:\wtC*\to\CC$ is biholomorphic, the following lemma tells us that $\wtu$ has the form given by equation \eqref{equation-limit-un}.

\begin{lemma}
\label{lemma-positive-harmonic}
Let $H$ be the upper half plane $\Im z> 0$ in $\CC$. Let $u$ be a positive harmonic function
 in $H\setminus\{q_1,\dots,q_m\}$ with boundary value $u=0$ on $\rR$, where each $q_i\in H$.
Then there exists non-negative constants $c_0,\dots,c_m$ such that
\[
u(z)=c_0\Im z-\sum_{i=1}^m c_i\log\left|\frac{z-q_i}{z-\overline{q_i}}\right|.
\]
\end{lemma}

\begin{proof}
By B\^ocher's Theorem (\cite{axler}, theorem 3.9), a positive harmonic
function in a punctured disk has a logarithmic singularity at the puncture.
Hence for each $1\leq i\leq m$,
there exists a non-negative constant $c_i$ such that $u(z)+c_i\log|z-q_i|$ extends
analytically at $q_i$. Consider the harmonic function
\[
h(z)=-\sum_{i=1}^m c_i\log \left|\frac{z-q_i}{z-\overline{q_i}}\right|.
\]
Observe that $h=0$ on $\rR$.
Then $u-h$ extends to a harmonic function in $H$ with boundary value
$0$ on $\rR$.
For every $\varepsilon>0$, there exists an $r>0$ such that $|h(z)|\leq \varepsilon$
for $|z|\geq r$.
Consequently, $u-h>-\varepsilon$ for $z\in H$, $|z|\geq r$.
By the maximum principle, $u-h>-\varepsilon$ in $H$.
Since this is true for arbitrary positive $\varepsilon$,
we conclude that $u-h$ is non negative in $H$.
Now a non negative harmonic function in $H$ with boundary value $0$ on $\rR$
is equal to $c_0\Im z$ for some non-negative constant $c_0$ (\cite{axler}, theorem 7.22).
\end{proof}

To conclude the proof of proposition~\ref{first-case-1-proposition},
it remains to compute the numbers $c_i$ for $1\leq i\leq m$.
Recall that $\phi_{i,n}$ is the vertical flux of $M_n$ on the graph of $f_n$
restricted to the circle $C(p_i,\varepsilon)$.
By proposition~\ref{proposition-flux2},
\[
\phi_{i,n}=\Im\int_{C(p_i,\varepsilon)} (2f_{n,z}+O(t_n^2))\,dz
=\Im\int_{C(p_i,\varepsilon)} (-2u_{n,z}+O(t_n^2))\,dz.
\]
Now
\[
\lim \frac{|\log t_n|}{t_n} u_n = - c_i\log|z-p_i|+\text{harmonic} \quad\text{ near $p_i$}.
\]
\[
\lim\frac{|\log t_n|}{t_n} 2u_{n,z}=- \frac{c_i}{z-p_i} +\text{holomorphic}\quad\text{ near $p_i$}.
\]
Hence by the Residue Theorem,
\[
\lim \frac{|\log t_n|}{t_n}\phi_{i,n}
= 2\pi c_i.
\]
This finishes the proof of proposition~\ref{first-case-1-proposition}.
\QED

As a corollary of the proof of proposition~\ref{first-case-1-proposition}, we have
an estimate of $u_n$ that we will need in section \ref{section45}.
For convenience, we state it here as a lemma.

Fix some $\beta\in(0,\alpha)$ and let $A'_n\subset A_n$ be the domain defined as $A_n$ in
definition~\ref{def-of-A_n}, replacing $\alpha$ by $\beta$, namely:
$A'_n$ is the set of points
$(z,\arg z)$ in $\wtC*$ which satisfy $t_n^{\beta}<|z|<1$ and $\arg z>0$,
minus the disks $D(p_{i,n},t_n^{\beta})$ for $1\leq i\leq N$.

\begin{lemma}
\label{lemma5} Assume that $p_1\neq 0$. Then
for $n$ large enough (depending only on $\beta$ and a lower bound on $|p_1|$), we have
\[
u_n\leq(N+2)\frac{\beta}{\alpha} t_n\quad \mbox{ in } A'_n.
\]
\end{lemma}
Recalling that $u_n<t_n/2$, this lemma is usefull when $\beta$ is small. We will use it
to get information about the level sets of $u_n$.

\begin{proof}
 As we have seen in the proof of proposition~\ref{first-case-1-proposition}, we have in $A_n$
\begin{equation}
\label{equation-un-vn}
u_n\leq \frac{t_n}{|\log t_n|} v_n=\frac{t_n}{|\log t_n|}(v_{1,n}+v_{2,n}+v_{3,n}).
\end{equation}
We need to estimate the functions $v_{1,n}$, $v_{2,n}$ and
$v_{3,n}$ in $A'_n$.
We have in $A_n$
\[
v_{1,n}\leq C_1(C_2+N)t_n^{2-2\alpha}|\log t_n|=o(|\log t_n|).
\]
By point 6 of lemma~\ref{H_t-properties-lemma}, we have in $A'_n$
\[
v_{3,n}\leq \frac{1}{\alpha}|\log z|\leq \frac{1}{\alpha}|\log t_n^{\beta}|=\frac{\beta}{\alpha}|\log t_n|.
\]
Regarding the function
$v_{2,n}$, we need to estimate each function $h_{p_{i,n}}$ in the domain  $A'_n$.
The function $h_{p_{i,n}}$ is harmonic in the domain
\[
\{z\in\wtC*\,:\, \arg z>0,\,t_n^{\beta}< |z|<1\}\setminus D(p_{i,n},t_n^{\beta})
\]
and goes to $0$ as $\arg z\to\infty$, so its maximum is on
the boundary. Since $h_{p_{i,n}}(1/\overline{z})=h_{p_{i,n}}(z)$, the maximum is not
on the circle $|z|=1$ (because it would be an interior maximum of $h_{p_{i,n}}$). Also $h_{p_{i,n}}=0$ on $\arg z=0$. Therefore,
the maximum is either on $|z|=t_n^{\beta}$ or on the circle $C(p_{i,n},t_n^{\beta})$.
On $|z|=t_n^{\beta}$, we have $h_{p_{i,n}}\to 0$ because $p_{i,n}$ is bounded away
from $0$. On the circle $C(p_{i,n},t_n^{\beta})$, we have for $n$ large
\[
\log z-\log p_{i,n}\simeq \frac{1}{p_{i,n}}(z-p_{i,n})
\]
\[
|\log z-\log p_{i,n}|\geq \frac{t_n^{\beta}}{2|p_{i,n}|}.
\]
Hence
\[
-\log|\log z-\log p_{i,n}|\leq  \log (2 |p_{i,n}|)+\beta|\log t_n|.
\]
Also,
\[
\log|\log z-\log \overline{p_{i,n}}|\leq \log (|\log z|+|\log p_{i,n}|)\simeq \log(2|\log p_{i,n}|).
\]
Since $|p_{i,n}|$ is bounded away from $0$, this gives for $n$ large enough
\[
h_{p_{i,n}}\leq C+\beta|\log t_n|\quad\text{ in $A'_n$.}
\]
Hence
\[
v_{2,n}\leq C+N\frac{\beta}{\alpha} |\log t_n|.
\]
Collecting all terms, we get, for $n$ large enough:
\[
v_n\leq C+(N+1)\frac{\beta}{\alpha}|\log t_n|\leq (N+2)\frac{\beta}{\alpha}|\log t_n|\quad\text{ in $A'_n$}.
\]
Using \eqref{equation-un-vn}, the lemma follows.
\end{proof}

\stepcounter{theorem}
\addtocontents{toc}{\SkipTocEntry}
\subsection{Proof of proposition~\ref{second-case-1-proposition}}
\label{section:proof-of-second-case-1-proposition}
We use the notation introduced
    at the end of
the proof of proposition~\ref{first-case-1-proposition}.
Fix some index $i \le m$ and let $J=\{j\in [1,N]\,:\, p_j=p_i\}$.
By permuting the indices of the $p_{j,n}$ within the cluster $J$, 
we can assume that
\[
r_{i,n}=\max\{r_{j,n}\,:\,j\in J\}.
\]
(Recall from section~\ref{section-setup} that $r_{i,n}$ is the distance (with
respect to the spherical metric) from $p_{i,n}$ to the two nearest points
in $M_n\cap Y$, namely the points $p_{i,n}'$ and $p_{i,n}''$.)
Fix some positive $\varepsilon$ such that $|p_j-p_i|\geq 2\varepsilon$ for
$j\notin J$.

From statement~$i.$ of proposition~\ref{Setup-proposition}, 
 we know that  near $p_{j,n}$ the surface $M_n$ is close to a 
 vertical catenoid with  waist circle of radius $r_{j,n}$. More precisely,
$\frac{1}{r_{j,n}}(M_n-p_{j,n})$ converges to the standard catenoid
\[
x_3=\cosh^{-1}\sqrt{x_1^2+x_2^2}.
\]
Since the vertical flux of the standard catenoid is $2\pi$, we have
\begin{equation}
\label{estimate-phi}
\phi_{i,n}\simeq 2\pi \sum_{j\in J}r_{j,n}\leq 2\pi |J| r_{i,n}.
\end{equation}
Let
\[
h_{j,n}=r_{j,n}\cosh^{-1}(2\lambda).
\]
Observe that $h_{j,n}\ll t_n$.
Consider the intersection of $M_n$ with the plane at height $h_{j,n}$ and project it
on the horizontal plane. There is one component which is close
to the circle $C(p_{j,n},2\lambda r_{j,n})$. We call this component $\gamma_{j,n}$.
Observe that $\gamma_{j,n}\subset\Omega_n$ and $f_n=h_{j,n}$ on $\gamma_{j,n}$.
Let $D_{j,n}$ be the disk bounded by $\gamma_{j,n}$.


We now estimate  $f_n$ on the circle $C(p_{i,n}, \varepsilon)$. 
By proposition~\ref{first-case-1-proposition}, we know that
 $|u_n|=O(\frac{t_n}{|log t_n|})$. Hence 
 $f_n=\frac{t_n}{2\pi}\arg z-u_n(z)\sim  \frac{t_n}{2\pi}\arg z$ on  $C(p_{i,n}, \varepsilon)$.
Since $p_{i,n}$ is on the positive imaginary axis, $\arg z=\pi/2 +O(\varepsilon)$ on $C(p_{i,n}, \varepsilon)$. Hence$f_n(z)\sim \frac{t_n}{4}$ on $C(p_{i,n}, \varepsilon)$.
Consequently, the level set $f_n=\frac{t_n}{8}$ inside $\Omega_n\cap D(p_{i,n},\varepsilon)$ is a closed curve, possibly with several components.
We select the component
which encloses the point $p_{i,n}$ and call it $\Gamma_n$.
(Note that by choosing a very slightly different height, we may assume that $\Gamma_n$
is a regular curve). Let $D_n$ be the disk bounded by $\Gamma_n$.
Let
\[
\Omega'_n=D_n\setminus\bigcup_{j\in J} D_{j,n}
\]
Then $\Omega'_n\subset\Omega_n$. 

We are now able to apply the height estimate proposition~\ref{proposition-height}
in the appendix with $r_1=\lambda r_{i,n}$,
$r_2=\varepsilon$, $h=t_n/8-h_{i,n}\simeq t_n/8$ and
$f$ equal to the
function $f_n(z-p_{i,n})-t_n/8$. (Observe that by proposition~\ref{Setup-proposition}, statement~${ii.}$, we may assume that 
  $|\nabla f_n|\leq 1$. Also the fact that $\partial f_n/\partial\nu<0$ on $\gamma_{j,n}$
follows from the convergence to a catenoid.)
We obtain
\[
\frac{t_n}{8}-h_{i,n}\leq \frac{\sqrt{2}}{\pi}\phi_{i,n} \log\frac{\varepsilon}{\lambda r_{i,n}}.
\]
Using \eqref{estimate-phi}, this gives for $n$ large enough
\begin{equation}\label{FluxEstimateOne}
\frac{t_n}{9}\leq \frac{\sqrt{2}}{\pi}\phi_{i,n} \log\frac{2\pi|J|\varepsilon}{\lambda \phi_{i,n}}
\end{equation}
We claim that for $n$ large enough,
\begin{equation}\label{FluxEstimateTwo}
\frac{\lambda\phi_{i,n}}{2\pi\,|J|\,\varepsilon} \geq t_n^2.
\end{equation}
Indeed, suppose on the contrary that 
\[
    w_n:=\frac{\lambda\phi_{i,n}}{2\pi\,|J|\,\varepsilon} < t_n^2
\]
 for 
infinitely many values of $n$. 
Since the function $x\log \frac{1}{x}$ is
increasing for $x>0$ small enough, we have 
\[
    w_n\log\frac{1}{w_n}     <       t_n^2\log\frac{1}{t_n^2}
\]
for all sufficiently large $n$.
But \eqref{FluxEstimateOne} gives
\[
     Ct_n\leq w_n\log\frac{1}{w_n}
\]
for some  positive  constant $C$ independent of $n$.  Combining these last two inequalities gives
\[
    C\leq  t_n \log\frac{1}{t_n^2}   =  2t_n |\log t_n|.
\]
Hence $t_n|\log t_n|$ is bounded below by a positive constant.
This is a contradiction since $t_n|\log t_n|\rightarrow 0$.
Thus inequality \eqref{FluxEstimateTwo} is proved.

Inequality \eqref{FluxEstimateTwo} implies (using that $|\log t_n|=-\log t_n$, since $t_n<1$ for $n$ large)
\[
    \log\frac{2\pi|J|\varepsilon}{\lambda \phi_{i,n}}\leq |\log t_n^2|.
\]
Then by \eqref{FluxEstimateOne}
\[
\frac{t_n}{9}\leq \frac{2\sqrt{2}}{\pi} \phi_{i,n} |\log t_n|,
\]
which  implies that $\frac{|\log t_n|}{t_n}\phi_{i,n}$ is bounded below by a positive constant independent of $n$.
 Therefore, the coefficient $c_i$  defined in \eqref{equation-ci} is
 positive,  as desired.\QED

 \begin{remark*}
 Together with \eqref{estimate-phi}, this gives
 \begin{equation}
 \label{equation-minoration-rin}
 r_{i,n}\geq \frac{1}{36|J|\sqrt{2}}\frac{t_n}{|\log t_n|}
 \end{equation}
 for large $n$.
 This is a lower bound on the size of the largest catenoidal neck in the cluster corresponding to $p_i$.
 We have no lower bound for $r_{j,n}$ if $j\in J$, $j\neq i$. Conceptually, we could have
 $r_{j,n}=o(\frac{t_n}{|\log t_n|})$, although this seems unlikely. 
 \end{remark*}
 
 \stepcounter{theorem}
 \addtocontents{toc}{\SkipTocEntry}
\subsection{Proof of proposition~\ref{third-case-1-proposition}}
\label{section:proof-of-third-case-1-proposition}
Let $g_n=u_{n,z}.$
We have to prove
\[
\lim \left( \frac{\log t_n}{t_n}\right)^2
   \operatorname{Re} \int_{C(p_1,\epsilon)} \frac{2t_n}{4\pi i z} u_{n,z} (1-z)^2\,dz = 0,
\]
i.e., that
\begin{equation}\label{the-integral}
\Re\int_{C(p_1,\varepsilon)}\frac{1}{2iz}g_n(z)(1-z^2)\,dz=o\left(\frac{t_n}{(\log t_n)^2}\right).
\end{equation}
Fix some $\alpha$ such that $0<\alpha<\frac{1}{2}$ and some small $\varepsilon>0$.
Let $J$ be the set of indices such that $p_j=p_1$.
Consider the domain
\[
A_n=D(p_1,\varepsilon)-\bigcup_{j\in J} D(p_{j,n},t_n^{\alpha})\subset\Omega_n.
\]
By proposition~\ref{proposition-schauder} in the appendix, we have in $A_n$
\[
|g_{n,\overline{z}}|=\frac{1}{4}|\Delta u_n|=\frac{1}{4}|\Delta f_n|\leq Ct_n^{3-4\alpha}.
\]
\[
|\nabla f_n|\leq Ct_n^{1-\alpha}.
\]
As the gradient of $t_n\arg z$ is $O(t_n)$ in $A_n$, this gives
\[
|\nabla u_n|\leq Ct_n^{1-\alpha}.
\]
Hence
\begin{equation}
\label{equation-gn}
|g_n|\leq Ct_n^{1-\alpha}.
\end{equation}
Proposition~\ref{proposition-laurent} gives us the formula
\[
g_n(z)=g^+(z)+
\sum_{j\in J}
g_j^-(z)+ \frac{1}{2\pi i}\int_{A_n}\frac{g_{n,\overline{z}}(w)}{w-z}\,dw\wedge\overline{dw}
\]
where of course the functions $g^+$ and $g_j^-$ depend on $n$.
\begin{enumerate}[\upshape$\bullet$]
\item The function $g^+$ is holomorphic in $D(p_1,\varepsilon)$ so 
       does not contribute to the integral~\eqref{the-integral}.
\item The last term is bounded by $Ct_n^{3-4\alpha}$. (The integral of $dw\wedge \overline{dw}/(w-z)$ is
uniformly convergent.) Therefore we need $3-4\alpha>1$, namely $\alpha<\frac{1}{2}$ 
so that the contribution of this term to the integral is $o(t_n/(\log t_n)^2)$.
\item Each function $g_j^-$ can be expanded  as $\sum_{k=1}^{\infty} \frac{a_{j,k}}{(z-p_i)^k}$
 by proposition~\ref{proposition-laurent}. 
 By proposition~\ref{proposition-real-residue},   
each residue $a_{j,1}$ is real. Hence
\begin{equation}
\label{equation-aj1}
\Re\int_{C(p_1,\varepsilon)} \frac{1}{2iz} \frac{a_{j,1}}{(z-p_{j,n})} (1-z^2)\,dz
=a_{j,1}\Re \left(\frac{2\pi i}{2i p_{j,n}}(1-p_{j,n}^2)\right)=0
\end{equation}
because $p_{j,n}$ is imaginary.  Thus $a_{j,1}$ does not contribute to the integral~\eqref{the-integral}.
\item It remains to estimate the coefficients $a_{j,k}$ for $k\geq 2$. Using
\eqref{equation-gn},
\[
|a_{j,k}|=\left|\frac{1}{2\pi i}\int_{C(p_{j,n},t_n^{\alpha})}g_n(z)(z-p_{j,n})^{k-1}\,dz\right|
\leq C t_n^{1+(k-1)\alpha}
\]
If $z\in C(p_1,\varepsilon)$, then $|z-p_{j,n}|\geq \varepsilon/2$, so
\[
\left|
\sum_{k=2}^{\infty} a_{j,k} (z-p_{j,k})^{-k}\right|\leq C\sum_{k\geq 2} t_n^{1+(k-1)\alpha}
\left(\frac{2}{\epsilon}\right)^k\leq \frac{4C}{\varepsilon^2} t_n^{1+\alpha}\sum_{k=2}^{\infty}
\left(\frac{2t_n^{\alpha}}{\varepsilon}\right)^{k-2}.
\]
The last sum converges because $\alpha>0$. Hence
the contribution of this term to the integral is $o(t_n/(\log t_n)^2)$ as desired.
\end{enumerate}
\QED

\section{Case 2: \texorpdfstring{$p_1=0$}{Lg}}
\label{case-2-section}
In this case we make a blow up at the origin.
Let
\[
R_n=\frac{1}{|p_{1,n}|}.
\]
(Here we assume again that the points $p_{i,n}$ are ordered by increasing imaginary
part as in section \ref{section-setup}.)
Let $\whM_n=R_n M_n$. This is a helicoidal minimal surface in $\sS^2(R_n)\times\rR$ with
pitch 
\[
\wht_n=R_n t_n.
\]
By choice of $p_{1,n}$, we have $|p_{1,n}| >> t_n$, so $\lim \wht_n=0$. 
Let $\whOmega_n=R_n\Omega_n$. $\whM_n$ is the graph on $\whOmega_n$
of the function
\[
\whf_n(z)=\frac{\wht_n}{2\pi}\arg z-\whu_n(z)
\]
where
\[
\whu_n(z)=R_n u_n(\frac{z}{R_n}).
\]
Let $\whp_{i,n}=R_n p_{i,n}$. After passing to a subsequence,
\begin{equation}\label{eq:wide-hat-limits}
\whp_j=\lim\whp_{j,n}\in[i,\infty] 
\end{equation}
exists for $j\in[1,N]$ and we have $\whp_1=i$.
Let $m$ be the number of distinct, finite points amongst $\whp_1,\dots,\whp_N$.
For each $n$, relabel the points $\whp_{i,n}$ (by permuting the indices) so that 
the points $\whp_1,\dots,\whp_m$ defined by~\eqref{eq:wide-hat-limits} are distinct and so that
\[
1=\Im \whp_1<\Im\whp_2<\dots<\Im \whp_m.
\]
\begin{proposition}
\label{first-case-2-proposition}
After passing to a subsequence,
\[
\lim\frac{|\log \wht_n|}{\wht_n}\whu_n(z)=
\whc_0 \arg z+\sum_{i=1}^{m} \whc_i h_{\whp_i}(z).
\]
The convergence is the smooth uniform convergence on compact subsets of $\wtC*$ minus the points
$\pm\whp_i$, for $1\leq i\leq m$.
The numbers $\whc_i$ for $1\leq i\leq m$ are given by
\[
\whc_i=\lim\frac{|\log\wht_n|}{\wht_n}\frac{\widehat{\phi}_{i,n}}{2\pi}
\]
where $\widehat{\phi}_{i,n}$ is the vertical flux of $\whM_n$ on the graph of $\whf_n$ restricted
to the circle $C(\whp_i,\varepsilon)$, for some fixed small enough $\varepsilon$.
Moreover, $\whc_i>0$ for $1\leq i\leq m$.
\end{proposition}
This proposition is proved in section \ref{section31}.
The proof is very similar to the proofs of proposition~\ref{first-case-1-proposition} and 
(for the last statement)  proposition~\ref{second-case-1-proposition}.


Fix some small $\varepsilon>0$.
Let $F_n$ be the flux of the Killing field $\chi_Y$ on the circle $C(\whp_1,\varepsilon)$
on $\whM_n$. Since we are in $\sS^2(R_n)\times\rR$,
\[
\chi_Y=\frac{i}{2}(1-\frac{z^2}{R_n^2}).
\]
\[
F_n=-\Im\int_{C(\whp_1,\varepsilon)}
2\left(
\frac{\partial}{\partial z}\left(\frac{\wht_n}{2\pi}\arg z-\whu_n\right)\right)^2\frac{i}{2}\left(1-\frac{z^2}{R_n^2}
\right)
\,dz+O((\wht_n)^4).
\]
Expand the square. As in Case 1, the cross product term can be neglected and
since $R_n\to\infty$:
\begin{proposition}
\label{second-case-2-proposition}
\[
\lim\left(\frac{\log \wht_n}{\wht_n}\right)^2 F_n=-\lim\left(\frac{\log \wht_n}{\wht_n}\right)^2\Re\int_{C(\whp_1,\varepsilon)} (\whu_{n,z})^2 \,dz.
\]
\end{proposition}
(Same proof as proposition~\ref{third-case-1-proposition}).


Assuming these results, we now prove
\begin{proposition}
Case 2 is impossible.
\end{proposition}

\begin{proof}
Write
$\whp_j=i y_j$.
By the same computation as in section~\ref{case-1-section}, we get
(the only difference is that there is no $(1-z^2)$ factor)
\[
-\Re\int_{C(\whp_1,\varepsilon)}(\wtu_z)^2\,dz
=
\frac{\pi}{2 y_1}\left(
\whc_1{}^2+\sum_{i=2}^m \frac{-2\pi^2\, \whc_1 \whc_i}{(\log y_1-\log y_i)|\log y_1-\log y_i+i\pi|^2}\right).
\]
Again, since $y_i>y_1$ for $i\geq 2$, all terms are positive, contradiction.
\end{proof}

\stepcounter{theorem}
\addtocontents{toc}{\SkipTocEntry}
\subsection{Proof of proposition~\ref{first-case-2-proposition}}
\label{section31}
The setup of proposition~\ref{first-case-2-proposition} is the same as proposition~\ref{first-case-1-proposition}
except that we are in $\sS^2(R_n)\times\rR$ with $R_n\to\infty$
instead of $\sS^2(1)\times\rR$, and the pitch is $\wht_n$.
 Remember that $\lim \wht_n=0$.

{\bf From now on forget all hats:} write $t_n$ instead of $\wht_n$,
$u_n$ instead of $\whu_n$, $p_{i,n}$ instead of $\whp_{i,n}$, etc...

The proof of proposition~\ref{first-case-2-proposition} is substantially the same as the proofs of propositions \ref{first-case-1-proposition} and \ref{second-case-1-proposition}.
The main difference is that the equatorial circle $|z|=1$ becomes
$|z|=R_n$.
\begin{enumerate}[\upshape$\bullet$]
\item The definition of the domain $A_n$ is the same with $|z|<1$ replaced by
$|z|<R_n$.
\item Lemma~\ref{lemma1} is the same (recall that now $p_{i,n}$ means $\whp_{i,n}$).
\item Lemma~\ref{lemma2} is the same. The last statement must be replaced by
$\partial g_n/\partial\nu\leq 0$ on $|z|=R$ for $R\geq 1$.
\item Lemma~\ref{H_t-properties-lemma} is the same, we do not change the definition of the function
$H_t$. Instead of point 3, we need $\partial H_t/\partial\nu\geq 0$ on
$|z|=R$ for $R\geq 1$. This is true by the following computation:
\[
\frac{\partial H_t}{\partial r}=\frac{2\log r(\log t)^2(|\log t|+\theta)}
{((\log t-\theta)^2+(\log r)^2)^2}.
\]
\item The definition of the function $\wtu_n$ is the same, and it has the same properties,
except that the last point must be replaced by $\partial \wtu_n/\partial\nu=0$ on
$|z|=R_n$.
\item The definition of the function $v_{2,n}$ is the same (with $\whp_{i,n}$
in place of $p_{i,n}$), now it is symmetric
with respect to the circle $|z|=R_n$.
\item At the end, $K$ is a compact of the set $\{z\in\wtC*,\arg z\geq 0\}\setminus
\{\whp_1,\dots,\whp_m\}.$
The fact that $v_{2,n}$ is uniformly bounded on $K$
requires some care, maybe, because some points $\whp_{i,n}$ are not
bounded: it is true by the fact that if $\arg z$ and $\arg p$ are positive, then
\[
|\log z-\log p|\leq |\log z-\log\overline{p}|.
\]
\item The proof of the last point is exactly the proof of proposition~\ref{second-case-1-proposition},
working in $\sS^2(R_n)\times\rR$ instead of $\sS^2(1)\times\rR$.
\end{enumerate}
\QED

\section{Case 3: \texorpdfstring{$p_1=i$}{Lg}}\label{case-3-section}
Note that in this case, all points $p_{j,n}$ converge to $i$, for $j\in[1,N]$.
Passing to a subsequence, we distinguish two sub-cases:

\begin{itemize}
\item
   Case 3a: there exists $\beta>0$ such that $|p_{1,n}-i|\leq t_n^{\beta}$ for $n$ large enough,
\item
   Case 3b: for all $\beta>0$, $|p_{1,n}-i|\geq t_n^{\beta}$ for $n$ large enough.
\end{itemize}

(Here we assume again that the points $p_{i,n}$ are ordered by increasing imaginary
part as in section \ref{section-setup}.)
Roughly speaking, in Case 3a, all points $p_{j,n}$ converge to $i$ quickly, whereas
in Case 3b, at least two ($p_{1,n}$ and $p_{N,n}$ by symmetry) converge to $i$ very slowly.
We will see (proposition~\ref{proposition-case3a}) that $N=1$ and $p_{1,n}=i$
in Case 3a, and (proposition~\ref{proposition-case3b}) that Case 3b is impossible.

In both cases, we make a blowup at $i$ as follows :
Let $\varphi:\sS^2\to\sS^2$ be the rotation of angle $\pi/2$ which fixes the $Y$ circle and maps
$i$ to $0$. Explicitly, in our model of $\sS^2(1)$
\[
\varphi(z)=\frac{z-i}{1-iz},\qquad \varphi^{-1}(z)=\frac{z+i}{1+iz}.
\]
It exchanges the equator $E$ and the great circle $X$.
$\varphi$ lifts in a natural way to an isometry $\widehat{\varphi}$ of $\sS^2\times\rR$.
We first apply the isometry $\widehat{\varphi}$ and
then we scale by $1/\mu_n$ where the ratio $\mu_n$ goes to zero and will be chosen later, depending on the case.
Let
\[
\whM_n=\frac{1}{\mu_n}\widehat{\varphi}(M_n)\subset \sS^2(1/\mu_n)\times\rR,
\]
\[
\whOmega_n=\frac{1}{\mu_n}\varphi(\Omega),\qquad
\whp_{i,n}=\frac{1}{\mu_n}\varphi(p_{i,n}),\qquad
\wht_n=\frac{t_n}{\mu_n}.
\]
The minimal surface $\whM_n$ is the graph over $\whOmega_n$ of the function
\[
\whf_n(z)=\frac{1}{\mu_n}f_n(\varphi^{-1}(\mu_n z))=
\wht_n w_n(z)-\whu_n(z)
\]
where
\begin{equation}
\label{equation-wn}
w_n(z)=\frac{1}{2\pi} \arg\left(\frac{\mu_n z+i}{1+i\mu_n z}\right)
\end{equation}
\[
\whu_n(z)=\frac{1}{\mu_n} u_n(\varphi^{-1}(\mu_n z)).
\]

\stepcounter{theorem}
\addtocontents{toc}{\SkipTocEntry}
\subsection{Case 3a}
\label{case-3a-section}
In this case, fix some positive number $\alpha$ such that
$\alpha<\min\{\beta,\frac{1}{8}\}$,
and take $\mu_n=t_n^{\alpha}$.
Then for all $j\in[1,N]$, $|p_{j,n}-i|=o(\mu_n)$, so $\lim \whp_{j,n}=0$.
\begin{proposition}
\label{first-case-3a-proposition}
In Case 3a, after passing to a subsequence,
\begin{equation}
\label{EQ*1}
\lim\frac{|\log t_n|}{\wht_n}\left(\whu_n(z)-\whu_n(z_0)\right)=-c(\log|z|-\log |z_0|).
\end{equation}
The convergence is the uniform smooth convergence on compact subsets of $\CC\setminus\{0\}$.
(Here $z_0$ is an arbitrary fixed nonzero complex number.)
The constant $c$ is positive.
\end{proposition}
The proof is in section \ref{section44}.

\begin{remark*}
In fact
\[
\lim \frac{|\log t_n|}{\wht_n}\whu_n(z)=\infty
\]
 for all $z$, so it is necessary to substract
something to get a finite limit.
Because of this, we believe it is not possible to prove this
proposition by a barrier argument as in the proof of proposition~\ref{first-case-1-proposition}. Instead, we will
prove the convergence of the derivative $\whu_{n,z}$ using the Cauchy Pompeieu integral formula
for $C^1$ functions.
\end{remark*}

We now prove

\begin{proposition}
\label{proposition-case3a}
In Case 3a, $N=1$.
\end{proposition}

\begin{proof}
From \eqref{equation-wn},
\[
w_n(z)
=\frac{1}{2\pi}\left(\frac{\pi}{2}+O(\mu_n)\right)
=\frac{1}{4}(1+O(t_n^{\alpha})).
\]
Since $\alpha>0$, $t_n^{\alpha}\to 0$
so using Equation \eqref{EQ*1} of
proposition~\ref{first-case-3a-proposition},
\[
\whf_n(z)-\frac{\wht_n}{4}+\whu_n(z_0)\simeq  c\frac{\wht_n}{|\log t_n|}(\log |z|-\log|z_0|).
\]
From this we conclude that for $n$ large enough,
the level curves of $\whf_n$ are convex.
Back to the original scale, we have found a horizontal convex curve $\gamma_n$
which encloses $N$ catenoidal necks and is invariant under reflection in the vertical cylinder $E\times\rR$.
In particular, this curve $\gamma_n$ is a graph on each side of
$E\times\rR$. Consider the domain on $M_n$
which is bounded by $\gamma_n$ and its symmetric image with respect to the $Y$-circle.
By Alexandrov reflection
(see the proof of proposition~\ref{proposition-alexandrov} in the appendix),
 this domain must be symmetric with respect to the vertical cylinder $E\times\rR$ -- which we already know -- and must be a graph on each side of $E\times\rR$. This implies that the centers of all necks must be on the circle $E$. But $E\cap Y^+$ is a single point.
Hence there is only one neck: $N=1$.
\end{proof}

\stepcounter{theorem}
\addtocontents{toc}{\SkipTocEntry}
\subsection{Case 3b}
\label{case-3b-section}
In this case we take $\mu_n=|p_{1,n}-i|$.
After passing to a subsequence, the limits
\begin{equation}\label{eq:case-3-limits}
\whp_j=\lim \whp_{j,n}\in  \left[ \frac{-i}{2},\frac{i}{2} \right]  
\end{equation}
exist for all $j\in[1,N]$. Moreover, we have
\[
\whp_1=\frac{-i}{2} \text{ and }\, \whp_N=\frac{i}{2}.
\]
(The $\frac{1}{2}$ comes from the fact that the rotation $\varphi$ distorts Euclidean lengths by the 
 factor $\frac{1}{2}$ at $i$.)
Let $m$ be the number of distinct points amongst $\whp_1,\dots,\whp_N$.
Observe that $m\geq 2$ because we know that $\whp_1$ and $\whp_N$
are distinct.
For each $n$, relabel the points $\whp_{i,n}$ (by permuting the indices)
so that the points $\whp_1,\dots,\whp_m$ defined by \eqref{eq:case-3-limits} are distinct and
so that
\[
\Im \whp_1<\Im \whp_2< \dots<\Im \whp_m.
\]
\begin{proposition}
\label{first-case-3b-proposition}
In Case 3b, after passing to a subsequence,
\[
\wtu(z):=\lim\frac{|\log t_n|}{\wht_n}\left(\whu_n(z)-\whu_n(z_0)\right)=
\sum_{i=1}^m -\whc_i(\log|z-\whp_i|-\log|z_0-\whp_i|).
\]
The convergence is the uniform smooth convergence on compact subsets of $\CC$
minus the points $\whp_1,\dots,\whp_m$.
(Here $z_0$ is an arbitrary fixed complex number different from these points.)
The constants $\whc_i$ are positive.
\end{proposition}
The proof of this proposition is in section~\ref{section45}.


Fix some small number $\varepsilon>0$.
Let $F_n$ be the flux of the Killing field $\chi_Y$ on the circle $C(\whp_1,\varepsilon)$
on $\whM_n$.
Because of the scaling we are in $\sS^2(1/\mu_n)\times\rR$ so
\[
\chi_Y(z)=\frac{i}{2}(1-\mu_n^2 z^2).
\]
Hence  using proposition~\ref{proposition-flux2}
in the appendix,
\begin{equation}
\label{eq-david2}
F_n=-\Im\int_{C(\whp_1,\varepsilon)}
2\left(\wht_n w_{n,z}-\whu_{n,z}\right)^2\frac{i}{2}(1-\mu_n^2 z^2)+O((\wht_n )^4).
\end{equation}
Expand the square. Then as in Case 1, the cross-product term
can be neglected, so the leading term is the one involving $(\whu_{n,z})^2$
and since $\mu_n\to 0$:
\begin{proposition}
\label{second-case-3b-proposition}
\begin{equation}
\label{EQ*3}
\lim \left(\frac{\log t_n}{\wht_n}\right)^2 F_n=
-\lim \left(\frac{\log t_n}{\wht_n}\right)^2\Re\int_{C(\whp_1,\varepsilon)}(\whu_{n,z})^2 \,dz.
\end{equation}
\end{proposition}
This proposition is proved in section \ref{section46}. The proof is similar to
the proof of proposition~\ref{third-case-1-proposition}.

We now prove
\begin{proposition}
\label{proposition-case3b}
Case 3b is impossible.
\end{proposition}

\begin{proof}
According to Lemma~\ref{flux-zero-lemma}, the flux $F_n$ is equal to zero.
Hence the left-hand side of \eqref{EQ*3} is zero.
By propositions \ref{first-case-3b-proposition} and \ref{second-case-3b-proposition},
\begin{equation}
\label{EQ*4}
0=-\Re\int_{C(\whp_1,\varepsilon)}(\wtu_z)^2.
\end{equation}
On the other hand,
\[
\wtu_z=-\sum_{i=1}^m \frac{\whc_i}{2(z-\whp_i)}
\]
\[
\Res_{\whp_1}(\wtu_z)^2=\frac{1}{2}\sum_{i=2}^m \frac{\whc_1 \whc_i}{\whp_1-\whp_i}.
\]
Write $\whp_i=i y_i$, then
\[
-\int_{C(\whp_1,\varepsilon)}(\wtu_z)^2
=-\pi\sum_{i=2}^m\frac{\whc_1 \whc_i}{y_1-y_i}.
\]
Since $m\geq 2$, $y_1<y_i$ for all $i\geq 2$ and $\whc_i>0$ for all $i$ by proposition
\ref{first-case-3b-proposition}, this is positive, contradicting \eqref{EQ*4}.
\end{proof}


This completes the proof of the main theorem, modulo the proof of
propositions \ref{first-case-3a-proposition}, \ref{first-case-3b-proposition} 
and \ref{second-case-3b-proposition}, which
were used in the analysis of Cases 3a and 3b.
We prove these propositions in sections \ref{section44}, 
\ref{section45} and \ref{section46} respectively, 
using an estimate that we prove in the next section.

\stepcounter{theorem}
\addtocontents{toc}{\SkipTocEntry}
\subsection{An estimate of \texorpdfstring{$\int |\nabla u_n|$}{Lg}}
\label{integral-gradient-section}
By proposition~\ref{first-case-1-proposition}, we have, since all points
$p_{j,n}$ converge to $i$,
\[
\lim \frac{|\log t_n|}{t_n}u_n=c_0\arg z-c_1\log\left|\frac{\log z-\log i}
{\log z+\log i}\right|.
\]
Moreover, $c_1$ is positive by proposition~\ref{second-case-1-proposition}.
The convergence is the smooth convergence on compact subsets of
$\wtC*\setminus\{i,-i\}$.
From this we get, for fixed $\varepsilon>0$,
\begin{equation}
\label{equation-gradient-integral}
\int_{C(i,\varepsilon)}|\nabla u_n|\leq C \frac{t_n}{|\log t_n|}.
\end{equation}
Let $i\in[1,n]$ be the index such that $r_{i,n}=\max\{r_{j,n}\,:\,1\leq j\leq N\}$.
Let $\phi_n=\phi_{i,n}$ be the vertical flux of $M_n$ on the graph of $f_n$ restricted to
$C(p_{i,n},\varepsilon)$. By the last point of proposition~\ref{first-case-1-proposition}, we have
\[
\phi_n\leq C \frac{t_n}{|\log t_n|}
\]
for some constant $C$.
We use proposition~\ref{proposition-height2} with $r_1=\lambda r_{i,n}$ and $r_2=\varepsilon$ as in
the proof of proposition~\ref{second-case-1-proposition}, and
\[
r'_1=(t_n)^{1/4},\qquad r'_2=(t_n)^{1/8}
\]
The proposition tells us that for each $j\in[1,N]$,
there exists a number $r$, which we call
$r'_{j,n}$, such that
\begin{equation}
\label{eq-david1}
(t_n)^{1/4}\leq r'_{j,n}\leq (t_n)^{1/8}
\end{equation}
and
\[
\int_{C(p_{j,n},r'_{j,n})\cap\Omega_n}|\nabla f_n|\leq
\sqrt{8}\phi_n\left(\log\frac{\varepsilon}{\lambda r_{i,n}}\right)^{1/2}
\left(\log\frac{(t_n)^{1/8}}{(t_n)^{1/4}}\right)^{-1/2}.
\]
Using \eqref{equation-minoration-rin}, we have
\[
\log \frac{\varepsilon}{\lambda r_{i,n}}
\leq \log \frac{\varepsilon |\log t_n|}{\lambda C_1 t_n}\leq C_2|\log t_n|
\]
for some positive constants $C_1$ and $C_2$. This gives
\[
\int_{C(p_{j,n},r'_{j,n})\cap\Omega_n}|\nabla f_n|\leq C\phi_n\leq C\frac{t_n}{|\log t_n|}.
\]
Now since $|\nabla \arg z|\simeq 1$ near $i$,
\[
\int_{C(p_{j,n}, r'_{j,n})} t_n|\nabla \arg z|\leq C t_n^{1+1/8}=o(\frac{t_n}{|\log t_n|}).
\]
Hence
\[
\int_{C(p_{j,n},r'_{j,n})\cap\Omega_n}|\nabla u_n|\leq C\frac{t_n}{|\log t_n|}.
\]
Consider the domain
\begin{equation}
\label{eqUn}
U_n=D(i,\varepsilon)\setminus\bigcup_{j=1}^N \overline{D}(p_{j,n},r'_{j,n}).
\end{equation}
Since $r'_{j,n}\gg t_n\gg r_{j,n}$, we have $\overline{U_n}\subset\Omega_n$ and
\begin{equation}
\label{equation-distance-bord}
d(U_n,\partial \Omega_n)\geq \frac{1}{2} (t_n)^{1/4}.
\end{equation}
Also, since $\partial U_n\subset\Omega_n$,
\[
\partial U_n\subset C(i,\varepsilon)\cup\bigcup_{j=1}^N (C(p_{j,n},r'_{j,n})\cap\Omega_n).
\]
This implies
\begin{equation}
\label{equation-integral-gradient}
\int_{\partial U_n} |\nabla u_n|\leq C\frac{t_n}{|\log t_n|}.
\end{equation}
This is the estimate we will use in the next sections.

\stepcounter{theorem}
\addtocontents{toc}{\SkipTocEntry}
\subsection{Proof of proposition~\ref{first-case-3a-proposition} (Case 3a)}
\label{section44}
Let $\beta>0$ be the number given by the hypothesis of case 3a.
Recall that we have fixed some positive number $\alpha$ such that
$0<\alpha<\min\{\beta,\frac{1}{8}\}$, that $\mu_n=t_n^{\alpha}$,
$\wht_n=\frac{t_n}{\mu_n}$ and
$\varphi_n=\frac{1}{\mu_n}\varphi$.
Let $U_n$ be the domain defined in \eqref{eqUn}
and $\whU_n=\varphi_n(U_n)$.
Since $\mu_n\gg t_n^{1/8}\geq r'_{j,n}$
by \eqref{eq-david1},
we have
\[
\lim \whU_n=\CC^*.
\]
Since $\varphi_n$ is conformal, we have, using \eqref{equation-integral-gradient}
(recall the definition of $\whu_n$ in \eqref{equation-wn})
\[
\int_{\partial \whU_n}|\nabla \whu_n|
=\int_{\partial\whU_n}\frac{1}{\mu_n}|\nabla (u_n\circ\varphi_n^{-1})|
=\frac{1}{\mu_n}\int_{\partial U_n}|\nabla u_n|\leq C\frac{t_n}{\mu_n|\log t_n|}=C\frac{\wht_n}{|\log t_n|}.
\]
Using \eqref{equation-distance-bord}, we have
\[
d(\whU_n,\partial\whOmega_n)\geq \frac{(t_n)^{1/4}}{4\mu_n}.
\]

By standard interior estimates for the minimal surface equation
 (see proposition~\ref{proposition-schauder} in the appendix),
\[
|\Delta \whu_n|=|\Delta \whf_n|\leq C \frac{(\wht_n)^3}
{((t_n)^{1/4}/(4\mu_n))^{4}}=C\mu_nt_n^{2}\qquad\mbox{ in } \whU_n.
\]
Let
\[
\wtu_n=\frac{|\log t_n|}{\wht_n} (\whu_n-\whu_n(z_0)).
\]
Proposition~\ref{first-case-3a-proposition} asserts that a subsequence of the $\tilde u_n$ converge to
$-c(\log |z| -\log |z_0|)$, where $c$ is a real positive constant.
By the above estimates,
\begin{equation}
\label{estimate-wtun}
\int_{\partial \whU_n}|\nabla \wtu_n|\leq C
\end{equation}
and
\begin{equation}
\label{equation-Delta-wtun}
|\Delta \wtu_n|\leq C\mu_n^2 t_n|\log t_n|\qquad\mbox{ in } \whU_n.
\end{equation}
Let $K$ be a compact set of $\CC^*$.
For $n$ large enough, $K$ is included in $\whU_n$.
The Cauchy Pompeieu integral formula
(see \eqref{eq-pompeieu} in the appendix)
gives for $\zeta\in K$
\[
\wtu_{n,z}(\zeta)=\frac{1}{2\pi i}\int_{\partial \whU_n}\frac{\wtu_{n,z}(z)}{z-\zeta}\,dz
+\frac{1}{8\pi i}\int_{\whU_n}\frac{\Delta\wtu_n(z)}{z-\zeta} \,dz\wedge\overline{dz}.
\]
We estimate each integral in the obvious way, using \eqref{estimate-wtun}
in the first line and \eqref{equation-Delta-wtun} in the third line:
\[
\left|\int_{\partial\whU_n}\frac{\wtu_{n,z}}{z-\zeta}\right|
\leq
\frac{1}{d(\zeta,\partial \whU_n)}\int_{\partial\whU_n}
|\nabla\wtu_n|\leq \frac{C}{d(\zeta,\partial \whU_n)}\to \frac{C}{|\zeta|}.
\]
\[
\int_{\whU_n}\frac{dx\,dy}{|z-\zeta|}\leq
\int_{D(0,\varepsilon/\mu_n)}\frac{dx\,dy}{|z-\zeta|}
\leq 2\pi\int_{r=0}^{2\varepsilon/\mu_n}\frac{rdr}{r}
=4\pi\frac{\varepsilon}{\mu_n}.
\]
\[
\left|\int_{\whU_n}\frac{\Delta\wtu_n}{z-\zeta} \,dx\,dy\right|
\leq C\mu_n t_n|\log t_n|\to 0.
\]
Hence for $n$ large enough, we have in $K$
\[
|\wtu_{n,z}(\zeta)|   \leq    \frac{C}{|\zeta|}
\]
for a constant $C$ independent of $K$.
Passing to a subsequence, $\wtu_{n,z}$ converges smoothly on compact sets
of $\CC^*$ to a holomorphic function with a zero at $\infty$ and at most a simple pole at $0$.
(The fact that the limit is holomorphic follows from \eqref{equation-Delta-wtun}.)
Hence 
\[
\lim\wtu_{n,z}=\frac{c}{2z}
\]
for some constant $c$.
Recalling that $(\log|z|)_z =\frac{1}{2z}$, this gives \eqref{EQ*1} of proposition~\ref{first-case-3a-proposition}.
It remains to prove that $c>0$.
Let $\widehat{\phi}_n$ be the vertical flux on the closed curve of $\whM_n$ that is the graph of
$\whf_n$ over the circle $C(0,1)\subset\CC^*$.
Then by the computation at the end of the proof of proposition~\ref{first-case-1-proposition} in
 section~\ref{section:proof-of-first-case-1-proposition} (after Lemma~\ref{lemma-positive-harmonic}),
\[
    \lim\frac{|\log t_n|}{\wht_n}\widehat{\phi}_n=2\pi c.
\]
Now by scaling and homology invariance of the flux,
$\widehat{\phi}_n=\frac{\phi_{1,n}}{\mu_n}$, where $\phi_{1,n}$ is
the vertical flux on the closed curve of $M_n$ that is the graph of $f_n$ over
the circle $C(i,\varepsilon)$. Hence
$c=c_1$ and $c_1$ is positive by proposition~\ref{first-case-1-proposition}.
\QED

\stepcounter{theorem}
\addtocontents{toc}{\SkipTocEntry}
\subsection{Proof of proposition~\ref{first-case-3b-proposition} (Case 3b)}
\label{section45}
Recall that in Case 3b, $\mu_n=|p_{1,n}-i|$ and for all $\beta>0$, 
$\mu_n\geq t_n^{\beta}$ for $n$ large enough.
Let $U_n$ be the domain defined in \eqref{eqUn}.
Since $\mu_n\gg t_n^{1/8}\geq r'_{j,n}$
by \eqref{eq-david1},
we have
\[
\lim \whU_n=\CC\setminus \{\whp_1,\dots,\whp_m\}.
\]
(Compare with case 3a, where the limit is $\CC^*$.)
Define again
\[
\wtu_n=\frac{|\log t_n|}{\wht_n} (\whu_n-\whu_n(z_0)).
\]
By the same argument as in section \ref{section44} we obtain that
$\wtu_{n,z}$ converges on compact subsets of $\CC\setminus\{\whp_1,\dots,\whp_m\}$
to a meromorpic function with at most simple poles at $\whp_1,\dots,\whp_m$ and
a zero at $\infty$, so
\[
\lim\wtu_{n,z}=\sum_{j=1}^m\frac{\whc_j}{2(z-\whp_j)}.
\]
It remains to prove that the numbers $\whc_1,\dots,\whc_m$ are positive.
For $1\leq j\leq m$, let $\whphi_{j,n}$ be the vertical flux of $\whM_n$ on the
graph of $\whf_n$ restricted to the circle $C(\whp_j,\varepsilon)$.
Then by the computation at the end of the proof of proposition~\ref{first-case-1-proposition},
we have
\[
\lim\frac{|\log t_n|}{\wht_n}\whphi_{j,n}=2\pi \whc_j.
\]
We will prove that $\whc_j$ is positive by estimating the vertical flux using the
height estimate as in section~\ref{section:proof-of-third-case-1-proposition}.
Take $\beta=\frac{1}{18(N+2)}$ and let
\[
B_n=\bigcup_{j=1}^N D(p_{j,n},t_n^{\beta}).
\]
By Lemma~\ref{lemma5} with $\alpha=\frac{1}{2}$, we have for $n$ large enough:
\[
u_n\leq (N+2)\frac{\beta}{\alpha}=\frac{t_n}{9}
\quad\text{ in } D(i,\varepsilon)\setminus B_n.
\]
(Lemma~\ref{lemma5} gives us this estimate for $|z|\leq 1$. The result follows because
$u_n$ is symmetric with respect to the unit circle).
Consequently, the level set
$u_n=\frac{t_n}{8}$ is contained in $B_n$.
By the hypothesis of Case 3b, for $n$ large enough, $\mu_n\gg t_n^{\beta}$ so
 the disks $D(p_{j,n},t_n^{\beta})$
for $1\leq j\leq m$ are disjoint. Hence $B_n$ has at least $m$ components.
Let $\Gamma_{j,n}$ be the component of the level set $u_n=\frac{t_n}{8}$ which
encloses the point $p_{j,n}$ and $D_{j,n}$ the disk bounded by $\Gamma_{j,n}$. Then $D_{j,n}$ contains no other point
$p_{k,n}$ with $1\leq k\leq m$, $k\neq j$. (It might contain points $p_{k,n}$ with
$k>m$).
The proof of proposition~\ref{second-case-1-proposition} in section~\ref{section:proof-of-second-case-1-proposition}
gives us a point $p_{k,n}\in D_{j,n}$ (with either $k=j$ or $k>m$ and $\whp_k=\whp_j$) such that
\[
r_{k,n}\geq C\frac{t_n}{|\log t_n|}
\]
for some positive constant $C$.
Scaling by $1/\mu_n$,
this implies that
\[
\whphi_{j,n}\geq 2\pi \frac{C}{2}\frac{\wht_n}{|\log t_n|}.
\]
Hence $\whc_j>0$.
\QED

\stepcounter{theorem}
\addtocontents{toc}{\SkipTocEntry}
\subsection{Proof of proposition~\ref{second-case-3b-proposition} (Case 3b)}
\label{section46}
Let $g_n=\whu_{n,z}$.
We have to prove that the cross-product term in \eqref{eq-david2} can
be neglected, namely:
\[
\Re \int_{C(\whp_1,\varepsilon)} w_{n,z}(z) g_n(z)(1-\mu_n^2z^2)\,dz  
=o\left(\frac{\wht_n}{(\log t_n)^2}\right).
\]
The proof of this fact is the same as the proof of proposition~\ref{third-case-1-proposition}
in section~\ref{section:proof-of-third-case-1-proposition}, with the following modifications:
\begin{enumerate}[\upshape$\bullet$]
\item $\arg z$ is replaced by the function $w_n$ defined in \eqref{equation-wn}, so its
derivative $\frac{1}{2 iz}$ is replaced by $w_{n,z}$.
\item $1-z^2$ is replaced by $1-\mu_n^2 z^2$.
\item $t_n$, $u_n$, etc... now have hats: $\wht_n$, $\whu_n$, etc...
\item From
\[
w_{n,z}=\frac{1}{4\pi i}\left(\frac{1}{\mu_n z+i}-\frac{i}{1+i\mu_n z}\right)
\]
we deduce that $|w_{n,z}|$ is bounded in $D(\whp_1,\varepsilon)$ and
since $\whp_{j,n}\in i\rR$, that
$w_{n,z}(\whp_{j,n})$ is real, which is what we need to ensure that the term $a_{j,1}$ does
not contribute to the integral (see \eqref{equation-aj1}).
\end{enumerate}
\QED

\appendix
\section{Auxiliary results}
This appendix contains several results about minimal surfaces
in $\sS^2\times\rR$ that have been used in the proof of theorem~\ref{maintheorem-N}.
Some of these results are true for minimal surfaces in the Riemannian product
$M\times\rR$ where $(M,g)$ is a 2-dimensional Riemannian manifold.
These results are local, so we can assume without loss of generality that
$M$ is a domain $\Omega\subset\CC$ equipped with a conformal
metric $g=\lambda^2|dz|^2$, where $\lambda$ is a smooth positive function
on $\overline{\Omega}$. Given a function $f$ on $\Omega$, the graph of
$f$ is a minimal surface in $M\times\rR$ if it satisfies the minimal surface equation
\begin{equation}
\label{mse}
\div_g \frac{\nabla_g f}{W}=0\quad\mbox{ with }
W=\sqrt{1+\|\nabla_g f \|^2_g}
\end{equation}
where the subscript $g$ means that the quantity is computed with respect to the
metric $g$, so for instance
\[
\nabla_g f =\lambda^{-2}\nabla f,\qquad
\div_g X=\lambda^{-2}\div(\lambda^2 X).
\]
In coordinates, \eqref{mse} gives the equation
\begin{equation}
\label{msecoord}
(1+\lambda^{-2} f_y^2)f_{xx}+(1+\lambda^{-2} f_x^2)f_{yy}-2\lambda^{-2}
f_x f_y f_{xy}+(f_x^2+f_y^2)\left(
\frac{\lambda_x}{\lambda} f_x+\frac{\lambda_y}{\lambda} f_y\right)=0.
\end{equation}
Propositions \ref{proposition-schauder}, \ref{proposition-flux2}, \ref{proposition-height}
and \ref{proposition-height2}
will be formulated in this setup.

\addtocontents{toc}{\SkipTocEntry}
\section*{Interior gradient and Laplacian estimate}

\begin{proposition}
\label{proposition-schauder}
Let $\Omega$ be a domain in $\CC$ equipped with a smooth conformal metric
$g=\lambda^2 |dz|^2$.
 Let $f:\Omega\to\rR$ be a solution of
the minimal surface equation \eqref{mse}. Assume that $|f|\leq t$ in $\Omega$ and
$\| \nabla f \|\leq 1$. Then
\[
\|\nabla f(z)\|\leq \frac{Ct}{d(z)}
\]
\[
|\Delta f(z)|\leq \frac{Ct^3}{d(z)^4}
\]
for all $z\in\Omega$ such that $d(z)\geq t$.
Here, $d(z)$ denotes the Euclidean distance to the boundary of $\Omega$.
The gradient and Laplacian are for the Euclidean metric.
The constant $C$ only depends on the diameter of $\Omega$ and on a bound
on $\lambda$, $\lambda^{-1}$ and its partial derivatives of first and second order.
\end{proposition}

\begin{proof}
Let us write the minimal surface equation \eqref{msecoord} as $L(f)=0$,
where $L$ is a second order linear elliptic operator whose
coefficients depend on $f_x$ and $f_y$.
Theorem 12.4 in Gilbarg-Trudinger~\cite{gilbarg-trudinger} gives us a uniform constant $C$
and $\alpha>0$ such that (with the Gilbarg-Trudinger notation)
\[
[Df]_{\alpha}^{(1)}\leq C \, \|f\|_0\leq Ct.
\]
If $d(z,\partial\Omega)\geq t$, this implies
\[
[Df]_{\alpha}^{(0)}\leq \frac{Ct}{t}=C.
\]
Then we have the required $C^{\alpha}$ estimates of the coefficients of $L$ to apply the interior Schauder estimate (theorem~6.2 in~\cite{gilbarg-trudinger}):
\[
|D^kf(z)|\leq \frac{C}{d(z)^k} \|f\|_{0}\leq C\frac{t}{d(z)^k},\qquad k=0,1,2.
\]
The minimal surface equation \eqref{msecoord} implies
\[
|\Delta f|\leq C(|Df|^2 |D^2f|+ |Df|^3)\leq C\frac{t^3}{d^4}.
\]
\end{proof}

\addtocontents{toc}{\SkipTocEntry}
\section*{Alexandrov moving planes}
We may use the Alexandrov reflection technique in $\sS^2\times\rR$ with the role of horizontal planes played by the level spheres $\sS^2\times\{t\}$, and the role of
vertical planes played by a family of totally geodesic cylinders.
Specifically, let
$E\subset\sS^2\times\{0\}$ be the closed geodesic that is the equator with respect to the antipodal points $O$, $O^*$, let $X\subset\sS^2\times\{0\}$ be a geodesic passing through
$O$ and $O^*$, and define $E_{\theta}$ to be the rotation of $E=E_0$ through an angle 
$\theta$ around the poles $E\cap X$. The family of geodesic cylinders
\[
E_{\theta}\times\rR, \quad -\pi/2\leq\theta<\pi/2,
\]
when restricted to the complement of $(E\cap X)\times\rR$ is a foliation.

\begin{proposition}
\label{proposition-alexandrov}
Let $\Gamma=\gamma_1\cup\gamma_2$ with each $\gamma_i$ a $C^2$
Jordan curve in $\sS^2\times\{t_i\}$, $t_1\neq t_2$, that is invariant under reflection
in $\Pi=E\times\rR$. Suppose further that each component of $\gamma_i\setminus\Pi$
is a graph over $\Pi$ with locally bounded slope. Then any minimal surface
$\Sigma$ with $\partial\Sigma=\Gamma$ that is disjoint from at least one of the
vertical cylinders $E_{\theta}\times\rR$, must be symmetric with respect to reflection in
$\Pi$, and each component of $\Sigma\setminus\Pi$ is a graph of locally bounded slope
over a domain in $\Pi$.
\end{proposition}
(Given a domain ${\mathcal O}\subset\Pi$ and a function 
$f:{\mathcal O}\to[-\pi/2,\pi/2)$, the graph of $f$ is the set of points
$\{\mbox{rot}_{f(p)} p\;:\; p\in{\mathcal O}\}$, where $\mbox{rot}_{\theta}$ is the rotational symmetry that takes
$\Pi$ to $E_{\theta}\times\rR$.)


The proof is the same as the classical proof for minimal surfaces in $\rR^3$ 
using Alexandrov reflection as described above.
(See, for example,~\cite{schoen-uniqueness} corollary 2.)

\addtocontents{toc}{\SkipTocEntry}
\section*{Flux}
Let $N$ be a Riemannian manifold, $M\subset N$ a minimal surface and $\chi$ a Killing field on $N$.
Let $\gamma$ be a closed curve on $M$ and $\mu$ be the conormal along $\gamma$.
Define
\[
            \flux_{\chi}(\gamma)=\int_{\gamma}\langle \mu,\chi\rangle \, ds.
\]
It is well know that this only depends on the homology class of $\gamma$.

\begin{proposition}
\label{proposition-flux1}
In the case $N=\sS^2(R)\times\rR$,
the space of Killing fields is 4 dimensional. It is generated by
the vertical unit vector $\xi$, and the following three horizontal vectors fields:
\begin{align*}
\chi_X(z) &= \frac{1}{2}(1+\frac{z^2}{R^2}) \\
\chi_Y(z) &= \frac{i}{2}(1-\frac{z^2}{R^2})  \\
\chi_E(z) &= \frac{iz}{R}.
\end{align*}
These vector fields are respectively unitary tangent to the great circles
$X$, $Y$, and $E$.
They are generated by the one-parameter families of rotations  about the poles whose equators are these great circles.
\end{proposition}

\begin{proof}
The isometry group of $\sS^2(R)\times\rR$ is well known to be 4-dimensional.
Recall that our model of $\sS^2(R)$ is $\CC\cup{\infty}$ with the conformal metric
$\frac{2R^2}{R^2+|z|^2}|dz|$.
By differentiating the 1-parameter group $z\mapsto e^{it}z$ of isometries of $\sS^2$, we obtain the horizontal Killing field $\chi(z)=iz$, which suitably normalized gives $\chi_E$.
Let
\[
\varphi(z)=\frac{Rz+iR^2}{iz+R}.
\]
This corresponds, in our model of $\sS^2(R)$, to the rotation about the $x$-axis of
angle $\pi/2$. It maps the great circle $E$ to the great circle $X$. We transport $\chi_E$ by this isometry
to get the Killing field $\chi_X$: a short computation gives
\[
\chi_X(z)=\varphi_*\chi_E(z)=\varphi'(\varphi^{-1}(z))\chi_E(\varphi^{-1}(z))=\frac{z^2+R^2}{2R^2}.
\]
Then we transport $\chi_X$ by the rotation $\psi(z)=iz$ to get the Killing field $\chi_Y$:
\[
\chi_Y(z)=\psi_*\chi_X(z)=i \frac{(-iz)^2+R^2}{2R^2}.
\]
\end{proof}

\begin{proposition}
\label{proposition-flux2}
Let $\Omega\subset\CC$ be a domain equipped with a conformal metric
$g=\lambda^2 |dz|^2$.
Let $f:\Omega\to\rR$ be a solution of the minimal surface equation \eqref{mse}.
Let $\gamma$ be a closed, oriented curve in $\Omega$ and $\nu$ be the Euclidean 
exterior normal vector along $\gamma$ (meaning that $(\gamma',\nu)$ is a negative orthonormal basis).
Let $M$ be the graph of $f$ and let $\widetilde{\gamma}$ be the closed curve in
$M$ that is the graph of $f$ over $\gamma$.
\begin{enumerate}
\item For the vertical unit vector $\xi$,
\[
\flux_{\xi}(\widetilde{\gamma})=
\int_{\gamma}\frac{\langle \nabla f,\nu\rangle}{W}
\]
where $W$ is defined in equation \eqref{mse}.
(Here the gradient, scalar product and line element are Euclidean.)
If $\|\nabla f \|$ is small, this gives
\[
\flux_{\xi}(\widetilde{\gamma})
=\Im\int_{\gamma} \left(2f_z + O(|f_z|^2)\right)\,dz
\]
\item If $\chi$ is a horizontal Killing field, 
\[
\flux_{\chi}(\widetilde{\gamma})=-\Im\int_{\gamma} \left(2(f_z)^2\chi(z) +O(|f_z|^4|)\right)\,dz.
\]
\end{enumerate}
\end{proposition}

\begin{proof}
Let $(N,g)$ be the Riemannian manifold $\Omega\times\rR$ equipped with
the product metric $g=\lambda^2 |dz|^2+dt^2$.
Let $M$ be the graph of $f$, parametrized by
\[
\psi(x,y)=(x,y,f(x,y)).
\]
The unit normal vector to $M$ is
\[
n=\frac{1}{W}\left(-\lambda^{-2}f_x,-\lambda^{-2}f_y,1\right).
\]
Assume that $\gamma$ is given by some parametrization $t\mapsto\gamma(t)$, fix some
time $t$ and let $(X,Y)=\gamma'(t)$. Then
\[
d\psi(\gamma')=(X,Y,X f_x+Y f_y)
\]
is tangent to $\psi(\gamma)$ and its norm is $ds$, the line element on $M$.
We need to compute the conormal vector in $N$.
The linear map $\varphi:(T_p N,g)\to (\rR^3,\mbox{Euclidean})$ defined by
\[
\varphi(u_1,u_2,u_3)=(\lambda u_1,\lambda u_2,u_3)
\]
is an isometry. Let $u=(u_1,u_2,u_3)$ and $v=(v_1,v_2,v_3)$ be two orthogonal vectors
in $T_p N$. Let
\[
w
= 
\varphi^{-1}(\varphi(u)\wedge\varphi(v))
=
\left(\begin{array}{l}
u_2 v_3-u_3 v_2\\u_3 v_1-u_1 v_3\\ \lambda^2 (u_1 v_2 - u_2 v_1)\end{array}
\right).
\]
Then $(u,v,w)$ is a direct orthogonal basis of $T_p N$ and
$\|w\|=\|u\|\;\|v\|$.
We use this with $u=d\psi(\gamma')$, $v=n$. Then $w=\mu \,ds$, where $\mu$
is the conormal to $\psi(\gamma')$. This gives
\[
\mu\, ds=\frac{1}{W}\left(\begin{array}{l}
Y+\lambda^{-2}f_y(X f_x+Y f_y)\\
-X-\lambda^{-2} f_x(X f_x+Y f_y)\\
-f_y X+f_x Y\end{array}\right).
\]
For the vertical unit vector $\xi=(0,0,1)$, this gives
\[
\flux_\xi(\widetilde{\gamma})=
\int_{\gamma} \frac{-f_y  dx+ f_x dy}{W}=
\int_{\gamma} \frac{\langle \nabla f,\nu\rangle}{W}.
\]
The second formula of point (1) follows from $W=1+O(\|\nabla f\|^2)$ and
\[
\Im (2 f_z dz)=\Im \left((f_x -i f_y)(dx+idy)\right)=f_x dy -f_y dx.
\]
To prove point (2), let $\chi$ be a horizontal Killing field, seen as a complex number. Then
\[
\langle\chi,\mu ds\rangle_g
=\lambda^2\Re\left(
\frac{\chi}{W}(Y+iX+\lambda^{-2}(f_y+if_x)(X f_x+Y f_y)\right)
\]
Hence
\[
\flux_{\chi}(\widetilde{\gamma})=\Re\int_{\gamma}
\frac{\lambda^2\chi}{W}(dy+i\,dx)
+\frac{\chi}{W}(f_y+if_x)(f_x dx+f_y dy).
\]
We then expand $1/W$ as a series
\[
\frac{1}{W}=1-\frac{1}{2}\lambda^{-2}(f_x^2+f_y^2)+O(|\nabla f|^4).
\]
This gives after some simplifications
\[
\flux_{\chi}(\widetilde{\gamma})=\Re\int_{\gamma}
\lambda^2\chi(dy+i\,dx)
+\Re\int_{\gamma} \frac{i}{2}\chi(f_x-i f_y)^2 (dx+i \,dy)+O(|\nabla f|^4).
\]
The second term is what we want. The first term, which does not depend on $f$,
vanishes. Indeed, if $f\equiv 0$ then $M$ is $\Omega\times\{0\}$ and the flux we are computing is zero (by homology invariance of the flux, say). 
\end{proof}

\addtocontents{toc}{\SkipTocEntry}
\section*{Height estimate}

The following proposition tells us that a minimal graph with small vertical flux cannot climb very high. It is the key to estimate from below the size of the catenoidal necks.
\begin{proposition}
\label{proposition-height}
Let $\Omega\subset\CC$ be a domain that consists of a (topological) disk $D$ minus $n\geq 1$
topological disks $D_1,\dots,D_n$ contained in $D$.
We denote by $\Gamma$ the boundary of $D$ and by
$\gamma_i$ the boundary of $D_i$. Assume that $D_1$ contains $D(0,r_1)$
and $D$ is contained in $D(0,r_2)$, for some numbers $0<r_1<r_2$.
(Here $r_1$, $r_2$ are Euclidean lengths).
(See Figure \ref{figure-height-estimate}).

Assume that $\Omega$ is equipped with a conformal metric $g=\lambda^2 |dz|^2$.
Let $f:\Omega\to\rR$ be a solution of the minimal surface equation~\eqref{mse}.
Assume that
\begin{enumerate}
\item $f\equiv 0$ on $\Gamma$.
\item $f\equiv -h<0$ is constant on $\gamma_1$.
\item $f$ is constant on $\gamma_i$ for $2\leq i\leq n$, with
$-2h\leq f\leq 0$.
\item $\partial f/\partial \nu\leq 0$ on $\gamma_i$ for $1\leq i\leq n$.
\item $\|\nabla_g f \|_g\leq 1$ in $\Omega$
\end{enumerate}
Let $\phi$ be the vertical flux on $\Gamma$:
\[
\phi=\int_{\Gamma} \frac{\langle\nabla f,\nu\rangle}{W}>0.
\]
Then
\[
h\leq\frac{\sqrt{2}}{\pi}\phi\log\frac{r_2}{r_1}.
\]
\end{proposition}
(Note that Hypothesis (4) is always satisfied if $f\equiv -h$ on all $\gamma_i$ by the
maximum principle.)

\begin{figure}
\begin{center}
\includegraphics[height=35mm]{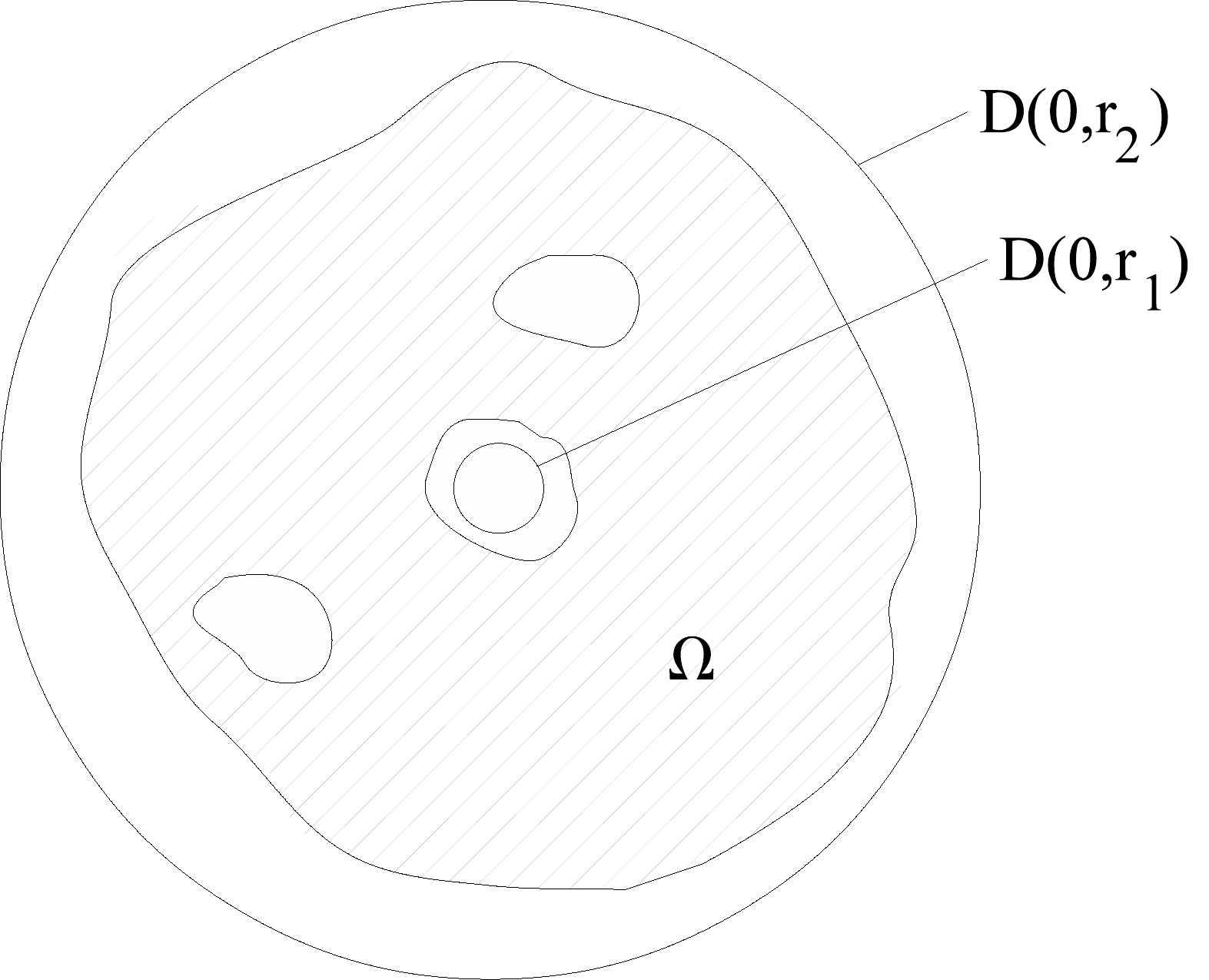}
\end{center}
\caption{}
\label{figure-height-estimate}
\end{figure}

\begin{proof}
 Let $A$ be the annulus $D(0,r_2)\setminus D(0,r_1)$.
Write $|df|$ for the norm of the Euclidean gradient of $f$.
Let $\rho$ be the function equal to $|d f|$
 on $\Omega$ and
$0$ on $\CC\setminus\Omega$. 
Then
\begin{align*}
\iint_{A} \rho^2\, dxdy
&=
\iint_{\Omega} \|\nabla_g f \|_g^2 \,d\mu_{g}\quad\mbox{ by conformal invariance of the energy}\\
&\leq \sqrt{2}\iint_{\Omega} \langle\frac{\nabla_g f}{W},\nabla_g f\rangle_g\, d\mu_g
\quad\mbox{ because $W\leq \sqrt{2}$ by hypothesis (5)}\\
&= \sqrt{2}\iint_{\Omega}\div_g(f\frac{\nabla_g f}{W})\,d\mu_g
\quad\mbox{ by the minimal surface equation \eqref{mse}}\\
&= \sqrt{2}\int_{\partial \Omega} \frac{f}{W}\langle\nabla_g f,\nu_g\rangle_g \,ds_g
\quad\mbox{ by the divergence theorem}\\
&= \sqrt{2}\int_{\partial \Omega} \frac{f}{W}\langle\nabla f,\nu\rangle
\quad\mbox{ where now all quantities are Euclidean}\\
&= \sqrt{2}\sum_{i=1}^n \int_{\gamma_i}\frac{f}{W}\langle\nabla f,\nu\rangle
\quad\mbox{ by hypothesis (1)}\\
&\leq  -2\sqrt{2}h \sum_{i=1}^n\int_{\gamma_i}\frac{\langle\nabla f,\nu\rangle}{W}
\quad\mbox{ by hypotheses (3) and (4).}
\end{align*}
Hence by homology invariance of the flux,
\begin{equation}
\label{eqq1}
\iint_{A} \rho^2 \,dx\,dy\leq 2\sqrt{2}h\phi.
\end{equation}
Consider the ray from $r_1 e^{i\theta}$ to $r_2 e^{i\theta}$.
The integral of $df$ along this ray, intersected with $\Omega$, is equal to $h$.
(If the ray happens to enter one of the disks $D_i$,
then this is true because $f$ is constant on $\partial D_i$.)
Integrating for $\theta\in[0,2\pi]$ we get
\begin{align*}
2\pi h&\leq \int_{r=r_1}^{r_2}\int_{\theta=0}^{2\pi} \rho\, dr\,d\theta\\
&= \int_A \frac{\rho}{r} \,dx\,dy\\
&\leq \left(\iint_A \rho^2 dx\,dy\right)^{1/2}\left(\iint_A \frac{1}{r^2}dx\,dy\right)^{1/2}
\quad\mbox{ by Cauchy Schwarz}\\
&\leq \left(2\sqrt{2}h\phi\right)^{1/2} \left( 2\pi \log\frac{r_2}{r_1}\right)^{1/2}
\quad\mbox{ using \eqref{eqq1}}
\end{align*}

The proposition follows.
\end{proof}


The next proposition is useful to find circles on which we have a good estimate of $\int |df|$.
\begin{proposition}
\label{proposition-height2}
Under the same hypotheses as proposition~\ref{proposition-height},
consider some point $p\in\Omega$. Given $0<r'_1<r'_2$, there exists
$r\in[r'_1,r'_2]$ such that
\[
\int_{C(p,r)\cap\Omega}\,|df|\leq \sqrt{8}\phi\left(\log\frac{r_2}{r_1}\right)^{1/2}
\left(\log\frac{r'_2}{r'_1}\right)^{-1/2}.
\]
\end{proposition}

\begin{proof}
Consider the function
\[
F(r)=\int_{C(p,r)\cap\Omega}|df|=\int_{\theta=0}^{2\pi} \rho(p+re^{i\theta})r \,d\theta.
\]
Then
\begin{eqnarray*}
\lefteqn{\min_{r'_1\leq r\leq r'_2}F(r)\log\frac{r'_2}{r'_1}}\\
&\leq&
\int_{r=r'_1}^{r'_2} \frac{F(r)}{r} \,dr\\
&=&
\int_{r=r'_1}^{r'_2}\int_{\theta=0}^{2\pi} \frac{\rho(p+re^{i\theta})}{r} r\,dr\,d\theta\\
&\leq&
\left(\int_{r'_1}^{r'_2}\int_{0}^{2\pi} \rho(p+re^{i\theta})^2 r\,dr\,d\theta\right)^{1/2}
\left(\int_{r'_1}^{r'_2}\int_{0}^{2\pi} \frac{1}{r^2} r\,dr\,d\theta\right)^{1/2}\\
&\leq&\left(\int_A\rho^2 \,dx \,dy\right)^{1/2}\left(2\pi\log\frac{r'_2}{r'_1}\right)^{1/2}\\
&\leq&\left(8\phi^2\log\frac{r_2}{r_1}\log\frac{r'_2}{r'_1}\right)^{1/2}
\quad\mbox{ using \eqref{eqq1} and proposition~\ref{proposition-height}}.
\end{eqnarray*}
The proposition follows.
\end{proof}

\addtocontents{toc}{\SkipTocEntry}
\section*{A Laurent-type formula for \texorpdfstring{$C^1$}{Lg} functions}

 \begin{proposition}
 \label{proposition-laurent}
 Let $\Omega\subset\CC$ be a domain of the form
 \[
 \Omega=D(0,R)\setminus \bigcup_{i=1}^n \overline{D}(p_i,r_i).
 \]
 Here we assume that the closed disks $\overline{D}(p_i,r_i)$ are disjoint and are included
 in $D(0,R)$.
 Let $g$ be a $C^1$ function on $\overline{\Omega}$. Then in $\Omega$,
 \[
 g(z)=g^+(z) + \sum_{i=1}^n g^-_i(z)+\frac{1}{2\pi i}\int_{\Omega}
 \frac{g_{\overline{z}}(w)}{w-z} \,dw\wedge\overline{dw}
 \]
 where
 $g^+$ is holomorphic in $D(0,R)$ and each $g^-_i$ is holomorphic
 in $\CC\setminus \overline{D}(p_i,r_i)$.
 Moreover, these functions have the following series expansions:
 \begin{align*}
 g^+(z) 
 &= \sum_{k=0}^{\infty} a_k z^k\quad 
                \mbox{ with } a_k=\frac{1}{2\pi i}\int_{C(0,R)}\frac{g(z)}{z^{k+1}}\,dz,
 \\
 g^-_i(z)
 &=  \sum_{k=1}^{\infty} \frac{a_{i,k}}{(z-p_i)^k}
        \quad \mbox{ with } a_{i,k}=\frac{1}{2\pi i}\int_{C(p_i,r_i)}g(z)(z-p_i)^{k-1}\,dz.
 \end{align*}
 The series converge uniformly in compact subsets of $\Omega$.  
 \end{proposition}
 
 \begin{remark*}
 This is the same as the Laurent series theorem except that there is a correction term which vanishes when $g$ is holomorphic. The integration circles in the formula
for $a_n$ and $a_{i,n}$ cannot be changed (as in the classical Laurent series theorem) since $g$ is not 
assumed to be holomorphic.
\end{remark*}

\begin{proof}
By the Cauchy Pompeieu integral formula for $C^1$ functions:
\begin{equation}
\label{eq-pompeieu}
 g(z)=\frac{1}{2\pi i}\int_{\partial \Omega} \frac{g(w)}{w-z}\,dw
 +
 \frac{1}{2\pi i}\int_{\Omega}\frac{g_{\overline{z}}(w)}{w-z}\,dw\wedge \overline{dw}.
\end{equation}
Define
\[
g^+(z) =  \frac{1}{2\pi i}\int_{C(0,R)}\frac{g(w)}{w-z}\,dw,
\qquad
g_i^-(z) = -\frac{1}{2\pi i}\int_{C(p_i,r_i)}\frac{g(w)}{w-z}\,dw.
\]
The function $g^+$ is holomorphic in $D(0,R)$. The function $g_i^-$ is holomorphic in
$\CC\setminus D(p_i,r_i)$ and extends at $\infty$ with $g_i^-(\infty)=0$. These two functions are expanded in power series exactly as
in the proof of the classical theorem on Laurent series (see e.g. \cite{conway}, page 107).
\end{proof}

\begin{proposition}
\label{proposition-real-residue}
Let $\Omega\subset\CC$ be a domain as in proposition~\ref{proposition-laurent}.
Let $u:\Omega\to\rR$ be a real-valued function of class $C^2$. Take
$g=\partial u/\partial z$. Then the coefficients $a_{i,1}$ which appear in
the conclusion of proposition~\ref{proposition-laurent}
are real.
\end{proposition}

\begin{proof}
\begin{align*}
\Im a_{i,1}
&=\frac{-1}{2\pi}\Re\int_{C(p_i,r_i)} u_z \,dz\\
&=\frac{-1}{4\pi}\int_{C(p_i,r_i)} u_z dz+u_{\overline{z}}\,d\overline{z}\quad
\mbox{ because $u$ is real valued}\\
&=\frac{-1}{4\pi}\int_{C(p_i,r_i)}\,du=0\quad
\mbox{ because $u$ is well defined in $\Omega$.}
\end{align*}
\end{proof}

\addtocontents{toc}{\SkipTocEntry}
\section*{Residue computation}

\begin{proposition}
\label{proposition-residue}
 \[
  \Res_p (\log z - \log p)^{-1}=p,\qquad
 \Res_p\left(\frac{1-z^2}{4z^2}\right)(\log z-\log p)^{-2} = -\frac{1+p^2}{4p}
.\]
\end{proposition}

\begin{proof}
\[
\log z - \log p=\log \left (1+\frac{z-p}{p}\right)
=\frac{z-p}{p}-\frac{1}{2}\left(\frac{z-p}{p}\right)^2+O(z-p)^3
\]
The first residue follows. 
By the binomial theorem,
\[
(\log z-\log p)^{-2}
=\frac{p^2}{(z-p)^2} + \frac{p}{z-p} + O(1).
\]
Let 
\[
f(z)=\frac{1-z^2}{4z^2}= \frac1{4z^2} - \frac14.
\]
Then
\begin{align*}
\Res_p\left(\frac{1-z^2}{4z^2}(\log z-\log p)^{-2}\right)
&=
\Res_p\left(\frac{f(z)p^2}{(z-p)^2}\right) + \Res_p\left( \frac{f(z)p}{(z-p)} \right)
\\
&=
f'(p)p^2 + f(p)p  \qquad(\text{by the Taylor expansion for $f$ at $p$})
\\
&=
-\frac1{2p^3}\, p^2 + \frac{1-p^2}{4p}
\\
&=
-\frac{1+p^2}{4p}.
\end{align*}
\end{proof}

%

\nocite{hoffman-wei}
\newcommand{\hide}[1]{}


\begin{bibdiv}

\begin{biblist}

\bib{allard}{article}{
  author={Allard, William K.},
  title={On the first variation of a varifold},
  journal={Ann. of Math. (2)},
  volume={95},
  date={1972},
  pages={417--491},
  issn={0003-486X},
  review={\MR {0307015},
  Zbl 0252.49028.}}
  \hide{(46 \#6136)}


\bib{allard-boundary}{article}{
   author={Allard, William K.},
   title={On the first variation of a varifold: boundary behavior},
   journal={Ann. of Math. (2)},
   volume={101},
   date={1975},
   pages={418--446},
   issn={0003-486X},
   review={\MR{0397520 (53 \#1379)}},
}

 
\bib{anderson}{article}{
   author={Anderson, Michael T.},
   title={Curvature estimates for minimal surfaces in $3$-manifolds},
   journal={Ann. Sci. \'Ecole Norm. Sup. (4)},
   volume={18},
   date={1985},
   number={1},
   pages={89--105},
   issn={0012-9593},
   review={\MR{803196 (87e:53098)}},
}

\bib{axler}{book}{
   author={Axler, Sheldon},
   author={Bourdon, Paul},
   author={Ramey, Wade},
   title={Harmonic function theory},
   series={Graduate Texts in Mathematics},
   volume={137},
   edition={2},
   publisher={Springer-Verlag, New York},
   date={2001},
   pages={xii+259},
   isbn={0-387-95218-7},
   review={\MR{1805196 (2001j:31001)}},
   doi={10.1007/978-1-4757-8137-3},
}
		


\bib{bernstein-breiner-conformal}{article}{
    AUTHOR = {Bernstein, Jacob and Breiner, Christine},
     TITLE = {Conformal structure of minimal surfaces with finite topology},
   JOURNAL = {Comment. Math. Helv.},
  FJOURNAL = {Commentarii Mathematici Helvetici. A Journal of the Swiss
              Mathematical Society},
    VOLUME = {86},
      YEAR = {2011},
    NUMBER = {2},
     PAGES = {353--381},
      ISSN = {0010-2571},
   MRCLASS = {53A10 (49Q05)},
  MRNUMBER = {2775132 (2012c:53009)},
MRREVIEWER = {Fei-Tsen Liang},
       DOI = {10.4171/CMH/226},
       URL = {http://dx.doi.org/10.4171/CMH/226},
}

\bib{bobenko}{article}{
   author={Bobenko, Alexander I.},
   title={Helicoids with handles and Baker-Akhiezer spinors},
   journal={Math. Z.},
   volume={229},
   date={1998},
   number={1},
   pages={9--29},
   issn={0025-5874},
   review={\MR{1649381 (2001d:53009)}},
   doi={10.1007/PL00004652},
}

\bib{choi-schoen}{article}{
   author={Choi, Hyeong In},
   author={Schoen, Richard},
   title={The space of minimal embeddings of a surface into a
   three-dimensional manifold of positive Ricci curvature},
   journal={Invent. Math.},
   volume={81},
   date={1985},
   number={3},
   pages={387--394},
   issn={0020-9910},
   review={\MR{807063 (87a:58040)}},
   doi={10.1007/BF01388577},
}


\bib{conway}{book}{
   author={Conway, John B.},
   title={Functions of one complex variable},
   series={Graduate Texts in Mathematics},
   volume={11},
   edition={2},
   publisher={Springer-Verlag, New York-Berlin},
   date={1978},
   pages={xiii+317},
   isbn={0-387-90328-3},
   review={\MR{503901 (80c:30003)}},
}
		
  
\bib{fischer-colbrie-schoen}{article}{
   author={Fischer-Colbrie, Doris},
   author={Schoen, Richard},
   title={The structure of complete stable minimal surfaces in $3$-manifolds
   of nonnegative scalar curvature},
   journal={Comm. Pure Appl. Math.},
   volume={33},
   date={1980},
   number={2},
   pages={199--211},
   issn={0010-3640},
   review={\MR{562550 (81i:53044)}},
   doi={10.1002/cpa.3160330206},
}

\bib{gilbarg-trudinger}{book}{
   author={Gilbarg, David},
   author={Trudinger, Neil S.},
   title={Elliptic partial differential equations of second order},
   series={Classics in Mathematics},
   note={Reprint of the 1998 edition},
   publisher={Springer-Verlag, Berlin},
   date={2001},
   pages={xiv+517},
   isbn={3-540-41160-7},
   review={\MR{1814364}},
}

\bib{gulliver-removability}{article}{
   author={Gulliver, Robert},
   title={Removability of singular points on surfaces of bounded mean
   curvature},
   journal={J. Differential Geometry},
   volume={11},
   date={1976},
   number={3},
   pages={345--350},
   issn={0022-040X},
   review={\MR{0431045 (55 \#4047)}},
}

\bib{hildebrandt}{article}{
   author={Hildebrandt, Stefan},
   title={Boundary behavior of minimal surfaces},
   journal={Arch. Rational Mech. Anal.},
   volume={35},
   date={1969},
   pages={47--82},
   issn={0003-9527},
   review={\MR{0248650 (40 \#1901)}},
}


\bib{hoffman-karcher-wei}{incollection}{
AUTHOR={Hoffman, David},
AUTHOR = {Karcher, Hermann},
AUTHOR={ Wei, Fu Sheng},
     TITLE = {The genus one helicoid and the minimal surfaces that led to
              its discovery},
 BOOKTITLE = {Global analysis in modern mathematics ({O}rono, {ME}, 1991;
              {W}altham, {MA}, 1992)},
     PAGES = {119--170},
 PUBLISHER = {Publish or Perish},
   ADDRESS = {Houston, TX},
      YEAR = {1993},
   MRCLASS = {53A10 (30F30)},
  MRNUMBER = {1278754 (95k:53011)},
MRREVIEWER = {Udo Hertrich-Jeromin},
}


\bib{hoffman-wei}{article}{
    AUTHOR = {Hoffman, David},
    AUTHOR = {Wei, Fusheng},
     TITLE = {Deforming the singly periodic genus-one helicoid},
   JOURNAL = {Experiment. Math.},
  FJOURNAL = {Experimental Mathematics},
    VOLUME = {11},
      YEAR = {2002},
    NUMBER = {2},
     PAGES = {207--218},
      ISSN = {1058-6458},
   MRCLASS = {53A10 (53C42)},
  MRNUMBER = {1959264 (2004d:53012)},
MRREVIEWER = {Rafael L{\'o}pez},
       URL = {http://projecteuclid.org/getRecord?id=euclid.em/1062621216},
}



\bib{hoffman-white-genus-one}{article}{
   author={Hoffman, David},
   author={White, Brian},
   title={Genus-one helicoids from a variational point of view},
   journal={Comment. Math. Helv.},
   volume={83},
   date={2008},
   number={4},
   pages={767--813},
   issn={0010-2571},
   review={\MR{2442963 (2010b:53013)}},
}


\bib{hoffman-white-geometry}{article}{
   author={Hoffman, David},
   author={White, Brian},
   title={The geometry of genus-one helicoids},
   journal={Comment. Math. Helv.},
   volume={84},
   date={2009},
   number={3},
   pages={547--569},
   issn={0010-2571},
   review={\MR{2507253 (2010f:53013)}},
   doi={10.4171/CMH/172},
}




\bib{hoffman-white-number}{article}{
   author={Hoffman, David},
   author={White, Brian},
   title={On the number of minimal surfaces with a given boundary},
   language={English, with English and French summaries},
   note={G\'eom\'etrie diff\'erentielle, physique math\'ematique,
   math\'ematiques et soci\'et\'e. II},
   journal={Ast\'erisque},
   number={322},
   date={2008},
   pages={207--224},
   issn={0303-1179},
   isbn={978-285629-259-4},
   review={\MR{2521657 (2010h:53007)}},
}

\bib{hoffman-white-axial}{article}{
   author={Hoffman, David},
   author={White, Brian},
   title={Axial minimal surfaces in $S^2\times\mathbf{R}$ are helicoidal},
   journal={J. Differential Geom.},
   volume={87},
   date={2011},
   number={3},
   pages={515--523},
   issn={0022-040X},
   review={\MR{2819547}},
}
		


\bib{lopez-ros}{article}{
   author={L{\'o}pez, Francisco J.},
   author={Ros, Antonio},
   title={On embedded complete minimal surfaces of genus zero},
   journal={J. Differential Geom.},
   volume={33},
   date={1991},
   number={1},
   pages={293--300},
   issn={0022-040X},
   review={\MR{1085145 (91k:53019)}},
}

\bib{meeks-perez-end}{article}{
   author={Meeks, William H., III},
   author={P{\'e}rez, Joaqu{\'{\i}}n},
 title={Embedded minimal surfaces of finite topology},
 date={2015},
 note={Preprint at arXiv:1506.07793 [math.DG]},
}

\bib{MeeksRosenbergUniqueness}{article}{
   author={Meeks, William H., III},
   author={Rosenberg, Harold},
   title={The uniqueness of the helicoid},
   journal={Ann. of Math. (2)},
   volume={161},
   date={2005},
   number={2},
   pages={727--758},
   issn={0003-486X},
   review={\MR{2153399 (2006f:53012)}},
   doi={10.4007/annals.2005.161.727},
}

\bib{MeeksRosenbergTheory}{article}{
  author={Meeks, William H.},
  author={Rosenberg, Harold},
  title={The theory of minimal surfaces in ${\mathbb {M}}\times \mathbb {R}$},
  journal={Comment. Math. Helv.},
  volume={80},
  date={2005},
  number={4},
  pages={811--858},
  issn={0010-2571},
  review={\MR {2182702},
  Zbl 1085.53049.}
} \hide{ (2006h:53007)}



\bib{rosenberg2002}{article}{
  author={Rosenberg, Harold},
  title={Minimal surfaces in ${\mathbb {M}}\sp 2\times \mathbb {R}$},
  journal={Illinois J. Math.},
  volume={46},
  date={2002},
  number={4},
  pages={1177--1195},
  issn={0019-2082},
  review={\MR {1988257},
  Zbl 1036.53008.}
} \hide{ (2004d:53015)}

\bib{schoen-uniqueness}{article}{
   author={Schoen, Richard M.},
   title={Uniqueness, symmetry, and embeddedness of minimal surfaces},
   journal={J. Differential Geom.},
   volume={18},
   date={1983},
   number={4},
   pages={791--809 (1984)},
   issn={0022-040X},
   review={\MR{730928 (85f:53011)}},
}

\bib{schmies}{thesis}{
  author={Schmies, Markus},
  title={Computational methods for Riemann surfaces and helicoids with handles},
  type={Ph.D. Thesis, Technische Universit\"at Berlin},
  date={2005},
}

\bib{smale-bridge}{article}{
   author={Smale, Nathan},
   title={A bridge principle for minimal and constant mean curvature
   submanifolds of ${\bf R}^N$},
   journal={Invent. Math.},
   volume={90},
   date={1987},
   number={3},
   pages={505--549},
   issn={0020-9910},
   review={\MR{914848 (88i:53101)}},
   doi={10.1007/BF01389177},
}

\bib{traizet-balancing}{article}{
   author={Traizet, Martin},
   title={A balancing condition for weak limits of families of minimal
   surfaces},
   journal={Comment. Math. Helv.},
   volume={79},
   date={2004},
   number={4},
   pages={798--825},
   issn={0010-2571},
   review={\MR{2099123 (2005g:53017)}},
   doi={10.1007/s00014-004-0805-1},
}

\bib{traizet-convex}{article}{
   author={Traizet, Martin},
   title={On minimal surfaces bounded by two convex curves in parallel
   planes},
   journal={Comment. Math. Helv.},
   volume={85},
   date={2010},
   number={1},
   pages={39--71},
   issn={0010-2571},
   review={\MR{2563680 (2010k:53012)}},
   doi={10.4171/CMH/187},
}

\bib{weber-hoffman-wolf}{article}{
    AUTHOR = {Weber, Matthias},
    AUTHOR= {Hoffman, David },
    AUTHOR={Wolf, Michael},
     TITLE = {An embedded genus-one helicoid},
   JOURNAL = {Ann. of Math. (2)},
  FJOURNAL = {Annals of Mathematics. Second Series},
    VOLUME = {169},
      YEAR = {2009},
    NUMBER = {2},
     PAGES = {347--448},
      ISSN = {0003-486X},
     CODEN = {ANMAAH},
   MRCLASS = {53A10 (49Q05)},
  MRNUMBER = {2480608 (2010d:53011)},
MRREVIEWER = {Antonio Alarc{\'o}n},
       DOI = {10.4007/annals.2009.169.347},
       URL = {http://dx.doi.org/10.4007/annals.2009.169.347},
}



\bib{white-curvature-estimates}{article}{
   author={White, Brian},
   title={Curvature estimates and compactness theorems in $3$-manifolds for
   surfaces that are stationary for parametric elliptic functionals},
   journal={Invent. Math.},
   volume={88},
   date={1987},
   number={2},
   pages={243--256},
   issn={0020-9910},
   review={\MR{880951 (88g:58037)}},
   doi={10.1007/BF01388908},
}



\bib{white-stable-bridge}{article}{
   author={White, Brian},
   title={The bridge principle for stable minimal surfaces},
   journal={Calc. Var. Partial Differential Equations},
   volume={2},
   date={1994},
   number={4},
   pages={405--425},
   issn={0944-2669},
   review={\MR{1383916 (97d:49043)}},
   doi={10.1007/BF01192091},
}

	
\bib{white-unstable-bridge}{article}{
   author={White, Brian},
   title={The bridge principle for unstable and for singular minimal
   surfaces},
   journal={Comm. Anal. Geom.},
   volume={2},
   date={1994},
   number={4},
   pages={513--532},
   issn={1019-8385},
   review={\MR{1336893 (96k:49063)}},
}

\bib{white-isoperimetric}{article}{
   author={White, Brian},
   title={Which ambient spaces admit isoperimetric inequalities for
   submanifolds?},
   journal={J. Differential Geom.},
   volume={83},
   date={2009},
   number={1},
   pages={213--228},
   issn={0022-040X},
   review={\MR{2545035}},
}

\bib{white-controlling-area}{article}{
   author={White, Brian},
   title={Controlling area blow-up in minimal or bounded mean curvature
   varieties},
   journal={J. Differential Geom.},
   volume={102},
   date={2016},
   number={3},
   pages={501--535},
   issn={0022-040X},
   review={\MR{3466806}},
}

\bib{white-embedded}{article}{
 author={White, Brian},
 title={On the compactness theorem for embedded minimal surfaces in $3$-manifolds with locally bounded
                          area and genus},
 date={2017},
 journal={Comm. Anal. Geom.},
 note={To appear. Preprint  at arXiv:1503.02190 [math.DG]},
}

\bib{white-bumpy}{article}{
 author={White, Brian},
 title={On the bumpy metrics theorem for minimal submanifolds},
 journal={American J. Math.}
 date={2017},
 note={To appear. Preprint at arXiv:1503.01803 [math.DG]},
}


\end{biblist}

\end{bibdiv}
\end{document}